\documentclass{article}[12pt]

\usepackage{makeidx}

\usepackage[english]{babel}
\usepackage[latin1]{inputenc}
\usepackage[T1]{fontenc}
\usepackage{amsmath, amsthm}
\usepackage{amscd}
\usepackage{amssymb}
\usepackage{amssymb}
\usepackage{latexsym}
\usepackage{amsmath}
\usepackage{amsfonts}
\usepackage{amsthm}

\usepackage{url}

\usepackage{color}

\usepackage{enumerate}
\usepackage{graphicx}
\usepackage{float}

\usepackage{lipsum}
\usepackage{tikz}
\usetikzlibrary{matrix,arrows,decorations.pathmorphing}
\usepackage[english]{babel}
\usepackage[latin1]{inputenc}
\usepackage[T1]{fontenc}
\usepackage{amsmath}
\usepackage{amssymb}
\usepackage{amsthm}
\usepackage{enumerate}

\newtheorem{thm}{Theorem}[section]
\newtheorem{theorem}[thm]{Theorem}
\newtheorem{claim}[thm]{Claim}

\newtheorem{corollary}[thm]{Corollary}
\newtheorem{lemma}[thm]{Lemma}
\newtheorem{proposition}[thm]{Proposition}
\newtheorem{definition}[thm]{Definition}
\newtheorem{remark}[thm]{Remark}

\theoremstyle{definition}
\newtheorem{example}[thm]{Example}
\newtheorem{examples}[thm]{Examples}

\makeatletter
\newcommand{\subalign}[1]{%
  \vcenter{%
    \Let@ \restore@math@cr \default@tag
    \baselineskip\fontdimen10 \scriptfont\tw@
    \advance\baselineskip\fontdimen12 \scriptfont\tw@
    \lineskip\thr@@\fontdimen8 \scriptfont\thr@@
    \lineskiplimit\lineskip
    \ialign{\hfil$\m@th\scriptstyle##$&$\m@th\scriptstyle{}##$\crcr
      #1\crcr
    }%
  }
}
\makeatother

\begin{document}

\centerline{\Large \bf Central elements in $\textbf{U}(gl(n))$,}

\bigskip

\centerline{\Large \bf shifted symmetric functions, }

\bigskip

\centerline{\Large \bf  and the superalgebraic Capelli's method}

\bigskip

\centerline{\Large \bf  of virtual variables}

\bigskip

\centerline{A. Brini and A. Teolis}
\centerline{\it $^\flat$ Dipartimento di Matematica, Universit\`{a} di
Bologna }
 \centerline{\it Piazza di Porta S. Donato, 5. 40126 Bologna. Italy.}
\centerline{\footnotesize e-mail: andrea.brini@unibo.it}
\medskip

\begin{abstract}
We propose a new method for a unified study of some of the main features of the theory of the center $\boldsymbol{\zeta}(n)$
of the enveloping algebra U(gl(n)) and of the algebra $\Lambda^*(n)$ of shifted symmetric polynomials, that allows the whole theory to be developed,
in a transparent and concise way, from the representation-theoretic point of view, that is entirely in the center of U(gl(n)).
Our methodological innovation is the systematic use of the superalgebraic method of virtual variables for gl(n), which is, in turn,
an extension of Capelli's method of ``variabili ausiliarie''.

The passage $n \rightarrow \infty$ for the
algebras $\boldsymbol{\zeta}(n)$ and $\Lambda^*(n)$
is here obtained both as   direct and inverse limit in the category of filtered algebras.

The present approach leads to    proofs  that are almost direct consequences of the definitions and
constructions: they often reduce to a few lines computation.

\end{abstract}


\tableofcontents

\section{Introduction}

The study of  the center
$\boldsymbol{\zeta}(n)$ of the enveloping algebra $\mathbf{U}(gl(n))$ of the
general linear Lie algebra $gl(n, \mathbb{C})$, and  the study of the algebra
$\Lambda^*(n)$ of shifted symmetric polynomials have noble and rather independent origins and motivations.

The theme of central elements  in $\mathbf{U}(gl(n))$ is a standard one in the general theory of Lie algebras, see e.g. \cite{DIX-BR}.
It is an old and actual one, since it
is an  offspring of the celebrated Capelli identity (\cite{Cap1-BR}, \cite{Cap4-BR}, \cite{Howe-BR}, \cite{HU-BR},
\cite{Procesi-BR}, \cite{Umeda-BR}, \cite{Weyl-BR}),
relates to its modern generalizations and applications (\cite{ABP-BR}, \cite{KostantSahi1-BR},
\cite{KostantSahi2-BR}, \cite{MolevNazarov-BR}, \cite{Nazarov-BR}, \cite{Okounkov-BR}, \cite{Okounkov1-BR},
\cite{Sahi3-BR}, \cite{UmedaCent-BR})
as well as to the theory of {\it{Yangians}} (see, e.g.  \cite{Molev1-BR}, \cite{Molev-BR},  \cite{Nazarov2-BR}).

The algebra $\Lambda^*(n)$ of shifted symmetric polynomials is a remarkable deformation of the  algebra
 $\Lambda(n)$ of symmetric polynomials and its study fits into the mainstream of generalizations of the classical theory
(e.g., {\it{factorial symmetric functions}}, \cite{BL1-BR}, \cite{BL2-BR}, \cite{CL-BR}, \cite{GG-BR}, \cite{GH-BR}, \cite{M2-BR}, \cite{M1-BR}).

Since the algebras $\boldsymbol{\zeta}(n)$ and $\Lambda^*(n)$
are related by the Harish-Chandra isomorphism $\chi_n$ (see, e.g. \cite{OkOlsh-BR}),
their investigation can be essentially regarded as a  single topic,
and this fact gave rise to a fruitful interplay between representation-theoretic
 methods (e.g., eigenvalues on irreducible representations)
and combinatorial techniques (e.g., generating functions).

In this work, we propose a new method for a unified study of some of the main features of the theory of the center
$\boldsymbol{\zeta}(n)$ of $\mathbf{U}(gl(n))$ and of the algebra $\Lambda^*(n)$ of shifted symmetric polynomials
that allows the whole theory to be developed, in a transparent and concise way, from the representation-theoretic point of view, that is entirely
in the center $\boldsymbol{\zeta}(n)$.

Our methodological innovation is the systematic use of the {\it{superalgebraic method of virtual variables}} for $gl(n)$
(see, e.g. \cite{BriUMI-BR}),
 which is, in turn, an extension
of Capelli's method of  {\textit{ variabili ausilarie}}. Capelli introduced the method of
{\textit{ variabili ausilarie}} in order
to manage symmetrizer operators in
terms of polarization operators, to  simplify the study of some skew-symmetrizer operators (namely, the famous central Capelli operator) and
developed this idea in a systematic way in his beautiful treatise \cite{Cap4-BR}. Capelli's idea was well suited to
treat symmetrization, but it did not work in the same efficient way while dealing with skew-symmetrization.

One had to wait the
 introduction of the notion of {\textit{superalgebras}} (see,e.g. \cite{Scheu-BR}, \cite{FSS-BR}, \cite{KAC1-BR}, \cite{VARAD-BR})
 to have the right conceptual framework to treat symmetry and skew-symmetry
 in one and the same way.
To the best of our knowledge, the first mathematician who intuited the connection between Capelli's idea and superalgebras was
Koszul in $1981$  \cite{Koszul-BR}; Koszul
 proved that the classical determinantal Capelli operator can be rewritten - in a much simpler way - by adding to the symbols to be dealt with
an extra auxiliary symbol that obeys to different commutation relations.

The {\textit{superalgebraic/supersymmetric method of virtual variables}} was developed in its full extent and generality (for the   general linear
Lie superalgebras $gl(m|n)$ - in the notation
of \cite{KAC1-BR}) in the series of notes \cite{Brini1-BR}, \cite{Brini2-BR}, \cite{Brini3-BR}, \cite{Brini4-BR}, \cite{BHT-BR} by  the present authors and
R.Q. Huang.

The superalgebraic method of virtual variables allows remarkable classes of elements in $\boldsymbol{\zeta}(n)$ to
be written as simple {\it{sums}} of {\it{monomials}} and their actions on irreducible representations to be given simple combinatorial descriptions;
 moreover, this method
throws a  bridge between the theory of  $\boldsymbol{\zeta}(n)$ and  the {\it{(super)straightening techniques}}
 \cite{rota-BR}, \cite{Brini3-BR}, \cite{Brini4-BR},
\cite{Brini5-BR}, \cite{Bri-BR}
(or, in the classical context, {\it{standard monomial theory}}, see e.g. \cite{Procesi-BR}).

We consider five classes of central elements, which arise in a natural way in the context of the virtual method
when dealing with symmetry and skew-symmetry in $\mathbf{U}(gl(n))$:
\begin{itemize}
\renewcommand{\labelitemi}{\normalfont -}
\item
The elements $\mathbf{H}_k(n), \quad k = 1, 2, \ldots, n$; a noteworthy fact is that these elements turn out to be  a supple form of
the classical determinantal Capelli elements  of $1893$
\footnote{These elements are different from the Capelli elements in the sense, for example, of Howe and Umeda \cite{HU-BR}, even if the two families
are closely related (see, e.g. \cite{BriUMI-BR}).}
(\cite{Cap2-BR}, \cite{Cap3-BR},
\cite{Brini5-BR}).
\item
The elements $\mathbf{I}_k(n), \quad k \in \mathbf{Z}^+$; a noteworthy fact is that these elements turn out to be  a supple form of
the central elements studied by Nazarov \cite{Nazarov-BR} and Umeda \cite{UmedaHirai-BR}.
\item
The elements $\mathbf{K}_{\lambda}(n)$ and $\mathbf{J}_{\lambda}(n)$, that are generalizations  to Young shapes $\lambda$ of the
$\mathbf{H}_k(n)$ and of the $\mathbf{I}_k(n)$, respectively. The actions of these elements on irreducible representations
have remarkable triangularity properties (Theorem \ref{K Triangularity} and Proposition \ref{J Triangularity}).
\item
The  elements $\mathbf{S}_{\lambda}(n)$, that provide  a common generalization of both $\mathbf{H}_k(n)$ and $\mathbf{I}_k(n)$.
By Theorem \ref{Schur action} and the Sahi/Okounkov Characterization Theorem (see Subsection \ref{Characterization  Theorems} below),
the elements $\mathbf{S}_{\lambda}(n)$ turn out to be  a supple form of the {\it{Schur elements}} discovered  by Sahi \cite{Sahi1-BR}
in the context of shifted symmetric  polynomials, by Okounkov \cite{Okounkov-BR} , \cite{Okounkov1-BR} as elements of $\boldsymbol{\zeta}(n)$,
and   extensively
investigated by Okounkov and Olshanski \cite{OkOlsh-BR} both from the point of view of central elements and of shifted symmetric functions
in infinitely many variables.
\end{itemize}

Just to mention a few remarkable features
of the method, we first note that the centrality of the elements we consider follows from extremely simple arguments on
their virtual presentations (see, e.g. Proposition \ref{central virtual determinantal} and Proposition \ref{central virtual permanental}).
The duality/reciprocity in $\boldsymbol{\zeta}(n)$ (Theorem \ref{finite duality}) immediately follows from a new (and rather surprising)
combinatorial description of the eigenvalues
of the Capelli elements $\mathbf{H}_k(n)$ on irreducible representations (Proposition \ref{horizontal strip}) that is
{\it{dual}} (in the sense of shapes/partitions) to the combinatorial description of the eigenvalues
of the Nazarov/Umeda elements $\mathbf{I}_k(n)$ (Theorem \ref{vertical strip}.1).
By the {\it{ Bitableaux correspondence and Koszul map Theorems}}
(\cite{Brini4-BR}, Thms. 1 and 2, see also \cite{BriUMI-BR}, \cite{Koszul-BR}),
the elements $\mathbf{H}_k(n), \mathbf{I}_k(n), \mathbf{K}_{\lambda}(n), \mathbf{J}_{\lambda}(n)$ expand into
``{\it{column bitableaux}}'' in $\mathbf{U}(gl(n))$ in a way that is in all respect similar to the ordinary Laplace
expansions of determinants and permanents of matrices with entries in a commutative algebra. This fact, in turn,
leads to further combinatorial descriptions of the eigenvalues of the central  elements
$\mathbf{H}_k(n), \mathbf{I}_k(n), \mathbf{K}_{\lambda}(n), \mathbf{J}_{\lambda}(n)$
on irreducible representations that make apparent the role
of permutations (for the  sake of brevity, we fully work out only the case of the $\mathbf{H}_k(n)$'s, in Subsection \ref{subsubsection vertical}).
Our representation-theoretic versions
(Theorem \ref{Vanishing theorem} and Theorem \ref{Schur action}) of the Okounkov
Vanishing Theorem  and of the Sahi/Okounkov Characterization Theorem
follow at once from some standard elementary facts
on Schur supermodules (Proposition \ref{Vanishing Lemma}) in combination with Regonati's hook
lemma \cite{Regonati-BR} (Proposition \ref{hook lemma}).

The passage to the infinite dimensional case $n \rightarrow \infty$ for the
algebras $\boldsymbol{\zeta}(n)$ is rather subtle;  the ``naive''
$\infty-$dimensional analogue of the algebras $\mathbf{U}(gl(n))$,
that is the direct limit algebra $\underrightarrow{lim} \ \mathbf{U}(gl(n))$ with respect to the {\it{``inclusion''  monomorphisms}},
has trivial center.

The $\infty-$dimensional analogue $\boldsymbol{\zeta}$ of the algebras $\boldsymbol{\zeta}(n)$
is here obtained as the  direct limit algebra $\underrightarrow{lim} \ \boldsymbol{\zeta}(n)$ (in the category of filtered algebras)
with respect to a family of monomorphisms $\mathbf{i}_{n+1,n} : \boldsymbol{\zeta}(n) \hookrightarrow \boldsymbol{\zeta}(n+1)$,
that we call the {\it{Capelli monomorphisms}} (Section \ref{Capelli monomorphisms}).
This construction is a Lie algebra analogue of the construction of the ring of symmetric functions $\Lambda$
as the  direct limit algebra $\underrightarrow{lim} \ \Lambda(n)$ with respect to the monomorphims that map
the elementary symmetric polynomial $e_k(n)$ in $n$ variables to the elementary symmetric polynomial $e_k(m)$ in $m$ variables,
$n < m$, $k = 1, 2, \ldots, n.$ (see, e.g. \cite{Fink-BR}, or \url{<https://en.wikipedia.org/wiki/Ring_of_symmetric_functions>}).

The direct limit construction of the algebra $\boldsymbol{\zeta}$ implies that,
if $P = \underrightarrow{lim} \ P(n)$ and $Q = \underrightarrow{lim} \ Q(n)$ are elements of $\boldsymbol{\zeta}$
of ``minimum filtration degree'' $m$, then  $ P = Q$ if and only if   $P(n) = Q(n)$ in $\boldsymbol{\zeta}(n)$,
 for some  $n \geq m;$ therefore, linear and algebraic relations among elements of $\boldsymbol{\zeta}$ are determined
 by the relations among their ``{\it{germs}}'' in $\boldsymbol{\zeta}(n)$, for $n$ sufficiently large.

An intrinsic/invariant presentation of the Capelli monomorphisms is obtained, in Section \ref{The Olshanski decomposition/projection},
via a family of projections $\boldsymbol{\mu}_{n,n+1} : \boldsymbol{\zeta}(n+1) \twoheadrightarrow \boldsymbol{\zeta}(n)$, the
{\it{Olshanski projections}} \cite{Olsh1-BR}, \cite{Olsh3-BR} (see also \cite{Molev1-BR}). The Olshanski projections
$\boldsymbol{\mu}_{n,n+1}$ are {\it{left}} inverses of the Capelli monomorphisms $\mathbf{i}_{n+1,n}$, and they become {\it{two-sided}}
inverses when restricted to the filtration elements $\boldsymbol{\zeta}(n+1)^{(m)}$ and $\boldsymbol{\zeta}(n)^{(m)}$,
for $n$ sufficiently large (Proposition \ref{inverso filtrato} and  Proposition  \ref{pi=mu}).

Amazingly, the Olshanski projection
$\boldsymbol{\mu}_{n,n+1}$ acts on virtual presentations just by truncating the sums of monomials
(see Proposition \ref{Olshanski Capelli} and Proposition \ref{Olshanski altri}).
Therefore, the direct limits $\mathbf{H}_k$ of
 the   $\mathbf{H}_k(n)$, $\mathbf{I}_k$ of the
$\mathbf{I}_k(n)$, $\mathbf{K}_{\lambda}$ of the
$\mathbf{K}_{\lambda}(n)$, $\mathbf{J}_{\lambda}$ of the
$\mathbf{J}_{\lambda}(n)$, and $\mathbf{S}_{\lambda}$ of the
  $\mathbf{S}_{\lambda}(n)$ can be consistently  written as  {\it{formal series}} of virtual {\it{monomials}}
  (Theorem \ref{germs}, Definition \ref{series} and Proposition \ref{formal series}).

The  interplay between Capelli monomorphisms and Olshanski projections shows the algebra $\boldsymbol{\zeta}$
admits a double presentation, both as a direct limit and as an inverse limit. Being the algebra $\boldsymbol{\zeta}$
isomorphic to the algebra $\Lambda^*$ of {\it{shifted symmetric functions}} (Theorem \ref{isomorfismo HC infinito}),
the Olshanski projections are the
natural counterpart, in the context of the centers $\boldsymbol{\zeta}(n)$, of the Okounkov-Olshanski {\it{stability principle}}
for the algebras $\Lambda^*(n)$ of shifted symmetric polynomials \cite{OkOlsh-BR}, the isomorphism
$\chi : \boldsymbol{\zeta} \rightarrow \Lambda^*$ is indeed the ``limit'' of the Harish-Chandra isomorphisms $\chi_n$ and it admits
a transparent representation-theoretic interpretation (Proposition \ref{The map universal HC}).

We tried to make the exposition self-contained.

\noindent The arguments are based on very few prerequisites:  the fact that the classical Capelli
elements of $1893$ (Section \ref{Capelli 1893}) provide a free
system of algebra generators of $\boldsymbol{\zeta}(n)$
and  a small bunch of combinatorial lemmas (Proposition \ref{Vanishing Lemma} and Proposition \ref{hook lemma}).
The use of the virtual variables  turns  all the proofs  into almost direct consequences of the definitions and
constructions.

\section{Synopsis}

The paper is organized as follows.

In Chapter \ref{virtual method}, we summarize the leading ideas of the present approach, that is the main facts and constructions
of the {\it{superalgebraic virtual variables method for $gl(n)$}} \cite{BriUMI-BR}.

As already mentioned,  we develop the whole theory from the representation-theoretic point of view, that is
our main concern are the eigenvalues of central elements of $\mathbf{U}(gl(n))$ on $gl(n)-$irreducible representations.
In order to do this, we embed any ``covariant'' $gl(n)-$irreducible representation (Schur module $Schur_\lambda(n)$) into an irreducible
supermodule $Schur_\lambda(m_0|m_1+n)$ on a suitable general linear superalgebra $gl(m_0|m_1+n)$. Besides the algebras
$\mathbf{U}(gl(n))  \subset \mathbf{U}(gl(m_0|m_1+n))$, in Subsection \ref{Virt}  we single out a third algebra,
the {\it{virtual algebra}} $Virt(m_0|m_1+n)$,
$$
\mathbf{U}(gl(n))  \subset Virt(m_0|m_1+n)   \subset \mathbf{U}(gl(m_0|m_1+n)),
$$
that has the remarkable properties:
\begin{itemize}
\renewcommand{\labelitemi}{\normalfont -}
\item
there is a ``canonical'' epimorphism $\mathfrak{p} : Virt(m_0|m_1+n) \twoheadrightarrow \mathbf{U}(gl(n))$, that we
call the {\it{Capelli epimorphism}};
\item
the Schur $\mathbf{U}(gl(n))-$irreducible module $Schur_\lambda(n)$ is an invariant subspace of the
$\mathbf{U}(gl(m_0|m_1+n))-$irreducible supermodule
$Schur_\lambda(m_0|m_1+n)$, with respect to the action of
the subalgebra $Virt(m_0|m_1+n) \subset \mathbf{U}(gl(m_0|m_1+n))$;
\item
the action of any element of $Virt(m_0|m_1+n)$ on $Schur_\lambda(n)$ is the same of the action of its
image in $\mathbf{U}(gl(n))$ with respect to the Capelli epimorphism (Theorem \ref{Capelli epimorphism}).
\end{itemize}
Therefore, instead of studying the action of an element in $\mathbf{U}(gl(n))$ one can study the action of a preimage of it in
$Virt(m_0|m_1+n)$ (called {\it{virtual presentation}}). The advantage of virtual presentations is that they
are frequently of monomial form,   admit quite transparent interpretations and are  much  easier to be dealt with
(see, e.g. \cite{Brini1-BR}, \cite{Brini2-BR}, \cite{BRT-BR}, \cite{Bri-BR}, \cite{BriUMI-BR}), so we even take them
as a definition of an element in $\mathbf{U}(gl(n)).$

In order to make the virtual variables method  effective, we  exhibit a class of nontrivial elements that belong
to $Virt(m_0+m_1,n)$, that is \textit{balanced monomials} (Subsection \ref{balanced monomial}).
A quite relevant subclass of balanced monomials arises in connection with pairs of Young tableaux
(Section \ref{Tableaux}).

When specialized to the center $\boldsymbol{\zeta}(n)$ of $\mathbf{U}(gl(n))$, this method reveals further features and benefits:
\begin{itemize}
\renewcommand{\labelitemi}{\normalfont -}
\item
the subalgebra $Virt(m_0|m_1+n)$ is an invariant subspace of $\mathbf{U}(gl(m_0|m_1+n))$  with respect to the adjoint action of
$gl(n)$;
\item
the Capelli epimorphism is an equivariant map with respect to the adjoint action of
$gl(n)$  (Proposition \ref{rappresentazione aggiunta-BR});
\item
the Capelli epimorphism image of an element of $Virt(m_0|m_1+n)$ that is an {\it{invariant}} with respect to the adjoint action of
$gl(n)$ belongs to the center $\boldsymbol{\zeta}(n)$ of $\mathbf{U}(gl(n))$ (Corollary \ref{invarianti virtuali}).
\end{itemize}
Therefore, in Chapter \ref{The center} we will systematically define classes of central elements through their virtual presentations;
in this way, the centrality
is immediately apparent from the definition.

In Section \ref{standard}, we recall a family of algebra generators of the supersymmetric algebra ${\mathbb C}[M_{m_0|m_1+n,d}]$,
called {\it{biproducts}}, which are supersymmetric generalizations of ``formal determinants'', as well as a a family of linear generators
associated to a pairs of Young tableaux, called {\it{bitableaux}}, for short.
Bitableaux are signed products of biproducts \cite{rota-BR} (for a virtual presentation of both, see \cite{Brini1-BR}, \cite{Bri-BR}),
both share a good behaviour with respect to the superpolarization action of $\mathbf{U}(gl(m_0|m_1+n))$  on ${\mathbb C}[M_{m_0|m_1+n,d}]$
(Proposition \ref{polarization biproduct} and Proposition \ref{action on tableaux}). By the  superstraightening algorithm, the set of
(super)standard bitableaux is a basis of ${\mathbb C}[M_{m_0|m_1+n,d}]$ \cite{rota-BR}, \cite{Brini5-BR}, \cite{Bri-BR}
(Theorem \ref{theorem: standard expansion of bitableaux} and Corollary \ref{theorem: standard basis}).

In Section \ref{Schur}, we review  the main facts about Schur supermodules as submodules of ${\mathbb C}[M_{m_0|m_1+n,d}]$,
and of  Schur modules as  $\mathbf{U}(gl(n))-$submodules of Schur supermodules.
The highest weight vectors of Schur modules turns out to be {\it{Deruyts bitableaux}} (Subsection \ref{Schur modules}).

In Subsection \ref{Lemmas}, we recall some results about the action of the Lie superalgebra $\mathbf{U}(gl(m_0|m_1+n))$
on the the subspace $Schur_\lambda(n)$
({\it{Vanishing Lemmas}}, Proposition \ref{Vanishing Lemma} and {\it{Regonati's Hook Lemma}}, Proposition \ref{hook lemma}), which
are direct consequences of the use of virtual variables in combination with the  superstraightening algorithm.
These results will play a crucial role in Chapter \ref{The center}; as a matter of fact, the
``triangularity/orthogonality'' results for the action of the central elements $\mathbf{K}_{\lambda}(n)$,
$\mathbf{J}_{\lambda}(n)$, $\mathbf{S}_{\lambda}(n)$ on highest weight vectors almost immediately follow from them.

In  Chapter \ref{The center}, we introduce five classes
 $\mathbf{H}_k(n)$, $\mathbf{I}_k(n)$, $\mathbf{K}_{\lambda}(n)$, $\mathbf{J}_{\lambda}(n)$, and $\mathbf{S}_{\lambda}(n)$
of elements of $\mathbf{U}(gl(n))$ as images of {\it{simple sums of
balanced monomials}} in $Virt(m_0|m_1+n)$. Due to their virtual presentations, these elements are almost immediately recognized as
belonging to the center $\boldsymbol{\zeta}(n)$. The  $\mathbf{H}_k(n)$'s and the $\mathbf{I}_k(n)$'s turn out to be the Capelli determinantal
elements and their permanental analogues, respectively.

In  Subsection \ref{subsubsection vertical}  we consider the  Capelli  generators $\mathbf{H}_k(n)$ from
the point of view of the {\it{ Koszul isomorphism}}
and show that they expand into ``{\it{column bitableaux}}'' in the same way as the determinants
of matrices with commutative entries expand into ordinary monomials (Proposition \ref{column expansion});
this result implies a new combinatorial description
of their eigenvalues on irreducible representations (Proposition \ref{third eigenvalue}).

The {\it{shaped}} generalizations $\mathbf{K}_{\lambda}(n)$ of  $\mathbf{H}_k(n)$ and
 $\mathbf{J}_{\lambda}(n)$ of  $\mathbf{I}_k(n)$ provide two new linear bases of
$\boldsymbol{\zeta}(n)$ (Corollary \ref{K basis} and Corollary \ref{J basis})
that satisfy remarkable triangularity properties when acting on highest weight vectors
(Theorem \ref{K Triangularity} and Theorem \ref{J Triangularity}). From an intuitive point of view,
the $\mathbf{K}_{\lambda}(n)$'s and the $\mathbf{J}_{\lambda}(n)$'s are elements with ``internal'' row skew-symmetry and symmetry
respectively,  as should be clear from their virtual presentations.

In Section \ref{Schur elements}, we introduce a third basis whose elements $\mathbf{S}_{\lambda}(n)$ have both ``internal''
row skew-symmetry and column symmetry. Their action on highest weight vectors (Theorem \ref{Schur action}) satisfy the condition
of the Sahi-Okounkov Characterization Theorem and, therefore, the $\mathbf{S}_{\lambda}(n)$'s turn out to be
the {\it{Schur  elements}} described
by Okounkov as {\it{quantum immanants}} (for further details, see Subsection \ref{Characterization  Theorems}).
It is remarkable that the same elements can also be obtained by interchanging the symmetries, that is the $\mathbf{S}_{\lambda}(n)$
can be defined as having column skew-symmetry and row symmetry (Subsection \ref{S second}).

In Section \ref{central duality} we deal with {\it{duality}} in $\boldsymbol{\zeta}(n)$\footnote{Equivalent results
- in the sense of the Harish-Chandra isomorphism - were obtained by Okounkov and Olshanski \cite{OkOlsh-BR} in the context of the algebra
$\Lambda^*(n)$ of shifted symmetric polynomials, by using generating functions techniques.},
which is a Lie algebra analogue of the
classical involution of the algebra $\Lambda(n)$ of symmetric polynomials.
The algebra $\boldsymbol{\zeta}(n)$ has an
involution $\mathcal{W}_n$ with notable properties:
\begin{itemize}
\renewcommand{\labelitemi}{\normalfont -}
\item If $\lambda_1, \ \widetilde{\lambda}_1 \leq n$, the image $\mathcal{W}_n\big(\mathbf{S}_{\lambda}(n)\big) $ equals
$\mathbf{S}_{\widetilde{\lambda}}(n)$ (Corollary \ref{Schur duality}).
\item
The eigenvalue of an element $\boldsymbol{\varrho} \in \boldsymbol{\zeta}(n)$ on a highest weight vector of weight $\mu$
equals the eigenvalue of its image $\mathcal{W}_n\big(\boldsymbol{\varrho})$
on a highest weight vector of weight $\widetilde{\mu}$ (Theorem \ref{finite duality}).
\end{itemize}

In the first Section of Chapter \ref{zeta limit}, we construct the $\infty-$dimensional analogue $\boldsymbol{\zeta}$ of
$\boldsymbol{\zeta}(n)$ as the direct limit (in the category of filtered algebras) with respect to the family of Capelli monomorphism
$$\mathbf{i}_{n+1,n} : \boldsymbol{\zeta}(n) \hookrightarrow \zeta(n+1), \qquad
\mathbf{i}_{n+1,n}\big(\mathbf{H}_k(n)\big) = \mathbf{H}_k(n+1), \quad k =1, 2, \ldots, n.$$
The left inverses of the $\mathbf{i}_{n+1,n}$:
$$
\boldsymbol{\mu}_{n,n+1} = \boldsymbol{\pi}_{n,n+1} : \boldsymbol{\zeta}(n+1) \twoheadrightarrow \boldsymbol{\zeta}(n),
$$
$$
\boldsymbol{\pi}_{n,n+1}(\mathbf{H}_k(n+1)) = \mathbf{H}_k(n), \ k = 1, 2, \ldots, n,
\quad
\boldsymbol{\pi}_{n,n+1}(\mathbf{H}_{n+1}(n+1)) = 0,
$$
 become {\it{two-sided}}
inverses when restricted to the filtration elements $\boldsymbol{\zeta}(n)^{(m)}$ and $\boldsymbol{\zeta}(n+1)^{(m)}$,
for $n$ sufficiently large.
The main point is that the projections $\boldsymbol{\pi}_{n,n+1}$ admit an intrinsic description in terms of the
{\it{Olshanski decomposition}} (Section \ref{The Olshanski decomposition/projection}).
Due to their virtual presentations, the Olshanski decompositions of the elements  $\mathbf{H}_k(n)$, $\mathbf{I}_k(n)$,
$\mathbf{K}_{\lambda}(n)$, $\mathbf{J}_{\lambda}(n)$, and $\mathbf{S}_{\lambda}(n)$ are amazingly simple
(Proposition \ref{Olshanski Capelli} and Proposition \ref{Olshanski altri}) and imply that
\begin{align*}
\mathbf{i}_{n+1,n}\big(\mathbf{I}_k(n)\big) &= \mathbf{I}_k(n+1), \qquad
\mathbf{i}_{n+1,n}\big(\mathbf{K}_\lambda(n)\big) = \mathbf{K}_\lambda(n+1),
\\
\mathbf{i}_{n+1,n}\big(\mathbf{J}_\lambda(n)\big) &= \mathbf{J}_\lambda(n+1), \qquad
\mathbf{i}_{n+1,n}\big(\mathbf{S}_\lambda(n)\big) = \mathbf{S}_\lambda(n+1),
\end{align*}
for $n$ ``sufficiently large'' (Section \ref{main}).

In Chapter \ref{Lambda(n)}, we recall the Harish-Chandra isomorphism
$\chi_n : \boldsymbol{\zeta}(n) \longrightarrow \Lambda^*(n)$, $\Lambda^*(n)$ the algebra of shifted symmetric polynomials,
and translate the results of Chapter \ref{The center} from $\boldsymbol{\zeta}(n)$  to  $\Lambda^*(n)$.

In Chapter \ref{Lambda}, we introduce the isomorphism
$\chi : \boldsymbol{\zeta} \longrightarrow \Lambda^*$, $\Lambda^*$ the algebra of shifted symmetric functions in infinitely many variables,
and translate the results of Chapter \ref{zeta limit} from $\boldsymbol{\zeta}$  to  $\Lambda^*$.

\section{The method of
 virtual supersymmetric variables for $\mathbf{U}(gl(n))$}\label{virtual method}

\subsection{The Lie algebra $gl(n)$ as a subalgebra of the general linear Lie superalgebra $gl(m_0|m_1+n)$}

 Given a vector space $V_n$ of dimension $n$, we will regard it as a subspace of a $\mathbb{Z}_2-$graded vector space
 $W = W_0 \oplus W_1$, where
$$
W_0 = V_{m_0}, \qquad W_1 = V_{m_1} \oplus V_n.
$$
The {\textit{ auxiliary}} vector spaces
$V_{m_0}$ and $V_{m_1}$ (informally, we assume that
$dim(V_{m_0})=m_0$ and $dim(V_{m_1})=m_1$ are ``sufficiently large'') are called
the spaces of {\textit{even virtual
vectors}} and of  {\textit{odd virtual vectors}}, respectively, and $V_n$ is called the space of {\textit{odd proper vectors}}.

 The inclusion $V_n \subset W$ induces a natural embedding of the general linear Lie algebra $gl(n)$ into the
general linear Lie {\it{superalgebra}} $gl(m_0|m_1+n)$ (see, e.g. \cite{KAC1-BR}, \cite{Scheu-BR}, \cite{FSS-BR}).

Let
$
A_0 = \{ \alpha_1, \ldots, \alpha_{m_0} \},$  $A_1 = \{ \beta_1, \ldots, \beta_{m_1} \},$
$L = \{ x_1, \ldots, x_n \}$
denote distinguished  bases of $V_{m_0}$, $V_{m_1}$ and $V_n$, respectively; therefore $|\alpha_s| = 0 \in \mathbb{Z}_2,$
and $|\beta_t| = |x_i|   = 1 \in \mathbb{Z}_2.$

Let
$$
\{ e_{a, b}; a, b \in A_0 \cup A_1 \cup L \}, \qquad |e_{a, b}| =
|a|+|b| \in \mathbb{Z}_2
$$
be the standard $\mathbb{Z}_2-$homogeneous basis of $gl(m_0|m_1+n)$ provided by the
elementary matrices.

The supercommutator of $gl(m_0|m_1+n)$ has the following explicit form:
$$
\left[ e_{a, b}, e_{c, d} \right] = \delta_{bc} \ e_{a, d} - (-1)^{(|a|+|b|)(|c|+|d|)} \delta_{ad}  \ e_{c, b},
$$
$a, b, c, d \in A_0 \cup A_1 \cup L.$

\subsection{The commutative algebra ${\mathbb C}[M_{n,d}]$ as a subalgebra of the supersymmetric algebra ${\mathbb C}[M_{m_0|m_1+n,d}]$}

The \textit{algebra of algebraic forms in $n$ vector variables
of dimension $d$} is the polynomial algebra in $n \times d$ variables:
$$
{\mathbb C}[M_{n,d}] =    {\mathbb C}[x_{ij}]_{i=1,\ldots,n; j=1,\ldots,d}
$$
where $M_{n,d}$ represents the  matrix
with $n$ rows and $d$ columns with ``generic" entries $x_{ij}$:
$$
M_{n,d} = \left[ x_{ij} \right]_{i=1,\ldots,n; j=1,\ldots,d}=
 \left[
 \begin{array}{ccc}
 x_{11} & \ldots & x_{1d} \\
 x_{21} & \ldots & x_{2d} \\
 \vdots  &        & \vdots \\
 x_{n1} & \ldots & x_{nd} \\
 \end{array}
 \right].
$$

In the following, we will write $(x_i|j)$ in place of $x_{ij},$ and regard the (commutative) algebra ${\mathbb C}[M_{n,d}]$
as a subalgebra of the \textit{supersymmetric algebra}
$$
{\mathbb C}[M_{m_0|m_1+n,d}] = {\mathbb C}\big[ (\alpha_s|j), (\beta_t|j), (x_i|j) \big]
$$
 generated by the ($\mathbb{Z}_2$-graded) variables $(\alpha_s|j), (\beta_t|j), (x_i|j)$,
$j = 1, 2, \ldots, d$,
 where
 $$
 |(\alpha_s|j)| = 1 \in \mathbb{Z}_2 \ \ and \ \  |(\beta_t|j)| = |(x_i|j)| = 0 \in \mathbb{Z}_2,
 $$
subject to the commutation relations:

$$
(a|h)(b|k) = (-1)^{|(a|h)||(b|k)|} \ (b|k)(a|h),
$$
for $a, b \in  \{ \alpha_1, \ldots, \alpha_{m_0} \} \cup \{ \beta_1, \ldots, \beta_{m_1} \} \cup \{x_1, x_2, \ldots , x_n\}.$

We have:
\begin{align*}
{\mathbb C}[M_{m_0|m_1+n,d}] & \cong
 \Lambda \left[ (\alpha_s|j) \right]
\otimes      {\mathrm{Sym}} \left[ (\beta_t|j), (x_h|j) \right] &  \\
& \cong \Lambda \left[ W_0 \otimes P_d \right]
\otimes      {\mathrm{Sym}} \left[ (W_1 \oplus V_n) \otimes P_d \right],
\end{align*}
where $P_d = (P_d)_1$ denotes the trivially (odd) $\mathbb{Z}_2-$graded vector space with distinguished basis $\{j; \ j = 1, 2, \ldots, d \}.$

The algebra  ${\mathbb C}[M_{m_0|m_1+n,d}]$ is a supersymmetric $\mathbb{Z}_2-$graded algebra (superalgebra), whose
$\mathbb{Z}_2-$graduation is  inherited by the natural one in the exterior algebra.

\subsection{Left superderivations and left superpolarizations}

A {\it{left superderivation}} ($\mathbb{Z}_2-$homogeneous of degree $|D|$) (see, e.g. \cite{FSS-BR}, \cite{Scheu-BR}, \cite{KAC1-BR}) on
${\mathbb C}[M_{m_0|m_1+n,d}]$ is an element $D \in End_\mathbb{C}[\mathbb{C}[M_{m_0|m_1+n,d}]]$ ($\mathbb{Z}_2-$homogeneous of degree $|D|$)
that satisfies "Leibniz rule"
$$
D(P \cdot Q) = D(P) \cdot Q + (-1)^{|D||P|} P \cdot D(Q),
$$
for every $\mathbb{Z}_2-$homogeneous of degree $|D|$ element $P \in \mathbb{C}[M_{m_0|m_1+n,d}].$

\medskip

Given two symbols $a, b \in A_0 \cup A_1 \cup L$, the {\textit{superpolarization}} $D_{a,b}$ of $b$ to $a$
is the unique {\it{left}} superderivation of ${\mathbb C}[M_{m_0|m_1+n,d}]$ of parity $|D_{a,b}| = |a| + |b| \in \mathbb{Z}_2$ such that
\begin{equation}
D_{a,b} \left( (c|j) \right) = \delta_{bc} \ (a|j), \ c \in A_0 \cup A_1 \cup L, \ j = 1, \ldots, d.
\end{equation}

Informally, we say that the operator $D_{a,b}$ {\it{annihilates}} the symbol $b$ and {\it{creates}} the symbol $a$.

\subsection{The superalgebra ${\mathbb C}[M_{m_0|m_1+n,d}]$ as a $\mathbf{U}(gl(m_0|m_1+n))-$module}

The map
$$
e_{a,b} \rightarrow D_{a,b}, \qquad a, b \in A_0 \cup A_1 \cup L.
$$
(that send the elementary
matrices to the corresponding superpolarizations) is a Lie superalgebra morphism from $gl(m_0|m_1+n)$ to $End_\mathbb{C}[\mathbb{C}[M_{m_0|m_1+n,d}]]$
and, hence, it uniquely defines the
morphism (i.e. a representation):
$$
\varrho : \mathbf{U}(gl(m_0|m_1+n)) \rightarrow End_\mathbb{C}[\mathbb{C}[M_{m_0|m_1+n,d}]].
$$

In the following, we always regard the superalgebra $\mathbb{C}[M_{m_0|m_1+n,d}]$ as a $\mathbf{U}(gl(m_0|m_1+n))-$supermodule,
with respect to the representation $\varrho$.
We recall that $\mathbb{C}[M_{m_0|m_1+n,d}]$ is  a {\it{semisimple}} $\mathbf{U}(gl(m_0|m_1+n))-$supermodule,
whose irreducible (simple) submodules are - up to isomorphism - {\it{Schur supermodules}} (\cite{Brini1-BR}, \cite{Brini2-BR}, \cite{Bri-BR},
\cite{CW-BR}).
Clearly, the subalgebra $\mathbb{C}[M_{n,d}]$ is a $gl(n)-$module of  $\mathbb{C}[M_{m_0|m_1+n,d}]$.

\subsection{The virtual algebra $Virt(m_0+m_1,n)$ and the virtual
presentations of elements in $\mathbf{U}(gl(n))$}

\subsubsection{The virtual algebra $Virt(m_0+m_1,n)$ as a subalgebra of $\mathbf{U}(gl(m_0|m_1+n))$ and
the Capelli devirtualization epimorphism $\mathfrak{p} : Virt(m_0+m_1,n) \twoheadrightarrow \mathbf{U}(gl(n))$}\label{Virt}

\begin{definition}{\textbf{(Irregular expressions)}}\label{Irregular expressions-BR} We say that a product
$$
e_{a_mb_m} \cdots e_{a_1b_1} \in \mathbf{U}(gl(m_0|m_1+n))
$$
is an {\textit{irregular expression}} whenever
  there exists a right subsequence
$$e_{a_i,b_i} \cdots e_{a_2,b_2} e_{a_1,b_1},$$
$i \leq m$ and a
virtual symbol $\gamma \in A_0 \cup A_1$ such that
\begin{equation}
 \# \{j;  b_j = \gamma, j \leq i \}  >  \# \{j;  a_j = \gamma, j < i \}.
\end{equation}
\end{definition}

The meaning of an irregular expression in terms of the action of  $\mathbf{U}(gl(m_0|m_1+n))$ on
the algebra $\mathbb{C}[M_{m_0|m_1+n,d}]$ is that there exists a
virtual symbol $\gamma$ and a right subsequence in which the symbol $\gamma$ is annihilated more times than
it was already created.

\begin{definition}{\textbf{(The ideal $\mathbf{Irr}$)}} The  left ideal $\mathbf{Irr}$  of $\mathbf{U}(gl(m_0|m_1+n))$ is
the {\textit{left ideal}} generated by the set of
irregular expressions.
\end{definition}

\begin{remark}
The action
of any element of $\mathbf{Irr}$ on the subalgebra $\mathbb C[M_{n,d}] \subset \mathbb{C}[M_{m_0|m_1+n,d}]$ - via the representation $\varrho$ -
is identically zero.
\end{remark}

\begin{proposition}(\cite{Brini3-BR}, \cite{BriUMI-BR})
The sum ${\mathbf{U}}(gl(0|n)) + \mathbf{Irr}$ is a direct sum of vector subspaces of $\mathbf{U}(gl(m_0|m_1+n)).$
\end{proposition}

\begin{definition} {\textbf{(The virtual algebra $Virt(m_0+m_1,n)$)}}
The vector space $Virt(m_0+m_1,n) = \mathbf{U}(gl(0|n)) \oplus \mathbf{Irr}$
is a subalgebra of ${\mathbf{U}}(gl(m_0|m_1+n)),$  called the {\textit{virtual subalgebra}} \cite{Brini3-BR},
and $\mathbf{Irr}$ is a two sided ideal of $Virt(m_0+m_1,n).$
\end{definition}

\begin{definition}{\textbf{(The Capelli devirtualization epimorphism)}}
The {\textit{Capelli epimorphism}} is the projection
$$
\mathfrak{p} : Virt(m_0+m_1,n) = \mathbf{U}(gl(0|n)) \oplus \mathbf{Irr} \twoheadrightarrow \mathbf{U}(gl(0|n)) = \mathbf{U}(gl(n))
$$
with $Ker(\mathfrak{p}) = \mathbf{Irr}.$
\end{definition}

\begin{remark}
Any element in $Virt(m_0+m_1,n)$ defines an element in
$\mathbf{U}(gl(n))$, and  is called a \textit{virtual
presentation} of it.  The map $\mathfrak{p}$ being a surjection, any element
$\mathbf{p} \in \mathbf{U}(gl(n))$ admits several virtual
presentations. We even take virtual presentations
as  definitions of special elements in $\mathbf{U}(gl(n)),$ and this method will turn out to be quite effective.
\end{remark}

The next results will play a crucial role in Section \ref{The center}.
\begin{proposition}\label{rappresentazione aggiunta-BR}
For every $e_{x_i, x_j} \in gl(n) \subset gl(m_0|m_1+n)$,   let $ad(e_{x_i, x_j})$ denote its adjoint action
on  $Virt(m_0+m_1,n)$; the ideal $\mathbf{Irr}$ is $ad(e_{x_i, x_j})-$invariant. Then
\begin{equation}
\mathfrak{p} \left( ad(e_{x_i, x_j})( \mathbf{m} ) \right) =  ad(e_{x_i, x_j}) \left( \mathfrak{p} ( \mathbf{m} ) \right) ,
\qquad  \mathbf{m} \in Virt(m_0+m_1,n).
\end{equation}
\end{proposition}

\begin{corollary}\label{invarianti virtuali}
The Capelli epimorphism image of an element of $Virt(m_0|m_1+n)$ that is an invariant for the adjoint action of
$gl(n)$ is in the center $\boldsymbol{\zeta}(n)$ of $\mathbf{U}(gl(n))$.
\end{corollary}

\begin{example}{\textbf{(A virtual presentation of the Capelli determinant)}}\label{Virtual Capelli}
Let $\alpha \in A_0$. The element
\begin{equation}\label{M}
M = e_{x_n,\alpha} \cdots e_{x_2,\alpha} e_{x_1,\alpha} \cdot e_{\alpha, x_1} e_{\alpha, x_2}
 \cdots e_{\alpha, x_n}
\end{equation}
belongs to the subalgebra $Virt(m_0|m_1+n)$ ( section \ref{balanced monomial} below).
The image of the element $M$ under the Capelli devirtualization epimorphism $\mathfrak{p}$
equals the {\it{column determinant}}\footnote{The symbol $\mathbf{cdet}$
denotes the column determinat of a matrix $A = [a_{ij}]$ with noncommutative entries:
$\mathbf{cdet} (A) = \sum_{\sigma} \ (-1)^{|\sigma|}  \ a_{\sigma(1), 1}a_{\sigma(2), 2} \cdots a_{\sigma(n), n}.$}
 $$
\mathbf{H}_n(n) = \mathbf{cdet}\left(
 \begin{array}{cccc}
 e_{x_1,x_1}+(n-1) & e_{x_1,x_2} & \ldots  & e_{x_1,x_n} \\
 e_{x_2,x_1} & e_{x_2,x_2}+(n-2) & \ldots  & e_{x_2,x_n}\\
 \vdots  &    \vdots                            & \vdots &  \\
e_{x_n,x_1} & e_{x_n,x_2} & \ldots & e_{x_n,x_n}\\
 \end{array}
 \right).
 $$
This result is a special case of the result that we called the  ``Laplace expansion for Capelli rows'' ( \cite{BRT-BR} Theorem $2$,
\cite{Bri-BR} Theorem $6.3$). A sketchy proof of it
can also be found in \cite{Koszul-BR}.

Recall that
$$ad(e_{x_i, x_j})\left( e_{x_h, \alpha} \right) = \delta_{j h}e_{x_i, \alpha},$$
$$ad(e_{x_i, x_j})\left( e_{\alpha, x_k} \right) = - \delta_{k i}e_{\alpha, x_j},$$
for every virtual symbol $\alpha$, and that $ad(e_{x_i, x_j})$ acts as a derivation, for every $i, j = 1,2, \ldots, n.$

The  monomial $M$ is annihilated by
$ad(e_{x_i, x_j}), \ i \neq j,$ by skew-symmetry. Furthermore, $ad(e_{x_i, x_i})\left( M \right) = M - M = 0, \ i = 1, 2, \ldots, n;$
hence, $M$ is an invariant for the adjoint action of $gl(n)$.

Since
$\mathfrak{p} \left( M \right) = \mathbf{H}_n(n),$ the element $\mathbf{H}_n(n)$ is central in $\mathbf{U}(gl(n)),$
by Corollary \ref{invarianti virtuali}.
\end{example}

\subsubsection{The action of  $Virt(m_0+m_1,n)$ on the subalgebra $\mathbb{C}[M_{n,d}]$}

From the representation-theoretic point of view, the core of the {\textit{method of virtual variables}}
lies in the following result.

\begin{theorem}\label{Capelli epimorphism}
The action of $Virt(m_0+m_1,n)$ leaves invariant the
subalgebra of algebraic forms $\mathbb{C}[M_{n,d}] \subseteq
\mathbb{C}[M_{m_0|m_1+n,d}],$ and, therefore, the action of $Virt(m_0+m_1,n)$ on $\mathbb{C}[M_{n,d}]$
is well defined.
Furthermore, for every $\mathbf{v} \in Virt(m_0+m_1,n)$, its action on $\mathbb{C}[M_{n,d}]$ equals
the action of $\mathfrak{p}(\mathbf{v}).$
\end{theorem}

Therefore, instead of studying the action of an element in $\mathbf{U}(gl(n))$, one can study the action of a virtual presentation of it in
$Virt(m_0|m_1+n)$. The advantage of virtual presentations is that they
are frequently of monomial form,   admit quite transparent interpretations and are  much  easier to be dealt with
(see, e.g. \cite{Brini1-BR}, \cite{Brini2-BR}, \cite{BRT-BR}, \cite{Bri-BR}, \cite{BriUMI-BR}).

A prototypical instance of this method is provided by the celebrated Capelli identity \cite{Cap1-BR}, \cite{Weyl-BR},
\cite{Howe-BR}, \cite{HU-BR}, \cite{Umeda-BR}. From Example \ref{Virtual Capelli}, it follows that the action of the Capelli
determinant $\mathbf{H}_n(n)$
on a form $f \in \mathbb{C}[M_{n,d}]$ is the same as the action of its monomial virtual presentation (\ref{M}), and this leads to
a few lines proof of the identity \cite{BRT-BR}, \cite{BriUMI-BR}.

\subsubsection{Balanced monomials as elements of the virtual algebra $Virt(m_0+m_1,n)$}\label{balanced monomial}

In order to make the virtual variables method  effective, we need to exhibit a class of nontrivial elements that belong
to $Virt(m_0+m_1,n)$. A quite relevant class of such elements is provided
by \textit{balanced monomials}.

\begin{definition}{\textbf{(Balanced monomials)}}\label{balanced monomials-BR}
In the algebra ${\mathbf{U}}(gl(m_0|m_1+n)),$ consider an
element of the forms:
\begin{itemize}\label{defbalanced monomials-BR}
\item $e_{x_{i_1},\gamma_{p_1}} \cdots e_{x_{i_k},\gamma_{p_k}} \cdot
e_{\gamma_{p_1},x_{j_1}} \cdots e_{\gamma_{p_k},x_{j_k}},$
\item
$e_{x_{i_1},\theta_{q_1}} \cdots e_{x_{i_k},\theta_{q_k}} \cdot
e_{\theta_{q_1},\gamma_{p_1}} \cdots e_{\theta_{q_k},\gamma_{p_k}} \cdot
e_{\gamma_{p_1},x_{j_1}} \cdots e_{\gamma_{p_k},x_{j_k}},$
\item $\cdots\cdots \quad \cdots\cdots$
\end{itemize}
where
$x_{i_1}, \ldots, x_{i_k}, x_{j_1}, \ldots, x_{j_k} \in L,$
i.e., the $x_{i_1}, \ldots, x_{i_k}, x_{j_1}, \ldots, x_{j_k}$ are $k$
proper  symbols;
\end{definition}

In plain words, a balanced monomial is product of two or more factors  where the
rightmost one  \textit{annihilates}
the $k$ proper symbols $ x_{j_1}, \ldots, x_{j_k}$ and
\textit{creates} some virtual symbols;
 the leftmost one  \textit{annihilates} all the virtual symbols
and \textit{creates} the $k$ proper symbols $ x_{i_1}, \ldots, x_{i_k}$;
between these two factors, there might be further factors that annihilate
 and create  virtual symbols only.

The next result is the (superalgebraic) formalization of the argument developed by Capelli in
\cite{Cap4-BR}, CAPITOLO I, §X.Metodo delle variabili ausiliarie, page $55$ ff.

\begin{proposition}(\cite{Brini1-BR}, \cite{Brini2-BR}, \cite{BRT-BR}, \cite{Bri-BR}, \cite{BriUMI-BR})
Every balanced monomial belongs to $Virt(m_0+m_1,n)$. Hence
its image under the Capelli epimorphism $\mathfrak{p}$ belongs to $\mathbf{U}(gl(n)).$
\end{proposition}

In plain words, the action of a balanced monomial on the subalgebra  $\mathbb{C}[M_{n,d}]$ equals
the action of a suitable element of $\mathbf{U}(gl(n)).$

The following result lies deeper and  is a major tool in the proof of identities involving
monomial virtual presentation of elements of $\mathbf{U}(gl(n)).$ Since the adjoint representation acts by superderivation,
it may be regarded as a version of the {\textit{Laplace expansion}} for the images of balanced monomials.

\begin{proposition}{\textbf{(Monomial virtual presentation and adjoint actions)}}\label{Monomial virtual presentation and adjoint actions-BR}
In $\mathbf{U}(gl(n)),$ the element
$$
\mathfrak{p} \left[ e_{x_{i_1},\gamma_{p_1}} \cdots e_{x_{i_n},\gamma_{p_n}} \cdot
e_{\gamma_{p_1},x_{j_1}} \cdots e_{\gamma_{p_n},x_{j_n}} \right]
$$
equals
$$
\mathfrak{p} \Big[ \ ad(e_{x_{i_1},\gamma_{p_1}}) \cdots ad(e_{x_{i_n},\gamma_{p_n}}) \big( e_{\gamma_{p_1},x_{j_1}}e_{\gamma_{p_2},x_{j_2}} \cdots
e_{\gamma_{p_n},x_{j_n}} \big) \ \Big].
$$

\end{proposition}

\begin{example}\label{Capelli determinants-BR}
Let $\alpha \in A_0$. Then
$$
 [ x_{i_k} \cdots x_{i_2}  x_{i_1} | x_{i_1} x_{i_2} \cdots x_{i_k} ] =
 \mathfrak{p} [ e_{x_{i_k},\alpha} \cdots e_{x_{i_2},\alpha} e_{x_{i_1},\alpha} \cdot e_{\alpha, x_{i_1}} e_{\alpha, x_{i_2}}
 \cdots e_{\alpha, x_{i_k}} ] =
 $$
 $$
 = \mathfrak{p} \Big[ \ ad(e_{x_{i_k},\alpha}) \cdots ad(e_{x_{i_2},\alpha}) ad(e_{x_{i_1},\alpha})  \left(e_{\alpha, x_{i_1}} e_{\alpha, x_{i_2}}
 \cdots e_{\alpha, x_{i_k}}\right) \ \Big] =
 $$
 $$
= \mathbf{cdet}\left(
 \begin{array}{cccc}
 e_{x_{i_1},x_{i_1}}+(k-1) & e_{x_{i_1},x_{i_2}} & \ldots  & e_{x_{i_1},x_{i_k}} \\
 e_{x_{i_2},x_{i_1}} & e_{x_{i_2},x_{i_2}}+(k-2) & \ldots  & e_{x_{i_2},x_{i_k}}\\
 \vdots  &    \vdots                            & \vdots &  \\
e_{x_{i_k},x_{i_1}} & e_{x_{i_k},x_{i_2}} & \ldots & e_{x_{i_k},x_{i_k}}\\
 \end{array}
 \right) \in {\mathbf{U}}(gl(n)).
 $$
\end{example}

\begin{example}\label{column permanent}
Let $\alpha \in A_1$. The element
$$\mathbf{p} =  \mathfrak{p} \left[ e_{x_{3},\alpha} e_{x_{2},\alpha} e_{x_{1},\alpha} \cdot
e_{\alpha,x_{1}} e_{\alpha,x_{2}} e_{\alpha,x_{3}} \right] =
$$
$$
= \mathfrak{p} \left[ ad(e_{{x_3},\alpha}) ad(e_{{x_2},\alpha}) ad(e_{{x_1},\alpha}) \left( e_{\alpha,x_1} e_{\alpha,x_2} e_{\alpha,x_3} \right) \right]
$$

equals the {\textit{column permanent}}\footnote{The symbol $\mathbf{cper}$
denotes the column permanent of a matrix $A = [a_{ij}]$ with noncommutative entries:
$\mathbf{cper} (A) = \sum_{\sigma} \ a_{\sigma(1), 1}a_{\sigma(2), 2} \cdots a_{\sigma(n), n}.$}

$$
\mathbf{cper}
\left(
 \begin{array}{ccc}
 e_{{x_1},{x_1}} - 2 & e_{{x_1},{x_2}} &  e_{{x_1},{x_3}} \\
 e_{{x_2},{x_1}} & e_{{x_2},{x_2}} - 1 &  e_{{x_2},{x_3}}\\
 e_{{x_3},{x_1}} & e_{{x_3},{x_2}} &  e_{{x_3},{x_3}}\\
 \end{array}
 \right) \in {\mathbf{U}}(gl(3)).
$$
\end{example}

\begin{example}\label{third example}
Let $\alpha, \beta \in A_0$. Then
$$
\left[
\begin{array}{c|c}
x_1 & x_2 \\
x_2 & x_1
\end{array}
\right] = \mathfrak{p} \left( e_{x_1,\alpha}e_{x_2,\beta} \cdot e_{\alpha,x_2}e_{\beta,x_1} \right) =
$$
$$
= - e_{x_1,x_2}e_{x_2,x_1} + e_{x_1,x_1} \in {\mathbf{U}}(gl(2)).
$$
\end{example}

\subsection{Bitableaux  in ${\mathbf{U}}(gl(m_0|m_1+n))$ and balanced monomials}\label{Tableaux}

\subsubsection{Young tableaux on the alphabet $A_0 \cup A_1 \cup L$}

Consider the alphabets
$
A_0,$  $A_1,$ $L,$
(that is a distinguished $\mathbb{Z}_2-$homogeneous bases of the $\mathbb{Z}_2-$graded vector space
$W = W_0 \oplus W_1$, $W_0 = V_{m_0}$, $W_1 = V_{m_1} \oplus V_n$),
where
$
A_0 = \{ \alpha_1, \ldots, \alpha_{m_0} \},$  $A_1 = \{ \beta_1, \ldots, \beta_{m_1} \},$
$L = \{ x_1, \ldots, x_n \}$
denote distinguished  bases of $V_{m_0}$, $V_{m_1}$ and $V_n.$

Given a partition $\lambda$, a {\it{Young tableau}} $X$ of  shape $\lambda = (\lambda_1 \geq \lambda_2 \geq \cdots \geq \lambda_p)$ on
the  alphabet $ A_0 \cup A_1 \cup L$ is an array
$$
X = \left(
\begin{array}{llllllllllllll}
z_{i_1}  \ldots    \ldots     \ldots    & z_{i_{\lambda_1}}     \\
z_{j_1}   \ldots  \ldots               z_{j_{\lambda_2}} \\
 \ldots  \ldots  & \\
z_{s_1} \ldots z_{s_{\lambda_p}}
\end{array}
\right),
$$
where the symbols $z'$s belong to the alphabet $A_0 \cup A_1 \cup L.$
The partition $\lambda$ is called the {\it{shape}} of the a Young tableau $X$, denoted by the symbol $sh(X).$

We also denote the Young tableau $X$ by the sequence of its {\it{row words}}, namely
$$
X = (\omega_1, \omega_2, \ldots, \omega_p),
$$
where
$$
\omega_k = z_{k_1}   \ldots                 z_{k_{\lambda_k}}, \quad k = 1, 2, \ldots, p.
$$

The {\it{row word}} of the Young tableau $X$ is the word
\begin{equation}\label{row word}
\omega(X) = \omega_1 \omega_2 \cdots \omega_p.
\end{equation}

The {\it{content}} of the Young tableau $X$ is the function
$$
c_X : A_0 \cup A_1 \cup L \rightarrow \mathbb{N},
$$
where
$c_X(z)$ equals the number of occurrences of $z$ in $X$, for $z \in A_0 \cup A_1 \cup L.$

Consider the linear order
$$
\alpha_1 < \ldots < \alpha_{m_0} < \beta_1 < \ldots < \beta_{m_1}
< x_1 < \ldots < x_n
$$
on the alphabet $A_0 \cup A_1 \cup L.$
The Young tableau $X$ is said to be
{\it{(super)standard}} whenever its row and column are nondecreasing words, its rows have no repetitions
of negative symbols in $A_1 \cup L$, its columns no repetitions
of positive symbols in $A_0.$ (\cite{Berele1-BR}, \cite{Berele2-BR}, \cite{rota-BR}, \cite{Brini1-BR}).

\subsubsection{Bitableaux monomials}\label{Bitableaux monomials}
Let $S$ and $T$ be two Young tableaux of same shape $\lambda = (\lambda_1 \geq \lambda_2 \geq \cdots \geq \lambda_p)$ on
the  alphabet $ A_0 \cup A_1 \cup L$:

\begin{equation}\label{bitableaux}
S = \left(
\begin{array}{llllllllllllll}
z_{i_1}  \ldots    \ldots     \ldots    & z_{i_{\lambda_1}}     \\
z_{j_1}   \ldots  \ldots               z_{j_{\lambda_2}} \\
 \ldots  \ldots  & \\
z_{s_1} \ldots z_{s_{\lambda_p}}
\end{array}
\right), \qquad
T = \left(
\begin{array}{llllllllllllll}
z_{h_1}  \ldots    \ldots     \ldots    & z_{h_{\lambda_1}}    \\
z_{k_1}   \ldots  \ldots               z_{k_{\lambda_2}} \\
 \ldots  \ldots  & \\
z_{t_1} \ldots z_{t_{\lambda_p}}
\end{array}
\right)
\end{equation}.

To the pair $(S,T)$, we associate the {\it{bitableau monomial}}:
$$
e_{S,T} =
e_{z_{i_1}, z_{h_1}}\cdots e_{z_{i_{\lambda_1}}, z_{h_{\lambda_1}}}
e_{z_{j_1}, z_{k_1}}\cdots e_{z_{j_{\lambda_2}}, z_{k_{\lambda_2}}}
 \cdots  \cdots
e_{z_{s_1}, z_{t_1}}\cdots e_{z_{s_{\lambda_p}}, z_{t_{\lambda_p}}}
$$
in ${\mathbf{U}}(gl(m_0|m_1+n)).$

\begin{example}\label{example monomial} Let $S$ and $T$ be tableaux of shape $\lambda = (3,2,2)$ on the alphabet $ A_0 \cup A_1 \cup L$:
\begin{equation}\label{bitableaux}
S = \left(
\begin{array}{lll}
z & x & y     \\
z & u \\
x & v
\end{array}
\right), \qquad
T = \left(
\begin{array}{lll}
z & s & w    \\
x & t \\
y & w
\end{array}
\right)
\end{equation}.
To the pair $(S,T)$, we associate the monomial:
$$
e_{S,T} =
e_{z,z} e_{x,s}
e_{y,w}e_{z,x}
e_{u,t}e_{x,y}e_{v,w}
$$
in ${\mathbf{U}}(gl(m_0|m_1+n)).$
\end{example}

\subsubsection{Deruyts and Coderuyts tableaux}

Let $\alpha_1, \ldots, \alpha_p \in A_0$, $\beta_1, \ldots, \beta_{\lambda_1} \in A_1$ and set

\begin{equation}\label{Deruyts and Coderuyts}
D_{\lambda}^* = \left(
\begin{array}{llllllllllllll}
\beta_1  \ldots    \ldots     \ldots    & \beta_{\lambda_1}     \\
\beta_1   \ldots  \ldots               \beta_{\lambda_2} \\
 \ldots  \ldots  & \\
\beta_1 \ldots \beta_{\lambda_p}
\end{array}
\right), \qquad
C_{\lambda}^* = \left(
\begin{array}{llllllllllllll}
\alpha_1  \ldots    \ldots     \ldots    & \alpha_1    \\
\alpha_2   \ldots  \ldots               \alpha_2 \\
 \ldots  \ldots  & \\
\alpha_p \ldots \alpha_p
\end{array}
\right)
\end{equation}.

The tableaux $D_{\lambda}^*$ and $C_{\lambda}^*$ are called the {\it{virtual Deruyts and Coderuyts tableaux}}
of shape $\lambda,$ respectively.

\subsubsection{Three special classes of elements in $Virt(m_0+m_1,n)$}

Given a pair $(S,T)$ of Young tableaux of the same shape $\lambda$ on the proper alphabet $L$, consider the elements
\begin{equation}\label{determinantal}
e_{S,C_{\lambda}^*} \ e_{C_{\lambda}^*,T} \in {\mathbf{U}}(gl(m_0|m_1+n)),
\end{equation}
\begin{equation}\label{permanental}
e_{S,D_{\lambda}^*} \ e_{D_{\lambda}^*,T} \in {\mathbf{U}}(gl(m_0|m_1+n)),
\end{equation}
\begin{equation}\label{YoungCapelli}
e_{S,D_{\lambda}^*} \ e_{D_{\lambda}^*,C_{\lambda}^*}  \   e_{C_{\lambda}^*,T} \in {\mathbf{U}}(gl(m_0|m_1+n)).
\end{equation}
Since elements (\ref{determinantal}), (\ref{permanental}), (\ref{YoungCapelli}) are balanced monomials in
${\mathbf{U}}(gl(m_0|m_1+n))$, then they belong to the subalgebra $Virt(m_0+m_1,n)$ (section \ref{balanced monomial}).
Hence, we can consider their images with respect to the Capelli epimorphism $\mathfrak{p}$ and set
\begin{equation}
\mathfrak{p} \Big( e_{S,C_{\lambda}^*} \ e_{C_{\lambda}^*,T}    \Big) =  SC_{\lambda}^* \ C_{\lambda}^*T  \in {\mathbf{U}}(gl(n)),
\end{equation}
\begin{equation}
\mathfrak{p} \Big( e_{S,D_{\lambda}^*} \ e_{D_{\lambda}^*,T} \Big) =  SD_{\lambda}^*  \ D_{\lambda}^*T    \in {\mathbf{U}}(gl(n)),
\end{equation}
\begin{equation}
\mathfrak{p} \Big( e_{S,D_{\lambda}^*} \ e_{D_{\lambda}^*,C_{\lambda}^*}  \   e_{C_{\lambda}^*,T} \Big)
= SD_{\lambda}^*  \  D_{\lambda}^*C_{\lambda}^*  \  C_{\lambda}^*T           \in {\mathbf{U}}(gl(n)).
\end{equation}

\begin{example} Let  $\lambda = (3,2,2)$, then
\begin{equation}
C_{\lambda}^* = \left(
\begin{array}{lll}
\alpha_1  & \alpha_1    & \alpha_1    \\
\alpha_2  & \alpha_2 \\
\alpha_3 & \alpha_3
\end{array}
\right), \quad \alpha_1, \alpha_2,    \alpha_3 \in A_0.
\end{equation}

Let $S$ and $T$ be the tableaux of shape $\lambda = (3,2,2)$ on the alphabet $ A_0 \cup A_1 \cup L$,
of Example \ref{example monomial}. Then
\begin{align*}
&SC_{\lambda}^* \ C_{\lambda}^*T =  \mathfrak{p}  \Big( e_{S,C_{\lambda}^*} \ e_{C_{\lambda}^*,T}    \Big) =
\mathfrak{p} \Big(  e_{{\fontsize{6} {6} \selectfont\substack{
x y x \ \alpha_1  \alpha_1  \alpha_1  \phantom{\beta_3}\\
z u \phantom{x}\ \alpha_2  \alpha_2 \phantom{\alpha_3} \phantom{\beta_3}\\
x v \phantom{x}\ \alpha_3  \alpha_3 \phantom{\alpha_3} \phantom{\beta_3}\\
}}}
e_{{\fontsize{6} {6} \selectfont\substack{
\alpha_1  \alpha_1  \alpha_1 \ z s w \phantom{\beta_3}\\
\alpha_2  \alpha_2 \phantom{\alpha_2} \ x t \phantom{w} \phantom{\beta_3}\\
\alpha_3  \alpha_3 \phantom{\alpha_3} \ y w \phantom{t} \phantom{\beta_3}\\
}}} \Big) =
\\
&= \mathfrak{p}  \Big(e_{z,\alpha_1} e_{x,\alpha_1}e_{y,\alpha_1}e_{z,\alpha_2}e_{u,\alpha_2}e_{x,\alpha_3}e_{v,\alpha_3}
e_{\alpha_1,z} e_{\alpha_1,s}e_{\alpha_1,w}e_{\alpha_2,x}e_{\alpha_2,t}e_{\alpha_3,y}e_{\alpha_3,w}
\Big).
\end{align*}

\end{example}

\begin{remark}

In the present notation, the element described in Example \ref{Capelli determinants-BR} is the element
$$
(-1)^{{k} \choose {2}} \  SC_{\lambda}^* \ C_{\lambda}^*S, \quad \lambda = (k), \quad S = \left( \begin{array}{ccc}
x_{i_1} & \ldots & x_{i_k}
\end{array}
\right).
$$
The element described in Example \ref{column permanent} is the element
$$
 SD_{\lambda}^* \ D_{\lambda}^*S, \quad \lambda = (1,1,1), \quad S = \left(
\begin{array}{rrr}
x_1  \\
x_2  \\
x_3
\end{array}
\right).
$$
The element described in Example \ref{third example} is the element
$$
 SC_{\lambda}^* \ C_{\lambda}^*T, \quad \lambda = (1,1), \quad S = \left(
\begin{array}{rr}
x_1  \\
x_2
\end{array}
\right), \quad
T = \left(
\begin{array}{rr}
x_2  \\
x_1
\end{array}
\right)
$$
\end{remark}

From Proposition \ref{rappresentazione aggiunta-BR}, it directly follows:

\begin{proposition}
For every $e_{x_i, x_j} \in gl(n) \subset gl(m_0|m_1+n)$,   let $ad(e_{x_i, x_j})$ denote its adjoint action
on  $Virt(m_0+m_1,n)$.  We have
\begin{equation}
ad(e_{x_i, x_j}) \left( SC_{\lambda}^* \ C_{\lambda}^*T \right) =
\mathfrak{p} \Big( ad(e_{x_i, x_j}) \left( e_{S,C_{\lambda}^*} \ e_{C_{\lambda}^*,T} \right)    \Big)
 \in {\mathbf{U}}(gl(n)),
\end{equation}
\begin{equation}
ad(e_{x_i, x_j}) \left( SD_{\lambda}^*  \ D_{\lambda}^*T \right) =
\mathfrak{p} \Big( ad(e_{x_i, x_j}) \left( e_{S,D_{\lambda}^*} \ e_{D_{\lambda}^*,T} \right)    \Big)
 \in {\mathbf{U}}(gl(n)),
\end{equation}
\begin{equation}
ad(e_{x_i, x_j}) \left( SD_{\lambda}^*  \  D_{\lambda}^*C_{\lambda}^*  \  C_{\lambda}^*T  \right) =
\mathfrak{p} \Big( ad(e_{x_i, x_j}) \left(  e_{S,D_{\lambda}^*} \ e_{D_{\lambda}^*,C_{\lambda}^*}  \   e_{C_{\lambda}^*,T} \right)    \Big)
 \in {\mathbf{U}}(gl(n)).
\end{equation}
\end{proposition}

\subsection{Bitableaux in ${\mathbb C}[M_{m_0|m_1+n,d}]$ and the standard monomial theory}\label{standard}

\subsubsection{Biproducts in ${\mathbb C}[M_{m_0|m_1+n,d}]$}

Embed the algebra
$$
{\mathbb C}[M_{m_0|m_1+n,d}] = {\mathbb C}[(\alpha_s|j), (\beta_t|j), (x_i|j)]
$$
into the (supersymmetric) algebra ${\mathbb C}[(\alpha_s|j), (\beta_t|j), (x_i|j), (\gamma|j)]$
generated by the ($\mathbb{Z}_2$-graded) variables $(\alpha_s|j), (\beta_t|j), (x_i|j), (\gamma|j)$,
$j = 1, 2, \ldots, d$,
 where
 $$
 |(\gamma|j)| = 1 \in \mathbb{Z}_2 \ \  for \ every \ j = 1, 2, \ldots, d,
 $$
and denote by $D_{z_i,\gamma}$ the superpolarization of $\gamma$ to  $z_i.$

\begin{definition}{\textbf{(Biproducts in}} ${\mathbb C}[M_{m_0|m_1+n,d}]${\textbf{)}}
Let $\omega = z_1z_2 \cdots z_p$ be a word on     $ A_0 \cup A_1 \cup L$, and $\varpi = j_{t_1}j_{t_2} \cdots j_{t_q}$ a word
on the alphabet
$P = \{1, 2, \ldots, d \}$. The {\it{biproduct}}
$$
(\omega|\varpi) = (z_1z_2 \cdots z_p|j_{t_1}j_{t_2} \cdots j_{t_q})
$$
is the element
$$
D_{z_1,\gamma}D_{z_2,\gamma} \cdots D_{z_p,\gamma} \Big( (\gamma|j_{t_1})(\gamma|j_{t_2}) \cdots
(\gamma|j_{t_q}) \Big) \in {\mathbb C}[M_{m_0|m_1+n,d}]
$$
if $p = q$ and is set to be zero otherwise.
\end{definition}

\begin{claim}

The biproduct $(\omega|\varpi) = (z_1z_2 \cdots z_p|j_{t_1}j_{t_2} \cdots j_{t_q})$ is supersymmetric in the  $z$'s and
skew-symmetric in the  $j$'s.
In symbols
\begin{enumerate}

\item
$
(z_1 z_2 \cdots z_i z_{i+1} \cdots z_p|j_{t_1} j_{t_2} \cdots j_{t_q}) =
\\ \null \hfill (-1)^{|z_i| |z_{i+1}|}
(z_1 z_2 \cdots z_{i+1} z_i \cdots z_p|j_{t_1} j_{t_2} \cdots j_{t_q})
$

\item
$
(z_1z_2 \cdots z_iz_{i+1} \cdots z_p|j_{t_1}j_{t_2} \cdots j_{t_i}j_{t_{i+1}} \cdots j_{t_q}) =
\\ \null \hfill
- (z_1z_2 \cdots z_iz_{i+1} \cdots z_p|j_{t_1} \cdots j_{t_{i+1}}j_{t_i} \cdots j_{t_q}).
$
\end{enumerate}

\end{claim}

\begin{proposition}{\textbf{(Laplace expansions)}}\label{Laplace expansions}
We have
\begin{enumerate}

\item
$
(\omega_1\omega_2|\varpi) = \Sigma_{(\varpi)} \ (-1)^{|\varpi_{(1)}| |\omega_2|} \ (\omega_1|\varpi_{(1)})(\omega_2|\varpi_{(2)}).
$

\item
$
(\omega|\varpi_1\varpi_2) = \Sigma_{(\omega)} \  (-1)^{|\varpi_1| |\omega_{(2)}|}  \ (\omega_{(1)}|\varpi_1)(\omega_{(2)}|\varpi_2.)
$

\end{enumerate}
where
$$
\bigtriangleup(\varpi) = \Sigma_{(\varpi)}  \ \varpi_{(1)} \otimes \varpi_{(2)}, \quad \bigtriangleup(\omega)
= \Sigma_{(\omega)} \ \omega_{(1)} \otimes \omega_{(2)}
$$
denote the coproducts in the Sweedler notation (see, e.g \cite{Abe-BR}) of the elements $\varpi$ and $\omega$ in the
supersymmetric Hopf algebra of $W$
(see, e.g \cite{Bri-BR})   and in
the free exterior Hopf algebra generated by
$j = 1, 2, \ldots, d$, respectively.

\end{proposition}

\begin{example} Let $\omega = x_1 x_2 x_3$, $\varpi = 1 2 3$. Then
$$
(\omega|\varpi) = D_{x_1,\gamma} D_{x_2,\gamma} D_{x_3,\gamma} \big( (\gamma|1)(\gamma|2)(\gamma|3) \big) =
-det \big[ (x_i|j) \big]_{i,j = 1,2,3},
$$
where the $(x_i|j)$'s are commutative variables.
\end{example}

\begin{example} Let $\omega = \alpha_1 \alpha_2 x_3$, $\varpi = 1 2 3$,
where $|(\alpha_1|j)| = |(\alpha_2|j)| = 1,$  $j = 1, 2, 3$ and $|(x_3|j)|  = 0, \ j = 1, 2, 3$.
Then
\begin{align*}
(\omega|\varpi) &= D_{ \alpha_1,\gamma}D_{ \alpha_2,\gamma}D_{ x_3,\gamma}  \big( (\gamma|1)(\gamma|2)(\gamma|3) \big) \\
&=   D_{ \alpha_1,\gamma}D_{ \alpha_2,\gamma} \Big( (x_3|1)(\gamma|2)(\gamma|3)
- (\gamma|1)(x_3|2)(\gamma|3)   +  (\gamma|1)(\gamma|2)(x_3|3) \Big) \\
&= D_{ \alpha_1,\gamma} \Big( (x_3|1)(\alpha_2|2)(\gamma|3) + (x_3|1)(\gamma|2)(\alpha_2|3)
- (\alpha_2|1)(x_3|2)(\gamma|3) \\
& \phantom{D_{ \alpha_1,\gamma} \Big( \quad}- (\gamma|1)(x_3|2)(\alpha_2|3)
+ (\alpha_2|1)(\gamma|2)(x_3|3) + (\gamma|1)(\alpha_2|2)(x_3|3) \Big) \\
&= (x_3|1)(\alpha_2|2)(\alpha_1|3) + (x_3|1)(\alpha_1|2)(\alpha_2|3)
- (\alpha_2|1)(x_3|2)(\alpha_1|3) \\
& \phantom{= } - (\alpha_1|1)(x_3|2)(\alpha_2|3)
+ (\alpha_2|1)(\alpha_1|2)(x_3|3) + (\alpha_1|1)(\alpha_2|2)(x_3|3).
\end{align*}

\noindent From Proposition \ref{Laplace expansions}.1, by setting $\varpi_1 = 1 2$, $\varpi_2 = 3$, it follows
$$
(\omega|\varpi) = (\alpha_1\alpha_2|12)(x_3|3) + (\alpha_1x_3|12)(\alpha_2|3) + (\alpha_2x_3|12)(\alpha_1|3).
$$

\noindent From Proposition \ref{Laplace expansions}.2, by setting $\omega_1 = \alpha_1 \alpha_2$, $\omega_2 = x_3$, it follows
$$
(\omega|\varpi) = (\alpha_1\alpha_2|12)(x_3|3) - (\alpha_1\alpha_2|13)(x_3|2) + (\alpha_1\alpha_2|23)(x_3|1).
$$
\end{example}

\subsubsection{Biproducts and polarization operators}

Following the notation introduced in the previous sections, let
$$
Super[W] = Sym[W_0] \otimes \Lambda[W_1]
$$
denote the {\it{(super)symmetric}} algebra of the space
$$
W = W_0 \oplus W_1
$$ (see, e.g. \cite{Scheu-BR}, \cite{VARAD-BR}).

By multilinearity, the algebra
$Super[W]$ is the same as the superalgebra  $Super[A_0 \cup  A_1 \cup L]$ generated by the "variables"
$$
 \alpha_1, \ldots, \alpha_{m_0}  \in  A_0, \quad \beta_1, \ldots, \beta_{m_1} \in A_1, \quad x_1, \ldots, x_n \in L,
$$
modulo the congruences
$$
z z' = (-1)^{|z| |z'|} z' z, \quad z, z' \in A_0 \cup  A_1 \cup L.
$$

Let $d_{z, z'}$ denote the polarization operator of $z'$ to $z$ on
$$
Super[W] = Super[A_0 \cup  A_1 \cup L],
$$
that is the unique superderivation of $\mathbb{Z}_2$-degree
$$
|z| + |z'| \in \mathbb{Z}_2
$$
such that
$$
d_{z, z'} (z'') = \delta_{z', z''} \cdot z,
$$
for every $z, z', z'' \in A_0 \cup  A_1 \cup L.$

Clearly, the map
$$
e_{z, z'}  \rightarrow   d_{z, z'}
$$
is a Lie superalgebra map and, therefore, induces a structure of
$$gl(m_0|m_1+n)-module$$
on
$Super[A_0 \cup  A_1 \cup L] = Super[W].$

\begin{proposition}\label{polarization biproduct}

Let $\varpi = j_{t_1}j_{t_2} \cdots j_{t_q}$ be a word on $P = \{1, 2, \ldots, d \}$.
The map
$$
\Phi_\varpi : \omega \mapsto (\omega|\varpi),
$$
$\omega$ any word on $A_0 \cup  A_1 \cup L$, uniquely defines $gl(m_0|m_1+n)-$equivariant linear operator
$$
\Phi_\varpi : Super[A_0 \cup  A_1 \cup L] \rightarrow {\mathbb C}[M_{m_0|m_1+n,d}],
$$
that is
\begin{equation}\label{polarization action}
\Phi_\varpi \big( d_{z, z'}(\omega) \big) = D_{z, z'} \big( (\omega|\varpi) \big),
\end{equation}
for every $z, z' \in A_0 \cup  A_1 \cup L.$
\end{proposition}

With a slight abuse of notation, we will write (\ref{polarization action}) in the form
\begin{equation}\label{abuse}
D_{z, z'} \big( (\omega|\varpi) \big) = ( d_{z, z'}(\omega)|\varpi).
\end{equation}

\subsubsection{Bitableaux in ${\mathbb C}[M_{m_0|m_1+n,d}]$}

Let $S = (\omega_1, \omega_2, \ldots, \omega_p$ and $T = (\varpi_1, \varpi_2, \ldots, \varpi_p)$ be Young tableaux on
$A_0 \cup  A_1 \cup L$ and $P = \{1, 2, \ldots, d \}$ of shapes $\lambda$ and $\mu$, respectively.

If $\lambda = \mu$, the {\it{Young bitableau}} $(S|T)$ is the element of ${\mathbb C}[M_{m_0|m_1+n,d}]$ defined as follows:
$$
(S|T) =
\left(
\begin{array}{c}
\omega_1\\ \omega_2\\ \vdots\\ \omega_p
\end{array}
\right| \left.
\begin{array}{c}
\varpi_1\\ \varpi_2\\ \vdots\\  \varpi_p
\end{array}
\right)
= \pm \ (\omega_1)|\varpi_1)(\omega_2)|\varpi_2) \cdots (\omega_p)|\varpi_p),
$$
where
$$
\pm  = (-1)^{|\omega_2||\varpi_1|+|\omega_3|(|\varpi_1|+|\varpi_2|)+ \cdots +|\omega_p|(|\varpi_1|+|\varpi_2|+\cdots+|\varpi_{p-1}|)}.
$$

If $\lambda \neq \mu$, the {\it{Young bitableau}} $(S|T)$ is set to be zero.

\subsubsection{Bitableaux and polarization operators}

By naturally extending the slight abuse of notation (\ref{abuse}), the action of any polarization on bitableaux
can be explicitly described:

\begin{proposition}\label{action on tableaux}
Let $z, z' \in A_0 \cup  A_1 \cup L$,  and let
$S = (\omega_1, \ldots, \omega_p) $, $T =
(\varpi_1, \ldots, \varpi_p)$. We have the
following identity:
\begin{align*}
D_{z, z'}(S\,|\,T)  \ & = \
 D_{z, z'} \ \left(
\begin{array}{c}
\omega_1\\ \omega_2\\ \vdots\\ \omega_p
\end{array}
\right| \left.
\begin{array}{c}
\varpi_1\\ \varpi_2\\ \vdots\\  \varpi_p
\end{array}
\right)  \\ &
= \ \sum_{s=1}^p  \
(-1)^{(|z| + |z'|)\epsilon_s}
\ \left(
\begin{array}{c}
\omega_1\\ \omega_2\\ \vdots\\  d_{z,
z'}(\omega_s)\\ \vdots \\ \omega_p
\end{array}
\right| \left.
\begin{array}{c}
\varpi_1\\ \varpi_2\\ \vdots\\
\vdots \\ \vdots\\ \varpi_p
\end{array}
\right),
\end{align*}
where
$$
\epsilon_1 = 1, \quad    \epsilon_s = |\omega_1| + \cdots + |\omega_{s-1}|, \quad s = 2,
\ldots, p.
$$
\end{proposition}

\begin{example} Let $\alpha_i \in A_0$, $x_1,  x_2, x_3, x_4 \in L$, $|D_{\alpha_i, x_2}| = 1$. Then
$$
D_{\alpha_i, x_2} \
\left(
\begin{array}{lll}
x_1 \ x_3 \ x_2\\
 x_2 \ x_3\\
x_4 \ x_2
\end{array}
\right| \left.
\begin{array}{lll}
1 \ 2 \ 3 \\
2 \ 3 \\
3 \ 1
\end{array}
\right) =
$$
$$ =
\left(
\begin{array}{lll}
x_1 \ x_3 \ \alpha_i\\
 x_2 \ x_3\\
x_4 \ x_2
\end{array}
\right| \left.
\begin{array}{lll}
1 \ 2 \ 3 \\
2 \ 3 \\
3 \ 1
\end{array}
\right) -
\left(
\begin{array}{lll}
x_1 \ x_3 \ x_2\\
 \alpha_i \ x_3\\
x_4 \ x_2
\end{array}
\right| \left.
\begin{array}{lll}
1 \ 2 \ 3 \\
2 \ 3 \\
3 \ 1
\end{array}
\right) +
\left(
\begin{array}{lll}
x_1 \ x_3 \ x_2\\
 x_2 \ x_3\\
x_4 \ \alpha_i
\end{array}
\right| \left.
\begin{array}{lll}
1 \ 2 \ 3 \\
2 \ 3 \\
3 \ 1
\end{array}
\right)
$$

\end{example}

\subsubsection{The straightening algorithm and the standard basis theorem for ${\mathbb C}[M_{m_0|m_1+n,d}]$}

Consider the set of all bitableaux $(S|T) \in {\mathbb C}[M_{m_0|m_1+n,d}]$, where $sh(S) = sh(T) \vdash h$,
$h$ a given positive integer. In the following, let
denote by $\leq$ the partial order on this set defined  by the following two steps:
\begin{itemize}

\item
$(S|T) < (S'|T')$ whenever $sh(S)  <_l sh(S') $,

\item
$(S|T) < (S'|T')$ whenever $sh(S) =  sh(S')$, $w(S) >_l w(S')$,  $w(T) >_l w(T')$,
\end{itemize}
where the shapes and the row-words are compared in the lexicographic
order.

The next results are superalgebraic versions of classical, well-known results for the symmetric algebra ${\mathbb C}[M_{n,d}]$
(\cite{drs-BR}, \cite{DKR-BR}, \cite{DEP-BR}, for the general theory of standard monomials see, e.g. \cite{Procesi-BR}, Chapt. 13)
and of their skew-symmetric analogues  (\cite{DR-BR}, \cite{ABW-BR}).

\begin{theorem}(The straightening algorithm)\label{theorem: standard expansion of bitableaux} \cite{rota-BR}

Let $(P|Q) \in {\mathbb C}[M_{m_0|m_1+n,d}]$.
Then $(P|Q)$ can be
written as a linear combination, with rational coefficients,
\begin{equation}\label{straightening}
(P|Q) = \sum_{S,T} c_{S,T}\ (S|T),
\end{equation}
of standard bitableaux $(S|T)$, where $(S|T) \geq (P|Q)$ and $c_S = c_P$, $c_T = c_Q$.
\end{theorem}
Since standard bitableaux are linearly independent in ${\mathbb C}[M_{m_0|m_1+n,d}]$, the expansion
(\ref{straightening}) is unique.

Following \cite{Berele1-BR}, \cite{Berele2-BR},  \cite{Brini1-BR},
a partition $\lambda$ satisfies the $(m_0,m_1+n)-${\it{hook condition}} (in symbols, $\lambda \in H(m_0,m_1+n)$)
if and only if $\lambda_{m_0+1} \leq m_1+n.$ We have:
\begin{lemma}
There exists a standard tableau on $A_0 \cup  A_1 \cup L$ of shape $\lambda$ if
and only if  $\lambda \in H(m_0,m_1+n).$
\end{lemma}

Given a positive integer $h \in \mathbb{Z}^+$,
let ${\mathbb C}_h[M_{m_0|m_1+n,d}]$ denote the $h-$th homogeneous component of
${\mathbb C}[M_{m_0|m_1+n,d}].$

From Theorem \ref{theorem: standard expansion of bitableaux}, it follows

\begin{corollary}\label{theorem: standard basis}
(The  Standard basis theorem for
${\mathbb C}_h[M_{m_0|m_1+n,d}]$, \cite{rota-BR})

The following set is a basis of ${\mathbb C}_h[M_{m_0|m_1+n,d}]$:
$$
\{ (S|T) \ standard;\ sh(S) = sh(T) = \lambda \vdash h,   \lambda \in H(m_0,m_1+n), \lambda_1 \leq d \ \}.
$$
\end{corollary}

\subsection{The Schur (covariant) ${\mathbf{U}}(gl(m_0|m_1+n))-$supermodules}\label{Schur}

\subsubsection{The Schur  supermodules as
submodules of $\mathbb{C}[M_{m_0|m_1+n,d}]$}

In the following, we regard the superalgebra ${\mathbb C}[M_{m_0|m_1+n,d}]$ as  $\mathbf{U}(gl(m_0|m_1+n))-$module.
We recall the definition of a ``Schur supermodule'' and the main facts about them. This material goes back to our work of $1988-1989$
(\cite{Brini1-BR}, \cite{Brini2-BR}); our construction produces irreducible $gl(m_0|m_1+n))-$supermodules that are isomorphic
to the modules constructed by Berele and Regev \cite{Berele1-BR}, \cite{Berele2-BR} as tensor modules induced by Young symmetrizers
(see, e.g. \cite{Weyl-BR}) when they act by a ``signed action" of the symmetric group (see also \cite{Dondi-BR},  \cite{King-BR}).
The description presented here
is simpler than the tensor description, provides a close connection with the superstraightening theory of Grosshans, Rota and Stein
\cite{rota-BR}, and allows the action of $\mathbf{U}(gl(m_0|m_1+n))$ to be described in a transparent way
(Proposition \ref{action on tableaux}).

Given $\lambda \in H(m_0,m_1+n)$, the {\it{Schur supermodule}} $Schur_\lambda(m_0,m_1+n)$ is
the subspace of ${\mathbb C}[M_{m_0|m_1+n,d}]$, $d \geq \lambda_1$, spanned by the set of all bitableaux
$(S|D^P_\lambda)$ of shape $\lambda,$ where $D^P_\lambda$ is the {\it{Deruyts}} tableau on $P = \{1, 2, \ldots, d \}$
$$
D^P_\lambda =  \left(
\begin{array}{llllllllllllll}
1  \ldots    \ldots     \ldots    & \lambda_1     \\
1   \ldots  \ldots               \lambda_2 \\
 \ldots  \ldots  & \\
1 \ldots \lambda_p
\end{array}
\right), \quad p = \textit{l}(\lambda).
$$
and $S$ is a Young tableau on the alphabet $A_0 \cup  A_1 \cup L.$

From Theorem \ref{theorem: standard expansion of bitableaux} and Corollary \ref{theorem: standard basis}, it follows

\begin{proposition}

The set
$$
\Big\{ (S|D^P_\lambda); S \ superstandard   \       \Big\}
$$
is a ${\mathbb C}-$linear basis of $Schur_\lambda(m_0,m_1+n)$.
\end{proposition}

Furthermore, we recall
\begin{proposition}(\cite{Brini1-BR}, \cite{Bri-BR})
The submodule $Schur_\lambda(m_0,m_1+n)$ is an irreducible $\mathbf{U}(gl(m_0|m_1+n))-$submodule of
${\mathbb C}[M_{m_0|m_1+n,d}]$, with highest weight
$$
(\lambda_1, \ldots, \lambda_{m_0}; \widetilde{\lambda}_1-m_0, \widetilde{\lambda}_2-m_0, \ldots ).
$$
\end{proposition}

\subsubsection{The classical Schur $gl(n)-$modules}\label{Schur modules}

Given $\lambda$ such that $\lambda_1 \leq n$, the {\it{Schur module}} $Schur_\lambda(n)$ is
the subspace of ${\mathbb C}[M_{n,d}]$, $d \geq \lambda_1$, spanned by the set of all bitableaux
$(X|D^P_\lambda)$ of shape $\lambda$  and $X$ is a Young tableau on the alphabet $ L.$

\begin{proposition}

The set
$$
\Big\{ (X|D^P_\lambda); X \ standard   \       \Big\}
$$
is a ${\mathbb C}-$linear basis of $Schur_\lambda(n)$.

Furthermore, $Schur_\lambda(n)$ is an irreducible $\mathbf{U}(gl(n))-$submodule of
${\mathbb C}[M_{n,d}]$, with highest weight $\widetilde{\lambda}.$
\end{proposition}

Let
$$
D_{\lambda} = \left(
\begin{array}{llllllllllllll}
x_1  \ldots    \ldots     \ldots    & x_{\lambda_1}     \\
x_1   \ldots  \ldots               x_{\lambda_2} \\
 \ldots  \ldots  & \\
x_1 \ldots x_{\lambda_p}
\end{array}
\right)
$$
denote the (proper) Deruyts tableau on the alphabet $L = \{x_1, x_2, \ldots, x_n \}.$
The element
$$
v_{\widetilde{\lambda}} = (D_\lambda|D^P_\lambda)
$$
is the ``canonical''  highest weight vector of the irreducible $gl(n)-$module $Schur_\lambda(n)$,
with highest weight $\widetilde{\lambda}$.

\subsubsection{The classical Schur modules as $gl(n)-$submodules of Schur supermodules}\label{Lemmas}

Let $\lambda$ be a partition such that $\lambda_1 \leq n$.

Given $m_0 \geq \textit{l}(\lambda)$, $m_1, d \geq \lambda_1$, consider the
Schur supermodule
$$
Schur_\lambda(m_0,m_1+n)
$$
(clearly, $\lambda \in H(m_0,m_1+n)$.)

The Schur module $Schur_\lambda(n)$ can be regarded as a $\mathbf{U}(gl(n))-$submodule of the
$\mathbf{U}(gl(m_0|m_1+n))-$supermodule $Schur_\lambda(m_0,m_1+n).$

Let $\mathfrak{p}$ be the Capelli epimorphism (Theorem \ref{Capelli epimorphism})
$$
\mathfrak{p} : Virt(m_0+m_1, n) \twoheadrightarrow
{\mathbf{U}}(gl(n)), \qquad  Ker(\mathfrak{p}) = \mathbf{Irr}.
$$

\begin{proposition}\label{virtual action}

The Schur module $Schur_\lambda(n)$ is  invariant (as a subspace of $Schur_\lambda(m_0,m_1+n)$) with respect to the action
of the subalgebra
$$
Virt(m_0+m_1, n)   \subset \mathbf{U}(gl(m_0|m_1+n)).
$$

Furthermore, for every element $\rho \in Virt(m_0+m_1, n)$, the action of $\rho$ on the Schur module $Schur_\lambda(n)$
is the same of the action of its image $\mathfrak{p}(\rho)) \in {\mathbf{U}}(gl(n)).$

\end{proposition}

\begin{proposition}{\textbf{(Vanishing Lemmas)}}\label{Vanishing Lemma}
Let $v_{\widetilde{\mu}} = (D_\mu|D^P_\mu)$ be the ``canonical''  highest weight vector of the irreducible $gl(n)-$module $Schur_\mu(n)$,
and let $\trianglelefteq$ denote the (partial) dominance order on partitions.
We have
\begin{align}
&\textrm{If} \ |\mu| < |\lambda|, \  \textrm{then}  &   D_\lambda^*S(v_{\widetilde{\mu}}) &= 0,\quad \forall S\label{uno}
\\
&\textrm{If} \ |\mu| < |\lambda|, \  \textrm{then}  &   C_\lambda^*S(v_{\widetilde{\mu}}) &= 0,\quad \forall S\label{due}
\\
&\textrm{If} \  |\mu| = |\lambda|, \ \mu \ntrianglerighteq \lambda, \  \textrm{then}  &   C_\lambda^*S(v_{\widetilde{\mu}}) &= 0,\quad \forall S\label{duebis}
\\
&\textrm{If} \  |\mu| = |\lambda|, \ \widetilde{\mu} \ntrianglerighteq \lambda, \  \textrm{then}  &   D_{\widetilde{\lambda}}^*S(v_{\widetilde{\mu}}) &= 0,
\quad \forall S\label{unobis}
\\
&\textrm{If} \ |\mu| = |\lambda|, \ \mu \neq \lambda, \ \textrm{then}  &  D_\lambda^*C_\lambda^* \ C_\lambda^*S(v_{\widetilde{\mu}}) &= 0,
\quad \forall S\label{tre}
\\
&\textrm{If} \ |\mu| = |\lambda|, \ \mu \neq \lambda, \ \textrm{then}  &  C_\lambda^*D_\lambda^* \ D_\lambda^*S(v_{\widetilde{\mu}}) &= 0,
\quad \forall S\label{quattro}
\\
&\textrm{If} \ \lambda  \nsubseteq \mu,  \ \textrm{then}  &  D_\lambda^*C_\lambda^* \ C_\lambda^*S(v_{\widetilde{\mu}}) &= 0,
\quad \forall S\label{cinque}
\end{align}
\end{proposition}
\begin{proof}
The assertions of eqs. (\ref{uno}), (\ref{due}), (\ref{duebis}), (\ref{unobis}), (\ref{tre}), (\ref{quattro}) are special cases of standard
elementary facts of the method
of virtual variables (see e. g. \cite{Bri-BR}).
About assertion of eq. (\ref{cinque}), assume  that $|\mu| \geq |\lambda|$ to avoid trivial cases (by eq. (\ref{due})).
The action $C_\lambda^*S(v_{\widetilde{\mu}}) = C_\lambda^*S \big( (D_\mu|D^P_\mu) \big)$ produces a linear combination of bitableaux
$(T|D^P_\mu) \in Schur_\mu(m_0,m_1+n)$, where each tableau $T$ contains exactly $\lambda_i$ occurrences of the positive virtual
symbols $\alpha_i \in A_0$.
By  {\it{straightening}} each of them (Theorem \ref{theorem: standard expansion of bitableaux}),
the element $C_\lambda^*S \big( (D_\mu|D^P_\mu) \big)$ is uniquely expressed as a linear combination of
(super)standard tableaux
\begin{equation}\label{dominance}
C_\lambda^*S \big( (D_\mu|D^P_\mu) \big) = \sum_i \ (S_i|D^P_\mu) \in Schur_\mu(m_0,m_1+n),
\end{equation}
where in each $S_i$ the positive virtual symbols $\alpha_i \in A_0$ occupies a subshape $\lambda' \subseteq \mu$ such that
$\lambda'   \trianglerighteq \lambda$.
If $\lambda \nsubseteq \mu$, any element $(S_i|D^P_\mu)$ in the canonical form (\ref{dominance}) is such that $\lambda'   \trianglerighteq \lambda$,
$\lambda'   \neq \lambda$. Then $D_\lambda^*C_\lambda^* \big( (S_i|D^P_\mu) \big) = 0$, by skew-symmetry, and the assertion follows.
\end{proof}

We recall that, given a shape $\lambda$,
the {\it{hook length}}  $H(x)$ of a box $x$ in the Ferrers diagram $F_\lambda$ of the shape $\lambda$
is the number of boxes that are in the same row to the
right of it plus those boxes in the same column below it, plus one (for the box itself).
The {\it{hook number}} the shape $\lambda$ is the product $H(\lambda) = \prod_{x \in F_\lambda} \ H(x).$

\begin{proposition}{\textbf{(Regonati's Hook Lemma, \cite{Regonati-BR})}}\label{hook lemma}
Let $H(\lambda)$ denotes the hook number of the shape (partition) $\lambda \vdash k.$
We have
\begin{align}
C_{\lambda}^{*}D_\lambda  (v_{\widetilde{\lambda}}) &=& C_{\lambda}^{*}D_\lambda \big( (D_\lambda|D^P_\lambda) \big)
&= (-1)^{{k} \choose   {2}} \
H(\lambda) \cdot  (C_{\lambda}^{*}|D^P_\lambda) \ (\lambda!)^{-1}\label{proper}
\\
\phantom{\quad C_{\lambda}^{*}D_\lambda  (v_{\widetilde{\lambda}})} &\ &
C_{\lambda}^{*}D_\lambda^* \big( (D_\lambda^*|D^P_\lambda) \big)
&= (-1)^{{k} \choose   {2}} \
H(\lambda) \cdot  (C_{\lambda}^{*}|D^P_\lambda) \ (\lambda!)^{-1}\label{virtual}.
\end{align}
Furthermore
\begin{equation}
D_{\lambda}^{*}C_\lambda^{*} \big( (C_{\lambda}^{*}|D^P_\lambda) \ (\lambda!)^{-1} \big) =
\quad  (D_{\lambda}^{*}|D^P_\lambda)\label{trivial}
\end{equation}
\end{proposition}

\section{The center $\boldsymbol{\zeta}(n)$ of $\mathbf{U}(gl(n))$}\label{The center}

 In order to make  the notation lighter, in this section we simply write $1, 2, \ldots, n$ in place of
$x_1, x_2, \ldots, x_n.$

\begin{remark}\label{centrality-BR}
Throughout this section  the role of Proposition \ref{rappresentazione aggiunta-BR}  is ubiquitous: in order to prove
that an element $\mathbf{\rho}$ is central
in $\mathbf{U}(gl(n))$, we can simply prove that a virtual presentation $\mathbf{\rho}'$ in $Virt(m_0+m_1,n)$ of $\mathbf{\rho}$
 (to wit, an element $\mathbf{\rho}'$ that is mapped to $\mathbf{\rho}$ by
the Capelli devirtualition epimorphism $\mathfrak{p}$, see Subsection 3.5.1)
is an {\it{invariant}} for the adjoint action of $\mathbf{U}(gl(n))$ on $Virt(m_0+m_1,n)$.
In symbols
$$
ad(e_{i j}) \left( \mathbf{\rho}' \right) = 0, \quad  \forall e_{i j} \in gl(n).
$$
\end{remark}

\subsection{The virtual form of the determinantal Capelli generators $\mathbf{H}_k(n)$ of $1893$}\label{Capelli 1893}

\subsubsection{The classical and the virtual definitions of the  Capelli generators $\mathbf{H}_k(n)$, and main results}
In the enveloping algebra $\mathbf{U}(gl(n))$, given any integer $k = 1, 2, \ldots, n,$ consider the element
(compare with Example \ref{Capelli determinants-BR})
\begin{align}\label{The classical Capelli generators of $1893$-BR}
\mathbf{H}_k(n) & =
 \sum_{1 \leq i_1 < \cdots < i_k \leq n} \ [ i_k \cdots i_2 i_1 | i_1 i_2 \cdots i_k ] = \\
 & =  \sum_{1 \leq i_1 < \cdots < i_k \leq n} \ \mathfrak{p} \big( e_{i_k , \alpha} \cdots e_{i_2, \alpha} e_{i_1, \alpha}
e_{\alpha, i_1}e_{\alpha, i_2} \cdots e_{\alpha, i_k } \big) = \label{30}
\\
& = \sum_{1 \leq i_1 < \cdots < i_k \leq n}  \mathbf{cdet}\left(
 \begin{array}{cccc}
 e_{{i_1},{i_1}}+(k-1) & e_{{i_1},{i_2}} & \ldots  & e_{{i_1},{i_k}} \\
 e_{{i_2},{i_1}} & e_{{i_2},{i_2}}+(k-2) & \ldots  & e_{{i_2},{i_k}}\\
 \vdots  &    \vdots                            & \vdots &  \\
e_{{i_k},{i_1}} & e_{{i_k},{i_2}} & \ldots & e_{{i_k},{i_k}}\\
 \end{array}
 \right) \label{31}
\end{align}
where $\alpha \in A_0$ denotes {\it{any}} positive virtual symbol.

\begin{remark} Clearly, as apparent  from the virtual presentation  (\ref{30}), each summand $[ i_k \cdots i_2 i_1 | i_1 i_2 \cdots i_k ]$ is
skew-symmetric both in the left and the right sequences.
As already observed in the Introduction, the ``nonvirtual form'' of the elements $\mathbf{H}_k(n)$ is  harder to manage than their
virtual form (see for example the proof of the centrality, Proposition \ref{central virtual determinantal} and
Theorem \ref{H centrality}).
\end{remark}

\begin{proposition}\label{central virtual determinantal}
Since the adjoint representation acts by derivation, we have
$$
ad(e_{i j})\big(  \sum_{1 \leq i_1 < \cdots < i_k \leq n} \  e_{i_k , \alpha} \cdots e_{i_2, \alpha} e_{i_1, \alpha}
e_{\alpha, i_1}e_{\alpha, i_2} \cdots e_{\alpha, i_k } \ \big) = 0 ,          \quad \forall  e_{i j} \in gl(n).
$$
\end{proposition}
\begin{proof}
 In order to make  the notation lighter, in this proof we write
 $$
 \{i_k  \cdots i_2 \ i_1|i_1 \ i_2 \cdots  i_k \}
 $$
 in place of
 $e_{i_k , \alpha} \cdots e_{i_2, \alpha} e_{i_1, \alpha}e_{\alpha, i_1}e_{\alpha, i_2} \cdots e_{\alpha, i_k }$,
 that is
\begin{align*}
\begin{split}
\sum_{1 \leq i_1 < \cdots < i_k \leq n} \  e_{i_k , \alpha} \cdots e_{i_2, \alpha} e_{i_1, \alpha}
e_{\alpha, i_1}e_{\alpha, i_2}& \cdots e_{\alpha, i_k } =
 \\ \null \hfill
& =\sum_{1 \leq i_1 < \cdots < i_k \leq n} \  \{i_k  \cdots i_2 \ i_1|i_1 \ i_2 \cdots  i_k \}.
\end{split}
\end{align*}
In the following, given an decreasing word $v = i_s \cdots i_2i_1$ on the set $1, 2, \ldots, n$, we denote by
$\overline{v} = i_1i_2 \cdots i_s$ its reverse.

Without loss of generality, let us consider the adjoint action of an element $e_{ij}$, with $i > j$.
note that the action on a summand $\{i_k  \cdots i_2 \ i_1|i_1 \ i_2 \cdots  i_k \}$ cannot be simultaneously
nontrivial both ``on the left'' and ``on the right''.

The action ``on the left''  is nontrivial only on the  summands of the form
$\{u \ \widehat{i} \ v \ j \ w|\overline{w} \ j \ \overline{v} \ \widehat{i} \ \overline{u}\}$ and it produces the
element
$$
\{u \ \widehat{i} \ v \ i \ w|\overline{w} \ j \ \overline{v} \ \widehat{i} \ \overline{u}\} =
(-1)^{l(v)}\{u \ i \ v  \  \widehat{j} \ w|\overline{w} \ j \ \overline{v} \ \widehat{i} \ \overline{u}\}
$$.

Consider the unique summand $\{u \ i \ v  \  \widehat{j} \ w|\overline{w} \ \widehat{j} \ \overline{v} \ i \ \overline{u}\}$;
the adjoint action ``on the right'' produces the element
$$
-\{u \ i \ v  \  \widehat{j} \ w|\overline{w} \ \widehat{i} \ \overline{v} \ j \ \overline{u}\}  =
-(-1)^{l(v)}\{u \ i \ v  \  \widehat{j} \ w|\overline{w} \ j \ \overline{v} \ \widehat{i} \ \overline{u}\}.
$$
Then, after the action of the adjoint representation,  the summands  delete pairwise.

The action ``on the right''  is nontrivial only on the  summands of the form
$\{u \ i \ v \ \widehat{j} \ w|\overline{w} \ \widehat{j} \ \overline{v} \ i \ \overline{u}\}$ and the reasoning is the same as before.
\end{proof}

From Remark \ref{centrality-BR}, it follows

\begin{theorem}\label{H centrality}
The elements $\mathbf{H}_k(n)$ are  central in $\mathbf{U}(gl(n))$.
\end{theorem}

Let $\boldsymbol{\zeta}(n)^{(m)}$ denote the $m-$th filtration element of $\boldsymbol{\zeta}(n)$ with respect to the filtration induced by
the standard filtration of $\mathbf{U}(gl(n)).$

Clearly,
$$
\mathbf{H}_k(n) \in \boldsymbol{\zeta}(n)^{(m)},
$$
for every $m \geq k.$

We recall the following fundamental result, indeed proved by  Capelli in  two  papers (\cite{Cap2-BR}, \cite{Cap3-BR}) with deceiving titles
(for a faithful description of
Capelli's original proof , quite simplified by means
of the superalgebraic method of virtual variables, see \cite{Brini4-BR}, $1993$)\footnote{We are indebted to Kostant and
Sahi (\cite{KostantSahi1-BR}, p. $72$),
who wrote in $1991$: ``It is remarkable that, in some sense, Capelli was
already aware of this connection!   \cite{Borel-BR}, p. $77$''.}

\begin{theorem}(Capelli, 1893)\label{Capelli generators}

The set
$$
\mathbf{H}_1(n), \mathbf{H}_2(n), \ldots, \mathbf{H}_n(n)
$$
is a set of algebraically independent generators of the center $\boldsymbol{\zeta}(n)$ of
$\mathbf{U}(gl(n)).$
\end{theorem}
We recall that  $v_{\widetilde{\mu}} = (D_\mu|D^P_\mu)$ denoted the  ``canonical''
highest weight vector of the Schur module $Schur_{\mu}(n)$, $\mu_1 \leq n$,  which is indeed of weight
$\widetilde{\mu}$ (Subsection \ref{Schur modules}).

Furthermore, we will write $\mathbf{H}_k(n)(v_{\widetilde{\mu}})$ to mean the action of the central element
$\mathbf{H}_k(n)$ on $v_{\widetilde{\mu}}.$

Let $e^*_k(\widetilde{\mu}) \in \mathbb{N}$ denote the sum of the $k-$th ``partial hooks'' numbers of the shape $\mu$,
that is
\begin{equation}\label{elementary shift}
e^*_k(\widetilde{\mu})
= \sum_{1 \leq i_1 < i_2 < \cdots < i_k \leq n} \ (\widetilde{\mu}_{i_1}  + k  - 1)
(\widetilde{\mu}_{i_2}  + k - 2) \cdots (\widetilde{\mu}_{i_k}).
\end{equation}

The classical presentation (\ref{31}) of the $\mathbf{H}_k(n)$'s implies the following result.

\begin{proposition}\label{Capelli eigenvalues}
We have

\begin{enumerate}

\item
$
\mathbf{H}_k(n)(v_{\widetilde{\mu}}) = e^*_k(\widetilde{\mu})\cdot v_{\widetilde{\mu}}, \quad e^*_k(\widetilde{\mu}) \in \mathbb{Z}.
$
\item
$
\mathbf{H}_k(n)(v_{\widetilde{\mu}}) = 0,
$
if $\mu_1 < k.$
\end{enumerate}

\end{proposition}

The ``{\it{virtual presentation}}'' (\ref{30}) of the $\mathbf{H}_k(n)$'s leads to a further combinatorial description of
the integer eigenvalues $e^*_k(\widetilde{\mu})$, which will turn out to be crucial
in the section on {\it{duality}}.

\begin{proposition}\label{horizontal strip} We have

$$
e^*_k(\widetilde{\mu}) = \sum \ hstrip_{\mu}(k)!,
$$
where the sum is extended to all ``horizontal strips'' \footnote{In this work, we use the expression
{\it{horizontal strip}} in a generalized sense. To wit, a  horizontal strip in a Ferrers diagram is a subset
of cells such that no two cells in the subset appear in the same column.}
of length $k$ in the Ferrers diagram of the partition $\mu$,
and the symbol $\ hstrip_{\mu}(k)!$ denotes the products of the factorials of the cardinality of each horizontal
component of the horizontal strip.
\end{proposition}

\begin{example}\label{horizontal example} Let $\mu = (3,2)$.
Given a bitableau   $(S|D_\mu^P)$ in the supermodule $Schur_\mu(m_0|m_1+3)$, we will write $(S|$ in place  of $(S|D_\mu^P),$ in order
to simplify the notation.

In particular, we write
$$
\left(
\begin{array}{lll}
1 \ 2 \ 3 \\
1 \ 2
\end{array}
\right|
$$
in place of the $gl(3)-$highest weight vector of the $gl(3)-$irreducible submodule  $Schur_\mu(3)$ of $Schur_\mu(m_0|m_1+3)$:
$$
(D_{(3,2)}|D_{(3,2)}^P)
=
\left(
\begin{array}{lll}
1 \ 2 \ 3\\
 1 \ 2\\
\end{array}
\right| \left.
\begin{array}{lll}
1 \ 2 \ 3 \\
1 \ 2\\
\end{array}
\right),
$$
of weight $\widetilde{\mu} = (2,2,1).$

The action of $\mathbf{H}_2(3)$ on $(D_\mu|D_\mu^P)$ is the same as the action of
$$
 e_{2, \alpha} e_{1, \alpha}e_{\alpha, 1}e_{\alpha, 2} + e_{3, \alpha} e_{1, \alpha}e_{\alpha, 1}e_{\alpha, 3} +
 e_{3, \alpha} e_{2, \alpha}e_{\alpha, 2}e_{\alpha, 3}, \quad \alpha \in A_0, \ |\alpha| = 0;
$$
hence, we have to compute
$$
\Big( D_{2, \alpha} D_{1, \alpha}D_{\alpha, 1}D_{\alpha, 2} + D_{3, \alpha} D_{1, \alpha}D_{\alpha, 1}D_{\alpha, 3} +
 D_{3, \alpha} D_{2, \alpha}D_{\alpha, 2}D_{\alpha, 3} \Big) \
\left(
\begin{array}{lll}
1 \ 2 \ 3 \\
1 \ 2
\end{array}
\right|.
$$
By considering the action of the ``virtualizing part'' of each summand, we have
\begin{align*}
D_{\alpha, 1}D_{\alpha, 2} \ \left(
\begin{array}{lll}
1 \ 2 \ 3 \\
1 \ 2
\end{array}
\right| &=
- \left(
\begin{array}{lll}
\alpha \ \alpha \ 3 \\
1 \ 2
\end{array}
\right|
+
 \left(
\begin{array}{lll}
\alpha \ 2 \ 3 \\
1 \ \alpha
\end{array}
\right|
-
 \left(
\begin{array}{lll}
1 \ \alpha \ 3 \\
\alpha \ 2
\end{array}
\right|
-
 \left(
\begin{array}{lll}
1 \ 2 \ 3 \\
\alpha \ \alpha
\end{array}
\right|,
\\
D_{\alpha, 1}D_{\alpha, 3} \ \left(
\begin{array}{lll}
1 \ 2 \ 3 \\
1 \ 2
\end{array}
\right| &=
 \left(
\begin{array}{lll}
\alpha \ 2 \ \alpha \\
1 \ 2
\end{array}
\right|
+
 \left(
\begin{array}{lll}
1 \ 2 \ \alpha \\
\alpha \ 2
\end{array}
\right|,
\\
D_{\alpha, 2}D_{\alpha, 3} \ \left(
\begin{array}{lll}
1 \ 2 \ 3 \\
1 \ 2
\end{array}
\right| &=
 - \left(
\begin{array}{lll}
1 \ \alpha  \ \alpha \\
1 \ 2
\end{array}
\right|
-
 \left(
\begin{array}{lll}
1 \ 2 \ \alpha \\
1 \ \alpha
\end{array}
\right|.
\end{align*}
Notice that the two occurrences of $\alpha$ distribute in all horizontal strips of length $2$ in the Ferrers diagram
of the partition $\mu = (3,2).$

By considering the action of the ``devirtualizing part'' of each summand, we have
\begin{multline*}
D_{2, \alpha} D_{1, \alpha} \Big(  - \left(
\begin{array}{lll}
\alpha \ \alpha \ 3 \\
1 \ 2
\end{array}
\right|
+
 \left(
\begin{array}{lll}
\alpha \ 2 \ 3 \\
1 \ \alpha
\end{array}
\right|
-
 \left(
\begin{array}{lll}
1 \ \alpha \ 3 \\
\alpha \ 2
\end{array}
\right|
-
 \left(
\begin{array}{lll}
1 \ 2 \ 3 \\
\alpha \ \alpha
\end{array}
\right|      \Big) =
\\
= \left(
\begin{array}{lll}
1 \ 2 \ 3 \\
1 \ 2
\end{array}
\right|
-
\left(
\begin{array}{lll}
2 \ 1 \ 3 \\
1 \ 2
\end{array}
\right|
+
\left(
\begin{array}{lll}
1 \ 2 \ 3 \\
1 \ 2
\end{array}
\right|
+
\left(
\begin{array}{lll}
1 \ 2 \ 3 \\
1 \ 2
\end{array}
\right|
+
\left(
\begin{array}{lll}
1 \ 2 \ 3 \\
1 \ 2
\end{array}
\right|
-
\left(
\begin{array}{lll}
1 \ 2 \ 3 \\
2 \ 1
\end{array}
\right| =
\\
= 6  \ \left(
\begin{array}{lll}
1 \ 2 \ 3 \\
1 \ 2
\end{array}
\right|,
\end{multline*}
\begin{multline*}
D_{3, \alpha} D_{1, \alpha} \Big( \left(
\begin{array}{lll}
\alpha \ 2 \ \alpha \\
1 \ 2
\end{array}
\right|
+
 \left(
\begin{array}{lll}
1 \ 2 \ \alpha \\
\alpha \ 2
\end{array}
\right|  \Big) = 3 \
\left(
\begin{array}{lll}
1 \ 2 \ 3 \\
1 \ 2
\end{array}
\right|,
\\
\end{multline*}

\begin{multline*}
D_{3, \alpha} D_{2, \alpha} \Big( - \left(
\begin{array}{lll}
1 \ \alpha \  \alpha \\
1 \ 2
\end{array}
\right|
-
 \left(
\begin{array}{lll}
1 \ 2 \ \alpha \\
1 \ \alpha
\end{array}
\right|  \Big) = 3 \
\left(
\begin{array}{lll}
1 \ 2 \ 3 \\
1 \ 2
\end{array}
\right|.
\\
\end{multline*}

Therefore,
\begin{multline*}
\Big( D_{2, \alpha} D_{1, \alpha}D_{\alpha, 1}D_{\alpha, 2} + D_{3, \alpha} D_{1, \alpha}D_{\alpha, 1}D_{\alpha, 3} +
 D_{3, \alpha} D_{2, \alpha}D_{\alpha, 2}D_{\alpha, 3} \Big) \
\left(
\begin{array}{lll}
1 \ 2 \ 3 \\
1 \ 2
\end{array}
\right| =
\\
= 12 \ \left(
\begin{array}{lll}
1 \ 2 \ 3 \\
1 \ 2
\end{array}
\right|.
\end{multline*}

Notice that, since $\widetilde{\mu} = (2,2,1)$, according to Corollary \ref{Capelli eigenvalues} and eq. (\ref{elementary shift}), we have
\begin{multline*}
e^*_2((2,2,1))
= \sum_{1 \leq i_1 < i_2 \leq 3}
(\widetilde{\mu}_{i_1}  + 2 - 1) (\widetilde{\mu}_{i_2}) = \\
= (2+2-1)2+(2+2-1)1+(2+2-1)1 = 12.
\end{multline*}
\end{example}
\qed

\noindent {\it{Proof of Proposition}} \ref{horizontal strip}.  The action of each summand of the ``virtualizing part''
$$
 e_{\alpha, i_1}e_{\alpha, i_2} \cdots e_{\alpha, i_k }
$$
distributes the $k$ occurrences of $\alpha$  in all horizontal strips of length $k$ (with column positions
$i_1, i_2, \ldots, i_k$) in the Ferrers diagram
of the partition $\mu,$ with signs - according to Proposition \ref{polarization biproduct} - since
$|e_{\alpha, i_h }| = 1$.
By applying   the ``devirtualizing part''
$$
e_{i_k , \alpha} \cdots e_{i_2, \alpha} e_{i_1, \alpha}
$$
it is easy to see that, for each horizontal strip, we obtain a sum of tableaux in which:
\begin{itemize}
\item by skew-symmetry, in any horizontal component, the occurrences of $\alpha$ are replaced exactly by all the permutations
of the elements that  have been previously polarized;
\item the sign that is produced at this stage is the same as the sign produced by the action of the ``virtualizing part'' times
the product of the signs of the permutations of the elements in each horizontal component.
\end{itemize}
By reordering  each horizontal component - again by skew-symmetry - all the signs cancel out, and therefore, it just appear an
integer coefficient that is the product of the factorials of the lengths of the horizontal components. \qed

\

As usual in the theory of symmetric functions, given a shape
$$\lambda = (\lambda_1 \geq \cdots \geq \lambda_p), \ \lambda_1 \leq n,$$
we set
$$
\mathbf{H}_{\lambda}(n) = \mathbf{H}_{\lambda_1}(n)\mathbf{H}_{\lambda_2}(n) \cdots \mathbf{H}_{\lambda_p}(n).
$$
By convention, if $\lambda$ is the empty partition, we set $\mathbf{H}_{\emptyset}(n) = \mathbf{1} \in \boldsymbol{\zeta}(n).$

From Theorem \ref{Capelli generators}, one infers

\begin{corollary}

The set
$$
\big\{ \mathbf{H}_{\lambda}(n); \lambda_1 \leq n, \ |\lambda| \leq m \ \big\}
$$
is a linear basis of $\boldsymbol{\zeta}(n)^{(m)}.$
\end{corollary}

\subsubsection{The Koszul isomorphism and the Laplace expansion into column bitableaux}\label{subsubsection vertical}

In this subsection, we will consider the  Capelli determinantal generators $\mathbf{H}_k(n)$ from the point of view of the
{\it{ Bitableaux correspondence and Koszul map Theorems}} (\cite{Brini4-BR}, Thms. 1 and 2, see also \cite{BriUMI-BR}, \cite{Koszul-BR}).
We will show that the Capelli  determinants in $\mathbf{U}(gl(n))$ expand into ``{\it{column bitableaux}}'' in the same way as the determinants
of matrices with commutative entries expand into ordinary monomials, once the integers on the main diagonal are omitted. Furthermore,
column bitableaux are indeed the analogues - in $\mathbf{U}(gl(n))$ - of monomials in a polynomial algebra, since they are  invariant
with respect to permutations of their rows, as monomials are invariant with respect to permutations of their factors.

\begin{proposition}\label{column expansion}
Given a $k-$tuple $(i_1, i_2, \ldots, i_k)$,
$1 \leq i_1 < i_2 < \ldots < i_k \leq n$,
the element
$$
[i_k \cdots i_2 i_1|i_1 i_2 \cdots i_k] = (-1)^{k \choose 2} [i_1 i_2 \cdots i_k|i_1 i_2 \cdots i_k] =
$$
$$
= \mathbf{cdet}\left( \begin{array}{cccc}
 e_{{i_1},{i_1}}+(k-1) & e_{{i_1},{i_2}} & \ldots  & e_{{i_1},{i_k}} \\
 e_{{i_2},{i_1}} & e_{{i_2},{i_2}}+(k-2) & \ldots  & e_{{i_2},{i_k}}\\
 \vdots  &    \vdots                            & \vdots &  \\
e_{{i_k},{i_1}} & e_{{i_k},{i_2}} & \ldots & e_{{i_k},{i_k}}\\
 \end{array}
 \right) \in \mathbf{U}(gl(n))
$$
 expand into  column bitableaux as

\begin{equation}\label{colbit}
(-1)^{k \choose 2}
\sum_{\sigma} \ (-1)^{|\sigma|} \
\left[
\begin{array}{c}
i_1\\  i_2 \\ \vdots \\ i_k
\end{array}
\right| \left.
\begin{array}{c}
i_{\sigma(1)}\\ i_{\sigma(2)}\\ \vdots\\  i_{\sigma(k)}
\end{array}
\right],
\end{equation}
where the summands
$$
\left[
\begin{array}{c}
i_1\\  i_2 \\ \vdots \\ i_k
\end{array}
\right| \left.
\begin{array}{c}
i_{\sigma(1)}\\ i_{\sigma(2)}\\ \vdots\\  i_{\sigma(k)}
\end{array}
\right]
$$
are the column bitableaux
$$
\mathfrak{p} \big( e_{i_1 , \gamma_1}   e_{i_2, \gamma_2} \cdots e_{i_k, \gamma_k}
e_{\gamma_1, i_{\sigma(1)}}e_{\gamma_2, i_{\sigma(2)}} \cdots e_{\gamma_k, i_{\sigma(k)} } \big),
$$
the symbols $\gamma_1, \gamma_2, \cdots, \gamma_k$  are distinct positive virtual symbols in $A_0$,
and the summation is extended to all permutations $\sigma$ of
the set $\{1, 2, \ldots, k \}.$
\end{proposition}

In plain words, the expression of $[i_k \cdots i_2 i_1|i_1 i_2 \cdots i_k]$ in terms of column bitableaux in (\ref{colbit})
formally eliminates the
``{{\it{queues}''}} due to the integer summands that appear on the main diagonal. Furthermore,
this fact leads  to a {\it{third}}  combinatorial description of the eigenvalues
of the central Capelli generators on irreducible representations; this phenomenon implies, in turn, a noteworthy
generating function formula for {\it{permutation statistics}}.

\begin{example}

The Capelli element $\mathbf{H}_2(2) = [21|12]$ equals
$$
\mathbf{cdet}\left( \begin{array}{cc}
 e_{{1},{1}}+1 & e_{{1},{2}} \\
 e_{{2},{1}} & e_{{2},{2}}\\
 \end{array}
 \right)
=
[21|12]
=
- [12|12]
=
 -
\left[
\begin{array}{c}
1\\ 2
\end{array}
\right| \left.
\begin{array}{c}
1\\ 2
\end{array}
\right] +
\left[
\begin{array}{c}
1\\ 2
\end{array}
\right| \left.
\begin{array}{c}
2\\ 1
\end{array}
\right],
$$
where
$$
 - \left[
\begin{array}{c}
1\\ 2
\end{array}
\right| \left.
\begin{array}{c}
1\\ 2
\end{array}
\right] = e_{1,1}e_{2,2}
=
- \left[
\begin{array}{c}
2\\ 1
\end{array}
\right| \left.
\begin{array}{c}
2\\ 1
\end{array}
\right] = e_{2,2}e_{1,1} ,
$$
and
$$
 \left[
\begin{array}{c}
1\\ 2
\end{array}
\right| \left.
\begin{array}{c}
2\\ 1
\end{array}
\right] = -e_{1,2}e_{2,1} + e_{1,1} =
\left[
\begin{array}{c}
2\\ 1
\end{array}
\right| \left.
\begin{array}{c}
1\\ 2
\end{array}
\right] = -e_{2,1}e_{1,2} + e_{2,2}.
$$

\end{example}
\qed

The actions (on highest weight vectors) of the vertical bitableaux that appear in
Proposition \ref{column expansion} admit a remarkable combinatorial description.

Given a shape $\mu$, $\mu_1 \leq n$, a subset $1 \leq i_1 < i_2 < \cdots < i_k \leq n$ and a permutation
$\sigma$ of $\{ i_1, i_2,  \ldots, i_k \}$, consider the integer
$$
\Gamma_{\sigma}(\mu; i_1, i_2, \ldots, i_k) =
(\widetilde{\mu}_{i_1})^{h_{\sigma}(i_1)} (\widetilde{\mu}_{i_2})^{h_{\sigma}(i_2)} \cdots (\widetilde{\mu}_{i_k})^{h_{\sigma}(i_k)},
$$
where
$$
h_{\sigma}(j) = 1, \qquad j \in \underline{n}
$$
 if $j \in \underline{n}$ is a maximum element in a cycle of the cycle decomposition of the permutation
$\sigma$ of the subset $\{ i_1, i_2, \ldots, i_k \} \subset \{ 1, 2, \ldots, n \}$, and
$0$ otherwise.

\begin{example} Let $n \geq 12$, $k = 8$, $(i_1, i_2, \ldots, i_8) = (2, 4, 5, 6, 7, 9, 11, 12)$.
Consider the permutation $\sigma = (6 \ 2 \ 4)(9 \ 5)(11 \ 7)(12).$ Let $\mu$ be a shape, $\mu_1 \leq n.$
Then
$$
\Gamma_{\sigma}(\mu; 2, 4, 5, 6, 7, 9, 11, 12) = \widetilde{\mu}_6 \widetilde{\mu}_9 \widetilde{\mu}_{11} \widetilde{\mu}_{12}.
$$
\end{example}
\qed

\begin{proposition}\label{action vertical tableaux}

The action of the vertical bitabeau
$$
(-1)^{k \choose 2} (-1)^{|\sigma|}
\left[
\begin{array}{c}
i_1 \\ i_2 \\ \vdots \\ i_k
\end{array}
\right| \left.
\begin{array}{c}
i_{\sigma(1)}\\ i_{\sigma(2)}\\ \vdots\\  i_{\sigma(k)}
\end{array}
\right]
$$
 on the highest weight vector $v_{\widetilde{\mu}}$ of
weight $(\widetilde{\mu}_1, \widetilde{\mu}_2, \ldots, \widetilde{\mu}_n)$
equals
$$
 \Gamma_{\sigma}(\mu; i_1, i_2, \ldots, i_k)  \cdot v_{\widetilde{\mu}}.
$$
\end{proposition}
\begin{proof}
Recall that
\begin{align*}
&(-1)^{k \choose 2} (-1)^{|\sigma|}
\left[
\begin{array}{c}
i_1 \\ i_2 \\ \vdots \\ i_k
\end{array}
\right| \left.
\begin{array}{c}
i_{\sigma(1)}\\ i_{\sigma(2)}\\ \vdots\\  i_{\sigma(k)}
\end{array}
\right] = \\
& = (-1)^{k \choose 2} (-1)^{|\sigma|} \mathfrak{p} \big( e_{i_1 , \gamma_1}  \ e_{i_2, \gamma_2} \cdots e_{i_k, \gamma_k} \cdot
e_{\gamma_1, i_{\sigma(1)}} \ e_{\gamma_2, i_{\sigma(2)}} \cdots e_{\gamma_k, i_{\sigma(k)} } \big) \\
& = (-1)^{|\sigma|} \mathfrak{p} \big( e_{i_k, \gamma_k} \cdots    e_{i_2, \gamma_2} \  e_{i_1 , \gamma_1} \cdot
e_{\gamma_1, i_{\sigma(1)}} \ e_{\gamma_2, i_{\sigma(2)}} \cdots e_{\gamma_k, i_{\sigma(k)} } \big).
\end{align*}
Given a shape $\mu$, $\mu_1 \leq n$, let $v_{\widetilde{\mu}} = (D_\mu|D_\mu^P)$ the canonical highest weight vector
 of the Schur module $Schur_\mu(n)$.

We have to study the action of
$$
(-1)^{|\sigma|} e_{i_k, \gamma_k} \cdots    e_{i_2, \gamma_2} \  e_{i_1 , \gamma_1} \cdot
e_{\gamma_1, i_{\sigma(1)}} \ e_{\gamma_2, i_{\sigma(2)}} \cdots e_{\gamma_k, i_{\sigma(k)} }
$$
on the bitableau $(D_\mu|D_\mu^P)  \in {\mathbb C}[M_{n,d}].$

When the ``virtualizing part'' $e_{\gamma_1, i_{\sigma(1)}} \ e_{\gamma_2, i_{\sigma(2)}} \cdots e_{\gamma_k, i_{\sigma(k)} }$
acts, it distributes the virtual symbols $\gamma_1, \gamma_2, \ldots, \gamma_k$ into the columns
$i_1, i_2, \ldots, i_k$ of the tableau $D_\mu$ in all possible ways with some signs. These signs will be
canceled by the the action of the ``devirtualizing part'' $e_{i_k, \gamma_k} \cdots    e_{i_2, \gamma_2} \  e_{i_1 , \gamma_1}.$

Furthermore, by skew-symmetry, the configurations of the virtual symbols that remain non zero after the action
of the ``devirtualizing part'' are those in that virtual symbols associated to elements in $\{ i_1, i_2, \ldots, i_k \}$
that belong to the same cycle of the permutation $\sigma$  appear in the same row of the tableau $D_\mu$.
Finally, by reordering  the rows, the global sign $(-1)^{|\sigma|}$ of the permutation $\sigma$ cancels in each summand.

Therefore, the eigenvalue is nothing but that the number of ways of distributing the virtual symbols  so  that
virtual symbols associated to elements that belong to the same cycle appear  in the same row,
that is $\Gamma_{\sigma}(\mu; i_1, i_2, \ldots, i_k).$
\end{proof}

\begin{example}

Let $n \geq 9$ and $k = 7$, $(i_1, i_2, \ldots, i_7) = (1, 2, 3, 4, 5, 7, 9)$,
and consider the column bitableau
\begin{equation}\label{ex column}
(-1)^{7 \choose 2}\left[
\begin{array}{c}
1 \\ 2  \\ 3 \\  4 \\ 5  \\ 7 \\ 9
\end{array}
\right| \left.
\begin{array}{c}
\sigma(1)\\ \sigma(2)\\ \sigma(3) \\ \sigma(4) \\ \sigma(5) \\ \sigma(7) \\ \sigma(9)
\end{array}
\right]
=
- \left[
\begin{array}{c}
1 \\ 2  \\ 3 \\  4 \\ 5  \\ 7 \\ 9
\end{array}
\right| \left.
\begin{array}{c}
5\\ 1 \\ 9 \\ 7 \\ 2 \\ 4 \\ 3
\end{array}
\right].
\end{equation}
Since the  permutation $\sigma$
of the set $\{ 1, 2, 3, 4, 5, 7, 9 \}$ is even and  has cycle  decomposition
$$
\sigma = (521)(74)(93),
$$
then the column bitableau (\ref{ex column}) acts on the highest weight vectors $v_{\widetilde{\mu}}$ (of highest weight $\widetilde{\mu}$)
just multiplying it by the
integer
$$
\Gamma_{\sigma}(\mu; 1, 2, 3, 4, 5, 7, 9) = \widetilde{\mu}_5 \ \widetilde{\mu}_7 \ \widetilde{\mu}_9.
$$
\end{example}
\qed

By combining Propositions \ref{column expansion} and \ref{action vertical tableaux}, one infers a third combinatorial description of
the eigenvalues of the Capelli generators $\mathbf{H}_k(n)$ on the irreducible modules
$Schur_{\mu}(n)$ of highest weight $\widetilde{\mu}.$

\begin{proposition}\label{third eigenvalue}
The eigenvalue of the action of the central generator
$\mathbf{H}_k(n)$ on the irreducible module $Schur_{\mu}(n)$
equals
$$
\sum_{1 \leq i_1 < i_2 < \cdots < i_k \leq n} \ \sum_{\sigma} \ \Gamma_{\sigma}(\mu; i_1, i_2, \ldots, i_k),
$$
where the inner sum ranges over all permutations $\sigma$ of the set $\{ i_1, i_2, \ldots, i_k \}.$
\end{proposition}

By comparing Propositions \ref{Capelli eigenvalues}, \ref{horizontal strip}, \ref{third eigenvalue},
it follows

\begin{corollary}\label{triple eigenvalue}
Let $\mu$ be a shape, $\mu_1 \leq n$, $k \leq n$. Then
\begin{align*}
e^*_k(\widetilde{\mu}) & = \sum_{1 \leq i_1 < i_2 < \cdots < i_k \leq n} \ (\widetilde{\mu}_{i_1}  + k  - 1)
(\widetilde{\mu}_{i_2}  + k - 2) \cdots (\widetilde{\mu}_{i_k})
\\
& =  \sum_{\phantom{1 \leq i_1 < i_2 < \cdots < i_k \leq n}} \ hstrip_{\mu}(k)!
\\
&= \sum_{1 \leq i_1 < i_2 < \cdots < i_k \leq n} \ \sum_{\sigma} \ \Gamma_{\sigma}(\mu; i_1, i_2, \ldots, i_k).
\end{align*}
\end{corollary}

Passing to polynomials, Corollary \ref{triple eigenvalue} implies (in the special case $k = n$) a generating function  formula for
{\it{permutation statistics}}, which is in turn equivalent - via the {\it{Foata correspondence}} - to a result of Wilf
\cite{Wilf-BR}.

Keeping the notation for the coefficients $\Gamma_{\sigma}(\mu; i_1, i_2, \ldots, i_k)$, consider the generating polynomial
$$
\mathbf{W}_{n} = \sum_{\sigma} \ x_1^{h_{\sigma}(1)}x_2^{h_{\sigma}(2)} \cdots x_n^{h_{\sigma}(n)},
$$
where $\sigma$ ranges over all permutations of the set $\{1,2, \ldots, n \}$.

\begin{proposition} We have
$$
\mathbf{W}_{n} = \sum_{\sigma} \ x_1^{h_{\sigma}(1)}x_2^{h_{\sigma}(2)} \cdots x_n^{h_{\sigma}(n)} = e_{n}^{*}(x_1, x_2, \ldots, x_n),
$$
where
$$
e_{n}^{*}(x_1, x_2, \ldots, x_n) = (x_1 + n - 1)(x_2 + n - 2) \cdots x_n,
$$
the $n-$th elementary shifted  symmetric polynomial  in $n$ variables.
\end{proposition}

\begin{example}
Let $n = 3$. Then
$$
\mathbf{W}_{3} = x_1x_2x_3 + x_1x_3 + 2x_2x_3 + 2x_3,
$$
that equals
$$
e_{3}^{*}(x_1, x_2, x_3) = (x_1 + 2)(x_2 + 1)x_3.
$$
\end{example}
\qed

\begin{remark}
The approach to Capelli elements by expansion into vertical bitableaux naturally extends to the other classes of
central elements studied in the present work. For the  sake of brevity, we  leave this extension to the reader.
\end{remark}

\subsection{The virtual form of the permanental Nazarov/Umeda elements $\mathbf{I}_k(n)$}

In this section we provide the virtual  form of the set of the preimages in $\boldsymbol{\zeta}(n)$
- with respect to the Harish-Chandra isomorphism - of the sequence of {\textit{shifted complete  symmetric polynomials}}
$\mathbf{h}_k^{*}(x_1, x_2, \ldots, x_n),$ for every $k \in \mathbb{Z}^+,$ see \cite{OkOlsh-BR} and \cite{Molev1-BR}, Theorem $4.9$.

The central elements $\mathbf{I}_k(n),$ $k \in \mathbb{Z}^+,$ coincide (see \cite{BriUMI-BR}) with the ``permanental generators'' of
$\boldsymbol{\zeta}(n)$ originally discovered and studied - through the  machinery of {\it{Yangians}} - by Nazarov \cite{Nazarov-BR}
and later described by Umeda \cite{UmedaHirai-BR} as sums of column permanents in $\mathbf{U}(gl(n))$
(see also \cite{MolevNazarov-BR}, \cite{Nazarov2-BR}, and Turnbull  \cite{TURN-BR}).

\begin{definition}\label{perm gen-BR}
For every $k \in \mathbb{Z}^+$, set
\begin{align*}
\mathbf{I}_k(n) & =
 \sum_{(h_1, h_2, \ldots, h_n) } \ (h_1! h_2! \cdots h_n!)^{-1} \ [n^{h_n} \cdots 2^{h_2} 1^{h_1} | 1^{h_1} 2^{h_2} \cdots n^{h_n} ]^{*} = \\
 & =  \sum_{(h_1, h_2, \ldots, h_n)} \ (h_1! h_2! \cdots h_n!)^{-1} \ \mathfrak{p} \big( e_{n , \beta}^{h_n} \cdots e_{2, \beta}^{h_2} e_{1, \beta}^{h_1}
e_{\beta, 1}^{h_1}e_{\beta, 2}^{h_2} \cdots e_{\beta, n}^{h_n} \big), \\
\end{align*}
where $\beta \in A_1$ denotes {\it{any}} negative virtual symbol,
the sum is extended to all $n-$tuples $(h_1, h_2, \ldots, h_n)$ such that $h_1 + h_2 + \cdots + h_n = k.$
\end{definition}

\begin{remark} Clearly, as apparent  from the virtual presentation, each summand
$$
[n^{h_n} \cdots 2^{h_2} 1^{h_1} | 1^{h_1} 2^{h_2} \cdots n^{h_n} ]^{*}
$$
is symmetric both in the left and the right sequences. In ``nonvirtual form'', the summands
can be written as column permanent in the algebra $\mathbf{U}(gl(n))$ (see also Example \ref{column permanent}).
As already observed in the Introduction, the nonvirtual form of the elements $\mathbf{I}_k(n)$ is much harder to manage than their
virtual form (see for example the proof of the centrality, Proposition \ref{central virtual permanental} and
Theorem \ref{I centrality}).
\end{remark}

\begin{example}
\begin{multline*}
\mathbf{I}_3(3) = \frac{1}{3!}[111|111]^*+\frac{1}{2!}[211|112]^*+\frac{1}{2!}[311|113]^*+\frac{1}{2!}[221|122]^*+[321|123]^*+
\\
+\frac{1}{2!}[331|133]^*+\frac{1}{3!}[222|222]^*+\frac{1}{2!}[322|223]^*+\frac{1}{2!}[332|233]^*+\frac{1}{3!}[333|333]^*=
\\
=\frac{1}{3!}\mathbf{cper}
\left(
 \begin{array}{ccc}
 e_{1,1} - 2 & e_{1,1} - 1 &  e_{1,1} \\
 e_{1,1} - 2 & e_{1,1} - 1 &  e_{1,1}\\
 e_{1,1} - 2 & e_{1,1} - 1 &  e_{1,1}\\
 \end{array}
 \right)
+
\frac{1}{2!}\mathbf{cper}
\left(
 \begin{array}{ccc}
 e_{1,1} - 2 & e_{1,1} - 1 &  e_{1,2} \\
 e_{1,1} - 2 & e_{1,1} - 1 &  e_{1,2}\\
 e_{2,1}  & e_{2,1}  &  e_{2,2}\\
 \end{array}
 \right)+
\\
+\frac{1}{2!}\mathbf{cper}
\left(
 \begin{array}{ccc}
 e_{1,1} - 2 & e_{1,1} - 1 &  e_{1,3} \\
 e_{1,1} - 2 & e_{1,1} - 1 &  e_{1,3}\\
 e_{3,1}  & e_{3,1}  &  e_{3,3}\\
 \end{array}
 \right)
+
\frac{1}{2!}\mathbf{cper}
\left(
 \begin{array}{ccc}
 e_{1,1} - 2 & e_{1,2}  &  e_{1,2} \\
 e_{2,1}  & e_{2,2} - 1 &  e_{2,2}\\
 e_{2,1}  & e_{2,2} - 1 &  e_{2,2}\\
 \end{array}
 \right)+
\\
+\mathbf{cper}
\left(
 \begin{array}{ccc}
 e_{1,1} - 2 & e_{1,2} &  e_{1,3} \\
 e_{2,1} & e_{2,2} - 1 &  e_{2,3}\\
 e_{3,1} & e_{3,2} &  e_{3,3}\\
 \end{array}
 \right)
+
\frac{1}{2!}\mathbf{cper}
\left(
 \begin{array}{ccc}
 e_{1,1} - 2 & e_{1,3}  &  e_{1,3} \\
 e_{3,1}  & e_{3,3} - 1 &  e_{3,3}\\
 e_{3,1}  & e_{3,3} - 1 &  e_{3,3}\\
 \end{array}
 \right)+
\\
+\frac{1}{3!}\mathbf{cper}
\left(
 \begin{array}{ccc}
 e_{2,2} - 2 & e_{2,2} - 1 &  e_{2,2} \\
 e_{2,2} - 2 & e_{2,2} - 1 &  e_{2,2}\\
 e_{2,2} - 2 & e_{2,2} - 1 &  e_{2,2}\\
 \end{array}
 \right)
+
\frac{1}{2!}\mathbf{cper}
\left(
 \begin{array}{ccc}
 e_{2,2} - 2 & e_{2,2} - 1 &  e_{2,3} \\
 e_{2,2} - 2 & e_{2,2} - 1 &  e_{2,3}\\
 e_{3,2}  & e_{3,2}  &  e_{3,3}\\
 \end{array}
 \right)+
\\
+\frac{1}{2!}\mathbf{cper}
\left(
 \begin{array}{ccc}
 e_{2,2} - 2 & e_{2,3}  &  e_{2,3} \\
 e_{3,2}  & e_{3,3} - 1 &  e_{3,3}\\
 e_{3,2}  & e_{3,3} - 1 &  e_{3,3}\\
 \end{array}
 \right)
+
\frac{1}{3!}\mathbf{cper}
\left(
 \begin{array}{ccc}
 e_{3,3} - 2 & e_{3,3} - 1 &  e_{3,3} \\
 e_{3,3} - 2 & e_{3,3} - 1 &  e_{3,3}\\
 e_{3,3} - 2 & e_{3,3} - 1 &  e_{3,3}\\
 \end{array}
 \right).
\end{multline*}\qed

\end{example}

\begin{proposition}\label{central virtual permanental}
Since the adjoint representation acts by derivation, we have
$$
ad(e_{i j})\big( \ \sum_{(h_1, h_2, \ldots, h_n)} \ (h_1! h_2! \cdots h_n!)^{-1} \
 e_{n , \beta}^{h_n} \cdots e_{2, \beta}^{h_2} e_{1, \beta}^{h_1}
e_{\beta, 1}^{h_1}e_{\beta, 2}^{h_2} \cdots e_{\beta, n}^{h_n}  \big) = 0 ,
$$
for every  $e_{i j} \in gl(n).$
\end{proposition}
\begin{proof}
 In order to make  the notation lighter, in this proof we write
 $$
 \{n^{h_n} \cdots 2^{h_2} 1^{h_1}|1^{h_1}2^{h_2}  \cdots  n^{h_n} \}^*
 $$
 in place of
 $e_{n , \beta}^{h_n} \cdots e_{2, \beta}^{h_2} e_{1, \beta}^{h_1}
e_{\beta, 1}^{h_1}e_{\beta, 2}^{h_2} \cdots e_{\beta, n}^{h_n}$,
 that is
\begin{align*}
\begin{split}
\sum_{(h_1, h_2, \ldots, h_n)} \ (h_1! h_2! \cdots h_n!)^{-1} \
 e_{n , \beta}^{h_n} \cdots e_{2, \beta}^{h_2} e_{1, \beta}^{h_1}
e_{\beta, 1}^{h_1}e_{\beta, 2}^{h_2} \cdots e_{\beta, n}^{h_n} =
 \\ \null \hfill
= \sum_{(h_1, h_2, \ldots, h_n)} \ (h_1! h_2! \cdots h_n!)^{-1} \
  \{n^{h_n} \cdots 2^{h_2} 1^{h_1}|1^{h_1}2^{h_2}  \cdots  n^{h_n} \}^*.
\end{split}
\end{align*}

In the following, given an decreasing word $v = i_s \cdots i_2i_1$ on the set $1, 2, \ldots, n$, we denote by
$\overline{v} = i_1i_2 \cdots i_s$ its reverse.

Without loss of generality, let us consider the adjoint action of the element $e_{12}$, that is $i = 1, \ j = 2.$

Since
$$\{n^{h_n} \cdots 2^{h_2} 1^{h_1}|1^{h_1}2^{h_2}  \cdots  n^{h_n} \}^* = \{1^{h_1}2^{h_2}  \cdots  n^{h_n}|1^{h_1}2^{h_2}  \cdots  n^{h_n} \}^*,
$$
we write
 $$
 \sum_{(h_1, h_2, \ldots, h_n)} \ (h_1! h_2! \cdots h_n!)^{-1} \
  \{n^{h_n} \cdots 2^{h_2} 1^{h_1}|1^{h_1}2^{h_2}  \cdots  n^{h_n} \}^*
 $$
in the form
$$
\sum_w \ \sum_{h=0}^k \ \big( \sum_{j=0}^h \ (j!)^{-1} \
(h-j!)^{-1} \ (w!)^{-1} \ \{1^j \ 2^{h-j} \ w|1^j \ 2^{h-j} \ w \}^* \big),
$$
where the outer sum is extended over all words $w = 3^{h_3} \ 4^{h_4} \cdots n^{h_n}$, with
with $h_3 \ + \ h_4 \ + \cdots  + \ h_n = k - h$, and $w! = h_3! \   h_4! \  \cdots   \ h_n!.$

We claim that the adjoint action of $e_{12}$ on
$$
\sum_{j=0}^h \ (j!)^{-1} \
(h-j!)^{-1} (w!)^{-1} \{1^j \ 2^{h-j} \ w|1^j \ 2^{h-j} \ w \}^*
$$
equals zero. Indeed it produces
$$
\sum_{j=0}^{h-1} \ (j!)^{-1} \
((h-j-1)!)^{-1} \ (w!)^{-1} \ \{1^{j+1}  \ 2^{h-j-1} \ w|1^j \ 2^{h-j} \ w \}^*  -
$$
$$- \sum_{j=1}^h \ \ ((j-1)!)^{-1} \
((h-j)!)^{-1} \ (w!)^{-1}\ (w!)^{-1} \{1^j \ 2^{h-j} \ w|1^{j-1} \ 2^{h-j+1} \ w \}^*
$$
By performing in the second sum the substitution of variables $j = p+1$, we get
$$
\sum_{j=0}^{h-1} \ (j!)^{-1} \
((h-j-1)!)^{-1} \ (w!)^{-1} \ \{1^{j+1}  \ 2^{h-j-1} \ w|1^j \ 2^{h-j} \ w \}^*  -
$$
$$- \sum_{p=0}^{h-1} \ \ (p!)^{-1} \
((h-p-1)!)^{-1} \ (w!)^{-1}\  \{1^{p+1} \ 2^{h-p-1} \ w|1^{p} \ 2^{h-p} \ w \}^* = 0.
$$

\end{proof}

From Remark \ref{centrality-BR}, it follows

\begin{theorem}\label{I centrality}
The elements $\mathbf{I}_k(n)$ are  central in $\mathbf{U}(gl(n))$.
\end{theorem}

Clearly,
$$
\mathbf{I}_k(n) \in \boldsymbol{\zeta}(n)^{(m)},
$$
for every $m \geq k.$

The proofs of the following results are almost trivial, as a consequence of the definition of the elements $\mathbf{I}_k(n)$
in terms of their {\textit{virtual presentations}} (Definition \ref{perm gen-BR}).

\begin{theorem}\label{vertical strip}

We have:

\begin{enumerate}

\item We have
$$
\mathbf{I}_k(n)(v_{\widetilde{\mu}}) = h^*_k(\widetilde{\mu})\cdot v_{\widetilde{\mu}}, \quad h^*_k(\widetilde{\mu}) \in \mathbb{N}
$$
with
$$
h^*_k(\widetilde{\mu}) = \sum \ vstrip_{\mu}(k)!,
$$
where the sum is extended to all ``vertical strips'' \footnote{In this work, we use the expression
{\it{vertical strip}} in a generalized sense. To wit, a  vertical strip in a Ferrers diagram is a subset
of cells such that no two cells in the subset appear in the same row.}
of length $k$ in the Ferrers diagram of the partition $\mu$,
and the symbol $\ vstrip_{\mu}(k)!$ denotes the products of the factorials of the cardinality of each vertical
component of the vertical strip.

\item
Therefore:
$$
\mathbf{I}_k(n)(v_{\widetilde{\mu}}) = h^*_k(\widetilde{\mu})\cdot v_{\widetilde{\mu}}, \quad h^*_k(\widetilde{\mu}) \in \mathbb{Z},
$$
where
\begin{equation}\label{complete shift}
h^*_k(\widetilde{\mu})
= \sum_{1 \leq i_1 \leq i_2 < \cdots \leq i_k \leq n} \ (\widetilde{\mu}_{i_1}  - k  + 1)
(\widetilde{\mu}_{i_2}  - k + 2) \cdots (\widetilde{\mu}_{i_k})
\end{equation}

\item
If $\widetilde{\mu}_1 < k$, then
$$
\mathbf{I}_k(n)(v_{\widetilde{\mu}}) = 0.
$$

\end{enumerate}

\end{theorem}

\begin{example}\label{vertical example} Let $\lambda = (2,2,1)$.
Given a bitableau   $(S|D_\lambda^P)$ in the supermodule $Schur_\lambda(m_0|m_1+3)$, we will write $(S|$ in place  of $(S|D_\lambda^P),$ in order
to simplify the notation.

In particular, we write
$$
\left(
\begin{array}{ll}
1 \ 2  \\
1 \ 2 \\
1 \\
\end{array}
\right|
$$
in place of the $gl(3)-$highest weight vector of the $gl(3)-$irreducible submodule of $Schur_\lambda(3)$ of $Schur_\lambda(m_0|m_1+3)$:
$$
(D_{(2,2,1)}|D_{(2,2,1)}^P)
=
\left(
\begin{array}{ll}
1 \ 2 \\
1 \ 2 \\
1  \\
\end{array}
\right| \left.
\begin{array}{lll}
1 \ 2 \\
1 \ 2 \\
1  \\
\end{array}
\right),
$$
of weight $\widetilde{\lambda} = (3,2).$

The action of $\mathbf{I}_2(3)$ on $(D_\lambda|D_\lambda^P)$ is the same as the action of
$$
(2!)^{-1} \ e_{1, \beta} e_{1, \beta}e_{\beta, 1}e_{\beta, 1} +  (2!)^{-1} \ e_{2, \beta} e_{2, \beta}e_{\beta, 2}e_{\beta, 2} +
 (2!)^{-1} \ e_{3, \beta} e_{3, \beta}e_{\beta, 3}e_{\beta, 3} +
$$
$$
+ e_{2, \beta} e_{1, \beta}e_{\beta, 1}e_{\beta, 2} + e_{3, \beta} e_{1, \beta}e_{\beta, 1}e_{\beta, 3} +
 e_{3, \beta} e_{2, \beta}e_{\beta, 2}e_{\beta, 3}, \quad \beta \in A_1, \ |\beta| = 1;
$$
hence, we have to compute
\begin{multline*}
\Big( \Big. (2!)^{-1} \ D_{1, \beta} D_{1, \beta}D_{\beta, 1}D_{\beta, 1} +  (2!)^{-1} \ D_{2, \beta} D_{2, \beta}D_{\beta, 2}D_{\beta, 2} + \\
+ (2!)^{-1} \ D_{3, \beta} D_{3, \beta}D_{\beta, 3}D_{\beta, 3} +
 + D_{2, \beta} D_{1, \beta}D_{\beta, 1}D_{\beta, 2} + \\
 + D_{3, \beta} D_{1, \beta}D_{\beta, 1}D_{\beta, 3} +
 D_{3, \beta} D_{2, \beta}D_{\beta, 2}D_{\beta, 3} \Big. \Big) \
\left(
\begin{array}{ll}
1 \ 2 \\
1 \ 2 \\
1 \\
\end{array}
\right|.
\end{multline*}
By considering the action of the ``virtualizing part'' of each summand, we have
\begin{align*}
(2!)^{-1} \ D_{\beta, 1}D_{\beta, 1} \ \left(
\begin{array}{ll}
1 \ 2  \\
1 \ 2 \\
1 \\
\end{array}
\right| &=
 \left(
\begin{array}{ll}
\beta \ 2  \\
\beta \ 2 \\
1 \\
\end{array}
\right|
+
 \left(
\begin{array}{ll}
\beta \ 2  \\
1 \ 2 \\
\beta \\
\end{array}
\right|
+
 \left(
\begin{array}{ll}
1 \ 2  \\
\beta \ 2 \\
\beta \\
\end{array}
\right|,
\\
(2!)^{-1} \ D_{\beta, 2}D_{\beta, 2 } \ \left(
\begin{array}{ll}
1 \ 2  \\
1 \ 2 \\
1 \\
\end{array}
\right| &=
 \left(
\begin{array}{ll}
1 \ \beta \\
1 \ \beta \\
1 \\
\end{array}
\right|,
\\
(2!)^{-1} \ D_{\beta, 3}D_{\beta, 3} \ \left(
\begin{array}{ll}
1 \ 2  \\
1 \ 2 \\
1 \\
\end{array}
\right| &=
0,
\\
D_{\beta, 1}D_{\beta, 2} \ \left(
\begin{array}{ll}
1 \ 2  \\
1 \ 2 \\
1 \\
\end{array}
\right| &=
 \left(
\begin{array}{ll}
\beta \ 2  \\
1 \ \beta \\
1 \\
\end{array}
\right|
+
 \left(
\begin{array}{ll}
1 \ \beta  \\
\beta \ 2 \\
1 \\
\end{array}
\right|
+
 \left(
\begin{array}{ll}
1 \ \beta  \\
1 \ 2 \\
\beta \\
\end{array}
\right|
+
 \left(
\begin{array}{ll}
1 \ 2  \\
1 \ \beta \\
\beta \\
\end{array}
\right|,
\\
D_{\alpha, 1}D_{\alpha, 3} \ \left(
\begin{array}{ll}
1 \ 2  \\
1 \ 2 \\
1 \\
\end{array}
\right| &=
 0,
\\
D_{\alpha, 2}D_{\alpha, 3} \ \left(
\begin{array}{lll}
1 \ 2  \\
1 \ 2 \\
1 \\
\end{array}
\right| &=
0.
\end{align*}
Notice that the two occurrences of $\beta$ distribute in all vertical strips of length $2$ in the Ferrers diagram
of the partition $\lambda = (2,2,1).$

By considering the action of the ``devirtualizing part'' of each summand, we have
\begin{multline*}
D_{1, \beta} D_{1, \beta} \Big(  \left(
\begin{array}{ll}
\beta \ 2  \\
\beta \ 2 \\
1 \\
\end{array}
\right|
+
 \left(
\begin{array}{ll}
\beta \ 2  \\
1 \ 2 \\
\beta \\
\end{array}
\right|
+
 \left(
\begin{array}{ll}
1 \ 2  \\
\beta \ 2 \\
\beta \\
\end{array}
\right|  \Big)
= 6 \
\left(
\begin{array}{lll}
1 \ 2 \\
1 \ 2 \\
1 \\
\end{array}
\right|,
\\
\end{multline*}

\begin{multline*}
D_{2, \beta} D_{2, \beta} \Big(  \left(
\begin{array}{ll}
1 \ \beta \\
1 \ \beta \\
1 \\
\end{array}
\right|
 = 2 \
\left(
\begin{array}{ll}
1 \ 2 \\
1 \ 2 \\
1 \\
\end{array}
\right|,
\\
\end{multline*}

\begin{multline*}
D_{2, \beta} D_{1, \beta} \Big(  \left(
\begin{array}{ll}
\beta \ 2  \\
1 \ \beta \\
1 \\
\end{array}
\right|
+
 \left(
\begin{array}{ll}
1 \ \beta  \\
\beta \ 2 \\
1 \\
\end{array}
\right|
+
 \left(
\begin{array}{ll}
1 \ \beta  \\
1 \ 2 \\
\beta \\
\end{array}
\right|
+
 \left(
\begin{array}{ll}
1 \ 2  \\
1 \ \beta \\
\beta \\
\end{array}
\right| \Big)
= 4 \
\left(
\begin{array}{lll}
1 \ 2 \\
1 \ 2 \\
1 \\
\end{array}
\right|.
\\
\end{multline*}
Therefore,
\begin{multline*}
\Big( (2!)^{-1} \ D_{1, \beta} D_{1, \beta}D_{\beta, 1}D_{\beta, 1} +  (2!)^{-1} \ D_{2, \beta} D_{2, \beta}D_{\beta, 2}D_{\beta, 2} + \\
+ (2!)^{-1} \ D_{3, \beta} D_{3, \beta}D_{\beta, 3}D_{\beta, 3} + D_{2, \beta} D_{1, \beta}D_{\beta, 1}D_{\beta, 2} + \\
+ D_{3, \beta} D_{1, \beta}D_{\beta, 1}D_{\beta, 3} +
 D_{3, \beta} D_{2, \beta}D_{\beta, 2}D_{\beta, 3}  \Big) \
\left(
\begin{array}{ll}
1 \ 2  \\
1 \ 2 \\
1 \\
\end{array}
\right| =
 12 \ \left(
\begin{array}{ll}
1 \ 2 \\
1 \ 2 \\
1 \\
\end{array}
\right|.
\end{multline*}

Notice that, since $\widetilde{\lambda} = (3,2)$, according to Theorem \ref{vertical strip}, eq. \ref{complete shift}, we have
$$
h^*_2((3,2))
= \sum_{1 \leq i_1 \leq i_2 \leq 3}
(\widetilde{\lambda}_{i_1}  + 2 - 1) (\widetilde{\lambda}_{i_2}) = (3-2+1)2+(3-2+1)3+(2-2+1)2 = 12.
$$

Moreover, by comparing with Example \ref{horizontal example}, since $\lambda = (2,2,1) = \widetilde{\mu}$,
$\mu = (3,2)$, we have
\begin{equation}\label{example duality}
e^*_2((2,2,1)) = 12 = h^*_2((3,2)).
\end{equation}
\end{example} \qed

\begin{proposition}\label{Umeda free}

The set
$$
\big\{ \mathbf{I}_1(n), \mathbf{I}_2(n), \ldots, \mathbf{I}_n(n) \big\}
$$
is a set of algebraically independent generators of the center $\boldsymbol{\zeta}(n)$ of
$\mathbf{U}(gl(n)).$
\end{proposition}

\

Given a shape
$$\lambda = (\lambda_1 \geq \cdots \geq \lambda_p), \ \lambda_1 \leq n,$$
we set
$$
\mathbf{I}_{\lambda}(n) = \mathbf{I}_{\lambda_1}(n)\mathbf{I}_{\lambda_2}(n) \cdots \mathbf{I}_{\lambda_p}(n).
$$
By convention, if $\lambda$ is the empty partition, we set $\mathbf{I}_{\emptyset}(n) = \mathbf{1} \in \boldsymbol{\zeta}(n).$

\subsection{The Shaped Capelli determinantal elements $\mathbf{K}_\lambda(n)$}

Let $\lambda = (\lambda_1, \ldots, \lambda_p),$  $\lambda_1 \leq n$.
We notice that any element
$$
e_{S_1,C_{\lambda}^{*}} \cdot e_{C_{\lambda}^{*},S_2} \in Virt(m_0+m_1,n),
$$
where $S_1, S_2$ are tableaux  on the proper alphabet $L = \{x_1, \ldots, x_n \}$ of shape $\lambda$,
$\lambda_1 \leq n$, $m_0 \geq \widetilde{\lambda}_1$, is
{\textit{skew-symmetric}}  in the rows of $S_1$ and  $S_2$, respectively.

\begin{definition}\label{shaped Capelli}

We set
$$
\mathbf{K}_\lambda(n) =
\sum_S \ \mathfrak{p} \big( e_{S,C_{\lambda}^*} \cdot e_{C_{\lambda}^*,S} \big) = \sum_S \ SC_{\lambda}^* \ C_{\lambda}^*S
\in {\mathbf{U}}(gl(n)),
$$
where the sum is extended to all  row-increasing tableaux $S$ on the proper alphabet $L = \{x_1, \ldots, x_n \}$.
\end{definition}

By convention, if $\lambda$ is the empty partition, we set $\mathbf{K}_{\emptyset}(n) = \mathbf{1} \in \boldsymbol{\zeta}(n).$

\begin{proposition}\label{K centr}
Since the adjoint representation acts by derivation, we have
$$
ad(e_{i j})\big(  \sum_S \   e_{S,C_{\lambda}^*} \cdot e_{C_{\lambda}^*,S}   \big) = 0,
$$
for every  $e_{i j} \in gl(n).$
\end{proposition}
Since the elements
$$
e_{S_1,C_{\lambda}^{*}} \cdot e_{C_{\lambda}^{*},S_2} \in Virt(m_0+m_1,n),
$$
are {\textit{skew-symmetric}}  in the rows of $S_1$ and  $S_2$,
the proof is essentially the same as the proof of Proposition \ref{central virtual determinantal}.

From Remark \ref{centrality-BR}, it follows

\begin{theorem}
The elements $\mathbf{K}_\lambda(n)$ are  central in $\mathbf{U}(gl(n))$.
\end{theorem}

\begin{example} Consider the action of $ad(e_{21})$ on $\mathbf{K}_{(2,1)}(2)$.
We have
\begin{align*}
ad(e_{21}) \big( \mathbf{K}_{(2,1)}(2) \big) =
&
\mathfrak{p}  \Big( +  e_{{\fontsize{6} {6} \selectfont \substack{
2  2  \ \alpha_1  \alpha_1  \\
1   \phantom{2}\ \alpha_2   \phantom{\alpha_2}\\
}}}
\times
e_{{\fontsize{6} {6} \selectfont \substack{
\alpha_1  \alpha_1  \ 1  2  \\
\alpha_2   \phantom{\alpha_1} \ 1   \phantom{2}\\
}}}
\\
& \phantom{\mathfrak{p} \Big(}
+ e_{{\fontsize{6} {6} \selectfont \substack{
1  2  \ \alpha_1  \alpha_1  \\
2   \phantom{2}\ \alpha_2   \phantom{\alpha_2}\\
}}}
\times
e_{{\fontsize{6} {6} \selectfont \substack{
\alpha_1  \alpha_1  \ 1  2  \\
\alpha_2   \phantom{\alpha_1} \ 1   \phantom{2}\\
}}}
\\
& \phantom{\mathfrak{p} \Big(}
- e_{{\fontsize{6} {6} \selectfont \substack{
1  2  \ \alpha_1  \alpha_1  \\
1   \phantom{2}\ \alpha_2   \phantom{\alpha_2}\\
}}}
\times
e_{{\fontsize{6} {6} \selectfont \substack{
\alpha_1  \alpha_1  \ 1  1  \\
\alpha_2   \phantom{\alpha_1} \ 1   \phantom{2}\\
}}}
\Big)
\\
+ &
\mathfrak{p} \Big(  + e_{{\fontsize{6} {6} \selectfont \substack{
2  2  \ \alpha_1  \alpha_1  \\
2   \phantom{2}\ \alpha_2   \phantom{\alpha_2}\\
}}}
\times
e_{{\fontsize{6} {6} \selectfont \substack{
\alpha_1  \alpha_1  \ 1  2  \\
\alpha_2   \phantom{\alpha_1} \ 2   \phantom{2}\\
}}}
\\
& \phantom{\mathfrak{p} \Big(} -
e_{{\fontsize{6} {6} \selectfont \substack{
1  2  \ \alpha_1  \alpha_1  \\
2   \phantom{2}\ \alpha_2   \phantom{\alpha_2}\\
}}}
\times
e_{{\fontsize{6} {6} \selectfont \substack{
\alpha_1  \alpha_1  \ 1  1  \\
\alpha_2   \phantom{\alpha_1} \ 2   \phantom{2}\\
}}}
\\
& \phantom{\mathfrak{p} \Big(} -
e_{{\fontsize{6} {6} \selectfont \substack{
1  2  \ \alpha_1  \alpha_1  \\
2   \phantom{2}\ \alpha_2   \phantom{\alpha_2}\\
}}}
\times
e_{{\fontsize{6} {6} \selectfont \substack{
\alpha_1  \alpha_1  \ 1  2  \\
\alpha_2   \phantom{\alpha_1} \ 1   \phantom{2}\\
}}}
\Big) = 0.
\end{align*}

Consider the action of $ad(e_{12})$ on $\mathbf{K}_{(2,1)}(2)$.
We have
\begin{align*}
ad(e_{12}) \big( \mathbf{K}_{(2,1)}(2) \big) =
&
\mathfrak{p}  \Big( +  e_{{\fontsize{6} {6} \selectfont \substack{
1 1 \ \alpha_1  \alpha_1  \\
1   \phantom{2}\ \alpha_2   \phantom{\alpha_2}\\
}}}
\times
e_{{\fontsize{6} {6} \selectfont \substack{
\alpha_1  \alpha_1  \ 1  2  \\
\alpha_2   \phantom{\alpha_1} \ 1   \phantom{2}\\
}}}
\\
& \phantom{\mathfrak{p} \Big(}
- e_{{\fontsize{6} {6} \selectfont \substack{
1  2  \ \alpha_1  \alpha_1  \\
1   \phantom{2}\ \alpha_2   \phantom{\alpha_2}\\
}}}
\times
e_{{\fontsize{6} {6} \selectfont \substack{
\alpha_1  \alpha_1  \ 2  2  \\
\alpha_2   \phantom{\alpha_1} \ 1   \phantom{2}\\
}}}
\\
& \phantom{\mathfrak{p} \Big(}
- e_{{\fontsize{6} {6} \selectfont \substack{
1  2  \ \alpha_1  \alpha_1  \\
1   \phantom{2}\ \alpha_2   \phantom{\alpha_2}\\
}}}
\times
e_{{\fontsize{6} {6} \selectfont \substack{
\alpha_1  \alpha_1  \ 1  2  \\
\alpha_2   \phantom{\alpha_1} \ 2   \phantom{2}\\
}}}
\Big)
\\
+ &
\mathfrak{p} \Big(  + e_{{\fontsize{6} {6} \selectfont \substack{
1 1  \ \alpha_1  \alpha_1  \\
2   \phantom{2}\ \alpha_2   \phantom{\alpha_2}\\
}}}
\times
e_{{\fontsize{6} {6} \selectfont \substack{
\alpha_1  \alpha_1  \ 1  2  \\
\alpha_2   \phantom{\alpha_1} \ 2   \phantom{2}\\
}}}
\\
& \phantom{\mathfrak{p} \Big(} +
e_{{\fontsize{6} {6} \selectfont \substack{
1  2  \ \alpha_1  \alpha_1  \\
1   \phantom{2}\ \alpha_2   \phantom{\alpha_2}\\
}}}
\times
e_{{\fontsize{6} {6} \selectfont \substack{
\alpha_1  \alpha_1  \ 1  2  \\
\alpha_2   \phantom{\alpha_1} \ 2   \phantom{2}\\
}}}
\\
& \phantom{\mathfrak{p} \Big(} -
e_{{\fontsize{6} {6} \selectfont \substack{
1  2  \ \alpha_1  \alpha_1  \\
2   \phantom{2}\ \alpha_2   \phantom{\alpha_2}\\
}}}
\times
e_{{\fontsize{6} {6} \selectfont \substack{
\alpha_1  \alpha_1  \ 2  2  \\
\alpha_2   \phantom{\alpha_1} \ 2   \phantom{2}\\
}}}
\Big) = 0.
\end{align*}

Consider the action of $ad(e_{11})$ on $\mathbf{K}_{(2,1)}(2)$.
We have
\begin{align*}
ad(e_{11}) \big( \mathbf{K}_{(2,1)}(2) \big) =
&
\mathfrak{p}  \Big( + 2  e_{{\fontsize{6} {6} \selectfont \substack{
1 2 \ \alpha_1  \alpha_1  \\
1   \phantom{2}\ \alpha_2   \phantom{\alpha_2}\\
}}}
\times
e_{{\fontsize{6} {6} \selectfont \substack{
\alpha_1  \alpha_1  \ 1  2  \\
\alpha_2   \phantom{\alpha_1} \ 1   \phantom{2}\\
}}}
\\
& \phantom{
\mathfrak{p} \Big(}
- 2 e_{{\fontsize{6} {6} \selectfont \substack{
1  2  \ \alpha_1  \alpha_1  \\
1   \phantom{2}\ \alpha_2   \phantom{\alpha_2}\\
}}}
\times
e_{{\fontsize{6} {6} \selectfont \substack{
\alpha_1  \alpha_1  \ 1  2  \\
\alpha_2   \phantom{\alpha_1} \ 1   \phantom{2}\\
}}}
\Big)
\\
+ &
\mathfrak{p} \Big(  + e_{{\fontsize{6} {6} \selectfont \substack{
1 2  \ \alpha_1  \alpha_1  \\
2   \phantom{2}\ \alpha_2   \phantom{\alpha_2}\\
}}}
\times
e_{{\fontsize{6} {6} \selectfont \substack{
\alpha_1  \alpha_1  \ 1  2  \\
\alpha_2   \phantom{\alpha_1} \ 2   \phantom{2}\\
}}}
\\
& \phantom{\mathfrak{p} \Big(} -
e_{{\fontsize{6} {6} \selectfont \substack{
1  2  \ \alpha_1  \alpha_1  \\
2   \phantom{2}\ \alpha_2   \phantom{\alpha_2}\\
}}}
\times
e_{{\fontsize{6} {6} \selectfont \substack{
\alpha_1  \alpha_1  \ 1  2  \\
\alpha_2   \phantom{\alpha_1} \ 2   \phantom{2}\\
}}}
\Big) = 0.
\end{align*}

Consider the action of $ad(e_{22})$ on $\mathbf{K}_{(2,1)}(2)$.
We have
\begin{align*}
ad(e_{22}) \big( \mathbf{K}_{(2,1)}(2) \big) =
&
\mathfrak{p}  \Big( +   e_{{\fontsize{6} {6} \selectfont \substack{
1 2 \ \alpha_1  \alpha_1  \\
1   \phantom{2}\ \alpha_2   \phantom{\alpha_2}\\
}}}
\times
e_{{\fontsize{6} {6} \selectfont \substack{
\alpha_1  \alpha_1  \ 1  2  \\
\alpha_2   \phantom{\alpha_1} \ 1   \phantom{2}\\
}}}
\\
& \phantom{
\mathfrak{p} \Big(}
-  e_{{\fontsize{6} {6} \selectfont \substack{
1  2  \ \alpha_1  \alpha_1  \\
1   \phantom{2}\ \alpha_2   \phantom{\alpha_2}\\
}}}
\times
e_{{\fontsize{6} {6} \selectfont \substack{
\alpha_1  \alpha_1  \ 1  2  \\
\alpha_2   \phantom{\alpha_1} \ 1   \phantom{2}\\
}}}
\Big)
\\
+ &
\mathfrak{p} \Big(  + 2 e_{{\fontsize{6} {6} \selectfont \substack{
1 2  \ \alpha_1  \alpha_1  \\
2   \phantom{2}\ \alpha_2   \phantom{\alpha_2}\\
}}}
\times
e_{{\fontsize{6} {6} \selectfont \substack{
\alpha_1  \alpha_1  \ 1  2  \\
\alpha_2   \phantom{\alpha_1} \ 2   \phantom{2}\\
}}}
\\
& \phantom{\mathfrak{p} \Big(} - 2
e_{{\fontsize{6} {6} \selectfont \substack{
1  2  \ \alpha_1  \alpha_1  \\
2   \phantom{2}\ \alpha_2   \phantom{\alpha_2}\\
}}}
\times
e_{{\fontsize{6} {6} \selectfont \substack{
\alpha_1  \alpha_1  \ 1  2  \\
\alpha_2   \phantom{\alpha_1} \ 2   \phantom{2}\\
}}}
\Big) = 0.
\end{align*}
\end{example}\qed

Clearly,
$$
\mathbf{K}_\lambda(n) \in \boldsymbol{\zeta}(n)^{(m)},
$$
for every $m \geq |\lambda|.$

The elements $\mathbf{K}_\lambda(n)$ will be called {\it{shaped Capelli determinantal elements}}.

\begin{theorem}(Triangularity of the actions on highest weight vectors)\label{K Triangularity}
Let $\vartriangleleft$ denote the  dominance order on partitions.  We have:
\begin{align}
&\textrm{If} \ |\mu| < |\lambda|, \  \textrm{then}  &   \mathbf{K}_{\lambda}(n)(v_{\widetilde{\mu}}) &= 0,\label{tr 1}
\\
&\textrm{If} \ |\mu| = |\lambda| \ \textrm{and} \ \mu \ntrianglerighteq \lambda, \ \textrm{then} &
\mathbf{K}_{\lambda} (n) (v_{\widetilde{\mu}}) &= 0,\label{tr 2}
\\
& \textrm{If} \ \lambda = \mu, \   \textrm{then} &
\mathbf{K}_{\lambda}(n)(v_{\widetilde{\lambda}}) &= (-1)^{{k} \choose   {2}}  H(\lambda)\cdot v_{\widetilde{\lambda}},\label{tr 3}
\end{align}
where $H(\lambda) = H(\widetilde{\lambda})$ denotes the hook coefficient of the shape $\lambda \vdash k.$
\end{theorem}
Assertions (\ref{tr 1}) and (\ref{tr 2})   follow from the Proposition \ref{Vanishing Lemma}, eq. (\ref{due}) and eq. (\ref{duebis}),
respectively;  assertion (\ref{tr 3}) follows from  Proposition  \ref{hook lemma}.

Therefore, the central elements $\mathbf{K}_{\lambda}(n)$ are linearly independent, and the next result
follows at once.

\begin{proposition}\label{K basis}

For every $m \in \mathbb{Z}^+$, the set
$$
\big\{ \mathbf{K}_{\lambda}(n); \lambda_1 \leq n, \ |\lambda| \leq m \ \big\}
$$
is a linear basis of $\boldsymbol{\zeta}(n)^{(m)}.$

The set
$$
\big\{ \mathbf{K}_{\lambda}(n); \lambda_1 \leq n \ \big\}
$$
is a linear basis of the center $\boldsymbol{\zeta}(n).$

\end{proposition}

The elements $\mathbf{K}_{\lambda}(n)$ are related to the Capelli generators $\mathbf{H}_k(n)$ by a {\it{Straightenig Law}}.

\begin{proposition}

\begin{equation}\label{sfilamenti 1}
\mathbf{K}_{\lambda}(n) = \pm \mathbf{H}_{\lambda}(n) + \sum \ c_{\lambda,\mu} \mathbf{K}_{\mu}(n),
\end{equation}
where $c_{\lambda,\mu} = 0$ unless $|\mu| \nless |\lambda|$.
\end{proposition}
\begin{proof} From the virtual presentations of $\mathbf{K}_{\lambda}(n)$ and $\mathbf{H}_{\lambda}(n)$,
it is easy to see - just applying the commutator relations in the algebra ${\mathbf{U}}(gl(m_0|m_1+n))$ -
that the difference between the l.h.s. and the first summand of the r.h.s. in eq. (\ref{sfilamenti 1}) is a central element in
$\boldsymbol{\zeta}(n)$ that belongs to a filtration element $\boldsymbol{\zeta}(n)^{(m)}$ with $m < |\lambda|.$
\end{proof}

\subsection{The Shaped Nazarov/Umeda permanental elements $\mathbf{J}_\lambda(n)$}

Given a column-nondecreasing tableaux $S$  on
the proper alphabet $L = \{x_1, \ldots, x_n \},$ let $o_S$ denote the product of the factorials of the
repetitions (of the symbols)  column by column. For example, let

$$
S = \left(
\begin{array}{lllll}
1 \quad 1 \quad 4 \\
1 \quad 1 \quad 4 \\
2 \quad 1 \quad 4 \\
3 \quad 2 \quad 5 \\
3 \quad 3
\end{array}
\right), \quad sh(S) = (3, 3, 3, 3, 2);
$$
then $o_S = 2! \ 2! \ 3! \ 3!.$

Let $\lambda = (\lambda_1, \ldots, \lambda_p),$  $\lambda \vdash d, \quad d \in \mathbb{Z}^{+}$.

\begin{definition}\label{shaped Nazarov/Umeda}

We set
$$
\mathbf{J}_\lambda(n) =
\sum_S \ (o_S)^{-1} \ \mathfrak{p} \big( e_{S,D^*_{\widetilde{\lambda}}} \cdot e_{D^*_{\widetilde{\lambda}},S} \big) =
\sum_S \  (o_S)^{-1} \ SD^*_{\widetilde{\lambda}} \ D^*_{\widetilde{\lambda}}S
\in \mathbf{U}(gl(n)),
$$
where the sum is extended to all  column-nondecreasing tableaux $S$ (of shape $\widetilde{\lambda}$) on
the proper alphabet $L = \{x_1, \ldots, x_n \}$.
\end{definition}

By convention, if $\lambda$ is the empty partition, we set $\mathbf{J}_{\emptyset}(n) = \mathbf{1} \in \boldsymbol{\zeta}(n).$

\begin{claim}

We have

$$
e_{S,D^*_{\widetilde{\lambda}}} = e_{\widetilde{S},\widetilde{D^*_{\widetilde{\lambda}}}}
$$
and
$$
e_{D^*_{\widetilde{\lambda}},S} =  e_{\widetilde{D^*_{\widetilde{\lambda}}},\widetilde{S}}.
$$
Therefore
\begin{align}
\mathbf{J}_\lambda(n) &=
\sum_T \ (o_{\widetilde{T}})^{-1} \ \mathfrak{p} \big( e_{T,\widetilde{D^*_{\widetilde{\lambda}}}} \cdot  e_{\widetilde{D^*_{\widetilde{\lambda}}},T} \big)
\label{J  definition}
\\
&= \sum_T \ (o_{\widetilde{T}})^{-1} \ T\widetilde{D^*_{\widetilde{\lambda}}} \ \widetilde{D^*_{\widetilde{\lambda}}}T
\in {\mathbf{U}}(gl(n)),
\end{align}
where the sum is extended to all  row-nondecreasing tableaux $T$ (of shape $\lambda$) on the proper alphabet $L = \{x_1, \ldots, x_n \}$.

In particular, if $\lambda = (k)$, the row shape of length $k$, then $\mathbf{J}_\lambda(n) = \mathbf{I}_k(n).$

\end{claim}

\begin{proposition}
Since the adjoint representation acts by derivation, we have
$$
ad(e_{i j})\big( \sum_S \ (o_S)^{-1} \  e_{S,D^*_{\widetilde{\lambda}}} \cdot e_{D^*_{\widetilde{\lambda}},S}   \big) = 0 ,
$$
for every  $e_{i j} \in gl(n).$
\end{proposition}
The proof is essentially the same as the proof of Proposition \ref{central virtual permanental}.

From Remark \ref{centrality-BR}, it follows

\begin{theorem}
The elements $\mathbf{J}_\lambda(n)$ are  central in $\mathbf{U}(gl(n))$.
\end{theorem}

Hence, the elements $\mathbf{J}_\lambda(n)$ are  central in $\mathbf{U}(gl(n))$, by Remark \ref{centrality-BR}.

Clearly,
$$
\mathbf{J}_\lambda(n) \in \boldsymbol{\zeta}(n)^{(m)},
$$
for every $m \geq |\lambda|.$

The elements $J_\lambda(n)$ will be called {\it{shaped Nazarov/Umeda permanental elements}}.

\begin{proposition}(Triangularity of the actions on highest weight vectors)\label{J Triangularity}
We have:
\begin{align}
&\textrm{If} \ |\mu| < |\lambda|, \  \textrm{then}  &   \mathbf{J}_{\lambda}(n)(v_{\widetilde{\mu}}) &= 0, \label{tr-unobis}
\\
&\textrm{If} \ |\mu| = |\lambda| \ \textrm{and} \ \widetilde{\mu} \ntrianglerighteq \lambda, \ \textrm{then} &
\mathbf{J}_{\lambda} (n) (v_{\widetilde{\mu}}) &= 0. \label{tr-duebis}
\end{align}
\end{proposition}
Assertions (\ref{tr-unobis}) and (\ref{tr-duebis})   follow from the Proposition \ref{Vanishing Lemma}, eq. (\ref{uno}) and eq. (\ref{unobis}),
respectively.

\begin{proposition}\label{J filtration} For every shape $\lambda$, we have:
\begin{equation}\label{J sfilamenti}
\mathbf{J}_{\lambda}(n) = \mathbf{I}_{\lambda}(n) + \mathbf{F}_{\lambda}(n),
\end{equation}
where
\begin{align*}
&\mathbf{I}_{\lambda}(n) \notin \boldsymbol{\zeta}(n)^{(m)}, \quad if \ m < |\lambda|,
\\
&\mathbf{F}_{\lambda}(n) \in \boldsymbol{\zeta}(n)^{(p)}, \quad for \ some \ p < |\lambda|.
\end{align*}
\end{proposition}
\begin{proof} From the virtual presentations of $\mathbf{J}_{\lambda}(n)$ and $\mathbf{I}_{\lambda}(n)$,
it is easy to see - just applying the commutator relations in the algebra ${\mathbf{U}}(gl(m_0|m_1+n))$ -
that the difference between the l.h.s. and the first summand of the r.h.s. in eq. (\ref{J sfilamenti}) is a central element in
$\boldsymbol{\zeta}(n)$ that belongs to a filtration element $\boldsymbol{\zeta}(n)^{(m)}$ with $m < |\lambda|.$
\end{proof}

\begin{example} In order to simplify the notation, we will write any summand
$e_{T,\widetilde{D^*_{\widetilde{\lambda}}}} \cdot  e_{\widetilde{D^*_{\widetilde{\lambda}}},T}$ in eq. (\ref{J  definition}) as
$[T|T]^*.$ Let $\lambda = (2,2)$, and $n = 2.$ In this notation, the central element $\mathbf{J}_{(2,2)}(2)$ can be written in the following way:
\begin{align*}
\mathbf{J}_{(2,2)}(2)=
\frac{1}{2!2!}
&\left[ \left.
\begin{array}{ll}
1 \ 1 \\
1 \ 1 \\
\end{array}
\right|
\begin{array}{ll}
1 \ 1 \\
1 \ 1 \\
\end{array} \right]^*&
+
\frac{1}{2!}
&\left[ \left.
\begin{array}{ll}
1 \ 1 \\
1 \ 2 \\
\end{array}
\right|
\begin{array}{ll}
1 \ 1 \\
1 \ 2 \\
\end{array} \right]^*&
+
\frac{1}{2!2!}
&\left[ \left.
\begin{array}{ll}
1 \ 1 \\
2 \ 2 \\
\end{array}
\right|
\begin{array}{ll}
1 \ 1 \\
2 \ 2 \\
\end{array} \right]^*+
\\
+
\frac{1}{2!}
&\left[ \left.
\begin{array}{ll}
1 \ 2 \\
1 \ 1 \\
\end{array}
\right|
\begin{array}{ll}
1 \ 2 \\
1 \ 1 \\
\end{array} \right]^*&
+&\left[ \left.
\begin{array}{ll}
1 \ 2 \\
1\ 2 \\
\end{array}
\right|
\begin{array}{ll}
1 \ 2 \\
1 \ 2 \\
\end{array} \right]^*&
+
\frac{1}{2!}
&\left[ \left.
\begin{array}{ll}
1 \ 2 \\
2\ 2 \\
\end{array}
\right|
\begin{array}{ll}
1 \ 2 \\
2 \ 2 \\
\end{array} \right]^*+
\\
+
\frac{1}{2!2!}
&\left[ \left.
\begin{array}{ll}
2 \ 2 \\
1 \ 1 \\
\end{array}
\right|
\begin{array}{ll}
2 \ 2 \\
1 \ 1 \\
\end{array} \right]^*&
+
\frac{1}{2!}
&\left[ \left.
\begin{array}{ll}
2 \ 2 \\
1\ 2 \\
\end{array}
\right|
\begin{array}{ll}
2 \ 2 \\
1 \ 2 \\
\end{array} \right]^*&
+
\frac{1}{2!2!}
&\left[ \left.
\begin{array}{ll}
2 \ 2 \\
2\ 2 \\
\end{array}
\right|
\begin{array}{ll}
2 \ 2 \\
2 \ 2 \\
\end{array} \right]^*.
\end{align*}
In turn, applying the commutator relations, we get:
\begin{align*}
\mathbf{J}_{(2,2)}(2)=&\frac{1}{2!} [1 1|1 1]^*\frac{1}{2!} [1 1|1 1]^* + \mathbf{G}_1 + \frac{1}{2!} [1 1|1 1]^* [1 2|1 2]^* + \mathbf{G}_2 +
\\
+&\frac{1}{2!} [1 1|1 1]^*\frac{1}{2!} [1 1| 2 2]^* + \mathbf{G}_3 +[1 2|1 2]^*\frac{1}{2!} [1 1|1 1]^* + \mathbf{G}_4 +
\\
+&[1 2|1 2]^* [1 2|1 2]^* + \mathbf{G}_5 +  [1 2|1 2]^*\frac{1}{2!} [2 2|2 2]^* + \mathbf{G}_6 +
\\
+ &\frac{1}{2!} [2 2|2 2 ]^*\frac{1}{2!} [1 1|1 1]^* + \mathbf{G}_7 + \frac{1}{2!} [2 2|2 2]^* [1 2|1 2]^* + \mathbf{G}_8 + \\
+&\frac{1}{2!} [2 2|2 2]^*\frac{1}{2!} [2 2|2 2]^* + \mathbf{G}_9,
\end{align*}
where $ \mathbf{G}_j \in \boldsymbol{U}(gl(2))^{(p)}$,  for  some $p < |\lambda| = 4$, for $j = 1, 2, \ldots, 9.$
Therefore,
\begin{align*}
\mathbf{J}_{(2,2)}(2) &= \Big(\frac{1}{2!} [1 1|1 1]^* + [1 2|1 2]^* + \frac{1}{2!} [2 2|2 2]^* \Big)^2 +  \sum_{j=1}^9 \ \mathbf{G}_j =
\mathbf{I}_2(2)^2 + \sum_{j=1}^9 \ \mathbf{G}_j
\\
&= \mathbf{I}_{(2,2)}(2) +  \sum_{j=1}^9 \ \mathbf{G}_j,
\end{align*}
where $\mathbf{I}_{(2,2)}(2) \notin \boldsymbol{\zeta}(2)^{(m)},$ if  $m < |\lambda| = 4,$  and
$\sum_{j=1}^9 \ \mathbf{G}_j \in \boldsymbol{\zeta}(2)^{(p)},$ for  some $p < |\lambda| = 4.$

\end{example}\qed

\begin{proposition}\label{J basis}

For every $m \in \mathbb{Z}^+$, the set
$$
\big\{ \mathbf{J}_{\lambda}(n); \ \lambda_1 \leq n, \ |\lambda| \leq m \ \big\}
$$
is a linear basis of $\boldsymbol{\zeta}(n)^{(m)}.$

The set
$$
\big\{ \mathbf{J}_{\lambda}(n); \ \lambda_1 \leq n \ \big\}
$$
is a linear basis of the center $\boldsymbol{\zeta}(n).$
\end{proposition}
\begin{proof} We first notice that the elements of the set $\big\{ \mathbf{J}_{\lambda}(n); \ \lambda_1 \leq n \ \big\}$
are linearly independent. Indeed, by
Proposition \ref{J filtration}, we can limit ourselves to consider a linear combination
$$
\sum_\lambda \ c_\lambda \ \mathbf{J}_{\lambda}(n), \quad \lambda_1 \leq n, \  |\lambda| = m.
$$
If this linear combination  equals zero, then, again by Proposition \ref{J filtration}, we should
have $\sum_\lambda \ c_\lambda \ \mathbf{I}_{\lambda}(n) = 0$; since each $\mathbf{I}_{\lambda}(n)$ is a product of Nazarov/Umeda
elements $\mathbf{I}_k(n)$, $k \leq n$, this is a contradiction with respect to Proposition \ref{Umeda free}.
Since the two sets
$$
\big\{ \mathbf{J}_{\lambda}(n); \ \lambda_1 \leq n, \ |\lambda| \leq m \ \big\} \ \text{and} \
\big\{ \mathbf{K}_{\lambda}(n); \ \lambda_1 \leq n, \ |\lambda| \leq m \ \big\}
$$
have the same cardinality, the assertions follow from Proposition \ref{K basis}.
\end{proof}

By    Proposition \ref{J basis}, we get a stronger version of Proposition \ref{J filtration}.

\begin{proposition}
We have
$$
\mathbf{J}_{\lambda}(n) =  \mathbf{I}_{\lambda}(n) + \sum \ d_{\lambda,\mu} \mathbf{J}_{\mu}(n),
$$
where $d_{\lambda,\mu} = 0$ unless $|\mu| < |\lambda|$.
\end{proposition}

\begin{example}
Consider the Harish-Chandra isomorphism
$$
\chi_2 : \boldsymbol{\zeta}(2) \rightarrow \Lambda^*(2),
$$
where $\Lambda^*(2)$ denotes the algebra of {\textit{shifted symmetric}} polynomials
in two variables \cite{OkOlsh-BR}.
By computing the eigevalues of highest weight vectors as polynomials in the weight, we have:
\begin{align*}
\chi_2(\mathbf{J}_{(2,2)}(2)) & = y^4+2xy^3-8y^3+3x^2y^2-11xy^2+21y^2+ \\
                              &  \phantom{ = = }  +2x^3y-11x^2y+19xy-18y+x^4-6x^3+11x^2-6x,
\\
\chi_2(\mathbf{J}_{(2,1)}(2)) & = y^3+2xy^2-4y^2+2x^2y-5xy+4y+x^3-3x^2+2x,
\\
\chi_2(\mathbf{J}_{(2)}(2)) & = \chi_2(\mathbf{I}_2(2)) = y^2+xy-2y+x^2-x,
\\
\chi_2(\mathbf{J}_{(1,1,1)}(2)) & = y^3+3xy^2-3y^2+3x^2y-6xy+2y+x^3-3x^2+2x,
\\
\chi_2(\mathbf{I}_1(2)) & =  y+x.
\end{align*}
By direct computations, we get the following identity in the algebra $\Lambda^*(2)$:
\begin{align*}
\chi_2 \big( \mathbf{J}_{(2,2)}(2) \big) &= \chi_2 \big(  \mathbf{I}_2(2)^2 - 7\mathbf{J}_{(2,1)}(2)  - 2\mathbf{J}_{(2)}(2) + 3\mathbf{J}_{(1,1,1)}(2) \big) \\
&= \chi_2 \big( \mathbf{I}_2(2)^2 - 7\mathbf{I}_2(2)\mathbf{I}_1(2) + 3\mathbf{I}_1(2)^3 + 12\mathbf{I}_2(2) - 9\mathbf{I}_1(2)^2 +
6\mathbf{I}_1(2) \big).
\end{align*}
Hence, in the center $\boldsymbol{\zeta}(2)$, we have:
\begin{align*}
\mathbf{J}_{(2,2)}(2) &=  \mathbf{I}_2(2)^2 - 7\mathbf{J}_{(2,1)}(2)  - 2\mathbf{J}_{(2)}(2) + 3\mathbf{J}_{(1,1,1)}(2),\\
&= \mathbf{I}_2(2)^2 - 7\mathbf{I}_2(2)\mathbf{I}_1(2) + 3\mathbf{I}_1(2)^3 + 12\mathbf{I}_2(2) - 9\mathbf{I}_1(2)^2 +
6\mathbf{I}_1(2),
\end{align*}
(Notice that $\mathbf{I}_2(2)^2 = \mathbf{I}_{(2,2)}(2)$).
\end{example}

\subsection{The virtual form of the Schur/Okounkov/Sahi basis  $\mathbf{S}_\lambda(n)$}\label{Schur elements}

\subsubsection{The  virtual definition of $\mathbf{S}_\lambda(n)$ and   main results}

Let $\lambda \vdash d$ be a partition,  $\lambda_1 \leq n$ .
We notice that any element
$$
e_{S_1,C_{\lambda}^{*}} \cdot e_{C_{\lambda}^{*},D_{\lambda}^{*}} \cdot  e_{D_{\lambda}^{*},C_{\lambda}^{*}}
\cdot e_{C_{\lambda}^{*},S_2} \in Virt(m_0+m_1,n),
$$
where $S_1, S_2$ are tableaux  on the proper alphabet $L = \{x_1, \ldots, x_n \}$ of shape $\lambda$,
$\lambda_1 \leq n$, $m_0 \geq \widetilde{\lambda}_1, \ m_1 \geq \lambda_1$, is
{\textit{skew-symmetric}}  in the rows of $S_1$ and  $S_2$, respectively.

\begin{definition}\label{Schur basis}

We set
\begin{align*}
\mathbf{S}_\lambda(n) &=
\frac {1} {H(\lambda)} \  \ \sum_S \ \mathfrak{p} \big(e_{S,C_{\lambda}^{*}} \cdot e_{C_{\lambda}^{*},D_{\lambda}^{*}} \cdot  e_{D_{\lambda}^{*},C_{\lambda}^{*}}
\cdot e_{C_{\lambda}^{*},S} \big) \\
&= \frac {1} {H(\lambda)} \ \  \sum_S \ SC_{\lambda}^{*} \ C_{\lambda}^{*}D_{\lambda}^{*} \
D_{\lambda}^{*}C_{\lambda}^{*} \ C_{\lambda}^{*}S \in {\mathbf{U}}(gl(n)), \\
\end{align*}
where the sum is extended to all  row (strictly) increasing tableaux $S$ on the proper alphabet $L = \{x_1, \ldots, x_n \}$.
\end{definition}

By convention, if $\lambda$ is the empty partition, we set $\mathbf{S}_{\emptyset}(n) = \mathbf{1} \in \boldsymbol{\zeta}(n).$

The element $\mathbf{S}_\lambda(n) \in {\mathbf{U}}(gl(n))$ is called the {\it{Schur element}} of shape $\lambda$ in dimension $n$.

\begin{example}   Let $\lambda = (2,1) \vdash d = 3$, $n = 2$. The Schur element $\mathbf{S}_{(2,1)}(2) \in {\mathbf{U}}(gl(2))$ is expressed
as the image (under the Capelli epimorphism $\mathfrak{p}$) of the sum of three monomials in $Virt(m_0+m_1,3) \subset {\mathbf{U}}(gl(m_0|m_1+2)),$
$m_0 \geq 2$ and $m_1 \geq :$
\begin{align*}
\mathbf{S}_{(2,1)}(2) &=
\frac {1} {3}  \times
\\
&\mathfrak{p} \Big( e_{1 \alpha_1}e_{2 \alpha_1}e_{1 \alpha_2}
e_{\alpha_1 \beta_1}e_{\alpha_1 \beta_2}e_{ \alpha_2 \beta_1}
e_{\beta_1 \alpha_1 }e_{ \beta_2 \alpha_1}e_{ \beta_1 \alpha_2 }
e_{ \alpha_1 1}e_{ \alpha_1 2}e_{ \alpha_2 1} +
\\
&+ e_{1 \alpha_1}e_{2 \alpha_1}e_{2 \alpha_2}
e_{\alpha_1 \beta_1}e_{\alpha_1 \beta_2}e_{ \alpha_2 \beta_1}
e_{\beta_1 \alpha_1 }e_{ \beta_2 \alpha_1}e_{ \beta_1 \alpha_2 }
e_{ \alpha_1 1}e_{ \alpha_1 2}e_{ \alpha_2 2} \Big),
\end{align*}
where $\alpha_1, \alpha_2 \in A_0$ and $\beta_1, \beta_2,  \in A_1$.

In the notation and terminology of Subsection \ref{Bitableaux monomials}, the Schur element $\mathbf{S}_{(2,1)}(2)$
is expressed as the image $\mathfrak{p}$ as the sum of two elements of ${\mathbf{U}}(gl(2))$, each of them being the product of four
{\it{bitableau monomials}}:
\begin{align*}
\mathbf{S}_{(2,1)}(2) &=
\frac {1} {3} \times
\mathfrak{p} \Big( e_{{\fontsize{6} {6} \selectfont \substack{
1  2  \ \alpha_1  \alpha_1  \\
1   \phantom{2}\ \alpha_2   \phantom{\alpha_2}\\
}}} \times
e_{{\fontsize{6} {6} \selectfont \substack{
\alpha_1  \alpha_1  \ \beta_1  \beta_2  \\
\alpha_2   \phantom{\alpha_1}\ \beta_1   \phantom{\beta_2}\\
}}}
\times
e_{{\fontsize{6} {6} \selectfont \substack{
\beta_1  \beta_2  \ \alpha_1  \alpha_1  \\
\beta_1   \phantom{\beta_2}\ \alpha_2   \phantom{\alpha_1}\\
}}}
\times
e_{{\fontsize{6} {6} \selectfont \substack{
\alpha_1  \alpha_1  \ 1  2  \\
\alpha_2   \phantom{\alpha_1} \ 1   \phantom{2}\\
}}}
 \Big) +
\\
&+\frac {1} {3} \times
\mathfrak{p} \Big( e_{{\fontsize{6} {6} \selectfont \substack{
1  2  \ \alpha_1  \alpha_1  \\
2   \phantom{2}\ \alpha_2   \phantom{\alpha_2}\\
}}} \times
e_{{\fontsize{6} {6} \selectfont \substack{
\alpha_1  \alpha_1  \ \beta_1  \beta_2  \\
\alpha_2   \phantom{\alpha_1}\ \beta_1   \phantom{\beta_2}\\
}}}
\times
e_{{\fontsize{6} {6} \selectfont \substack{
\beta_1  \beta_2  \ \alpha_1  \alpha_1  \\
\beta_1   \phantom{\beta_2}\ \alpha_2   \phantom{\alpha_1}\\
}}}
\times
e_{{\fontsize{6} {6} \selectfont \substack{
\alpha_1  \alpha_1  \ 1  2  \\
\alpha_2   \phantom{\alpha_1} \ 2   \phantom{2}\\
}}}
 \Big).
\end{align*}
Therefore, the Schur element $\mathbf{S}_{(2,1)}(2)$ acts on the Schur module $Schur_\mu(2)$, $\mu_1 \leq 2$, by the following
polynomial in superpolarization operators:
\begin{scriptsize}
\begin{align*}
& \frac {1} {3} \times
\big( D_{1 \alpha_1}D_{2 \alpha_1}D_{1 \alpha_2}
D_{ \alpha_1 \beta_1}D_{ \alpha_1 \beta_2 }D_{ \alpha_2 \beta_1}
D_{  \beta_1 \alpha_1}D_{ \beta_2 \alpha_1 }D_{  \beta_1 \alpha_2}
D_{ \alpha_1 1}D_{ \alpha_1 2}D_{ \alpha_2 1} +
\\
&\phantom{\frac {1} {3} \times} + D_{1 \alpha_1}D_{2 \alpha_1}D_{2 \alpha_2}
D_{ \alpha_1 \beta_1}D_{ \alpha_1 \beta_2 }D_{ \alpha_2 \beta_1}
D_{  \beta_1 \alpha_1}D_{ \beta_2 \alpha_1 }D_{  \beta_1 \alpha_2}
D_{ \alpha_1 1}D_{ \alpha_1 2}D_{ \alpha_2 2}
\big).
\end{align*}
\end{scriptsize}
\end{example}\qed

\begin{proposition}
Consider the element
$$
 \sum_S \  e_{S,C_{\lambda}^{*}} \cdot e_{C_{\lambda}^{*},D_{\lambda}^{*} } \cdot e_{D_{\lambda}^{*},C_{\lambda}^{*}}
\cdot e_{C_{\lambda}^{*},S},
$$
where the sum is extended to all  row (strictly) increasing tableaux $S$ on the proper alphabet $L = \{x_1, \ldots, x_n \}$.

Since the adjoint representation acts by derivation, we have
$$
ad(e_{i j})\big( \sum_S \  e_{S,C_{\lambda}^{*}} \cdot e_{C_{\lambda}^{*},D_{\lambda}^{*} } \cdot e_{D_{\lambda}^{*},C_{\lambda}^{*}}
\cdot e_{C_{\lambda}^{*},S}  \big) = 0,
$$
for every  $e_{i j} \in gl(n).$
\end{proposition}

The proof is essentially the same as the proof of Proposition \ref{K centr}.

\begin{example} Consider the action of $ad(e_{21})$ on $\mathbf{S}_{(2,1)}(2)$.
We have
\begin{align*}
ad(e_{21}) \big( \mathbf{S}_{(2,1)}(2) \big) =
& \frac {1} {3} \times
\mathfrak{p}  \Big( +  e_{{\fontsize{6} {6} \selectfont \substack{
2  2  \ \alpha_1  \alpha_1  \\
1   \phantom{2}\ \alpha_2   \phantom{\alpha_2}\\
}}} \times
e_{{\fontsize{6} {6} \selectfont \substack{
\alpha_1  \alpha_1  \ \beta_1  \beta_2  \\
\alpha_2   \phantom{\alpha_1}\ \beta_1   \phantom{\beta_2}\\
}}}
\times
e_{{\fontsize{6} {6} \selectfont \substack{
\beta_1  \beta_2  \ \alpha_1  \alpha_1  \\
\beta_1   \phantom{\beta_2}\ \alpha_2   \phantom{\alpha_1}\\
}}}
\times
e_{{\fontsize{6} {6} \selectfont \substack{
\alpha_1  \alpha_1  \ 1  2  \\
\alpha_2   \phantom{\alpha_1} \ 1   \phantom{2}\\
}}}
\\
& \phantom{\frac {1} {3} \times
\mathfrak{p} \Big(} +
e_{{\fontsize{6} {6} \selectfont \substack{
1  2  \ \alpha_1  \alpha_1  \\
2   \phantom{2}\ \alpha_2   \phantom{\alpha_2}\\
}}} \times
e_{{\fontsize{6} {6} \selectfont \substack{
\alpha_1  \alpha_1  \ \beta_1  \beta_2  \\
\alpha_2   \phantom{\alpha_1}\ \beta_1   \phantom{\beta_2}\\
}}}
\times
e_{{\fontsize{6} {6} \selectfont \substack{
\beta_1  \beta_2  \ \alpha_1  \alpha_1  \\
\beta_1   \phantom{\beta_2}\ \alpha_2   \phantom{\alpha_1}\\
}}}
\times
e_{{\fontsize{6} {6} \selectfont \substack{
\alpha_1  \alpha_1  \ 1  2  \\
\alpha_2   \phantom{\alpha_1} \ 1   \phantom{2}\\
}}}
\\
& \phantom{\frac {1} {3} \times
\mathfrak{p} \Big(} -
e_{{\fontsize{6} {6} \selectfont \substack{
1  2  \ \alpha_1  \alpha_1  \\
1   \phantom{2}\ \alpha_2   \phantom{\alpha_2}\\
}}} \times
e_{{\fontsize{6} {6} \selectfont \substack{
\alpha_1  \alpha_1  \ \beta_1  \beta_2  \\
\alpha_2   \phantom{\alpha_1}\ \beta_1   \phantom{\beta_2}\\
}}}
\times
e_{{\fontsize{6} {6} \selectfont \substack{
\beta_1  \beta_2  \ \alpha_1  \alpha_1  \\
\beta_1   \phantom{\beta_2}\ \alpha_2   \phantom{\alpha_1}\\
}}}
\times
e_{{\fontsize{6} {6} \selectfont \substack{
\alpha_1  \alpha_1  \ 1  1  \\
\alpha_2   \phantom{\alpha_1} \ 1   \phantom{2}\\
}}}
\Big)
\\
+ &
\frac {1} {3} \times
\mathfrak{p} \Big(  + e_{{\fontsize{6} {6} \selectfont \substack{
2  2  \ \alpha_1  \alpha_1  \\
2   \phantom{2}\ \alpha_2   \phantom{\alpha_2}\\
}}} \times
e_{{\fontsize{6} {6} \selectfont \substack{
\alpha_1  \alpha_1  \ \beta_1  \beta_2  \\
\alpha_2   \phantom{\alpha_1}\ \beta_1   \phantom{\beta_2}\\
}}}
\times
e_{{\fontsize{6} {6} \selectfont \substack{
\beta_1  \beta_2  \ \alpha_1  \alpha_1  \\
\beta_1   \phantom{\beta_2}\ \alpha_2   \phantom{\alpha_1}\\
}}}
\times
e_{{\fontsize{6} {6} \selectfont \substack{
\alpha_1  \alpha_1  \ 1  2  \\
\alpha_2   \phantom{\alpha_1} \ 2   \phantom{2}\\
}}}
\\
& \phantom{\frac {1} {3} \times
\mathfrak{p} \Big(} -
e_{{\fontsize{6} {6} \selectfont \substack{
1  2  \ \alpha_1  \alpha_1  \\
2   \phantom{2}\ \alpha_2   \phantom{\alpha_2}\\
}}} \times
e_{{\fontsize{6} {6} \selectfont \substack{
\alpha_1  \alpha_1  \ \beta_1  \beta_2  \\
\alpha_2   \phantom{\alpha_1}\ \beta_1   \phantom{\beta_2}\\
}}}
\times
e_{{\fontsize{6} {6} \selectfont \substack{
\beta_1  \beta_2  \ \alpha_1  \alpha_1  \\
\beta_1   \phantom{\beta_2}\ \alpha_2   \phantom{\alpha_1}\\
}}}
\times
e_{{\fontsize{6} {6} \selectfont \substack{
\alpha_1  \alpha_1  \ 1  1  \\
\alpha_2   \phantom{\alpha_1} \ 2   \phantom{2}\\
}}}
\\
& \phantom{\frac {1} {3} \times
\mathfrak{p} \Big(} -
e_{{\fontsize{6} {6} \selectfont \substack{
1  2  \ \alpha_1  \alpha_1  \\
2   \phantom{2}\ \alpha_2   \phantom{\alpha_2}\\
}}} \times
e_{{\fontsize{6} {6} \selectfont \substack{
\alpha_1  \alpha_1  \ \beta_1  \beta_2  \\
\alpha_2   \phantom{\alpha_1}\ \beta_1   \phantom{\beta_2}\\
}}}
\times
e_{{\fontsize{6} {6} \selectfont \substack{
\beta_1  \beta_2  \ \alpha_1  \alpha_1  \\
\beta_1   \phantom{\beta_2}\ \alpha_2   \phantom{\alpha_1}\\
}}}
\times
e_{{\fontsize{6} {6} \selectfont \substack{
\alpha_1  \alpha_1  \ 1  2  \\
\alpha_2   \phantom{\alpha_1} \ 1   \phantom{2}\\
}}}
\Big),
\end{align*}
that equals, by row skew-symmetry,
\begin{align*}
ad(e_{21}) \big( \mathbf{S}_{(2,1)}(2) \big) =
& \frac {1} {3} \times
\mathfrak{p} \Big(
+
e_{{\fontsize{6} {6} \selectfont \substack{
1  2  \ \alpha_1  \alpha_1  \\
2   \phantom{2}\ \alpha_2   \phantom{\alpha_2}\\
}}} \times
e_{{\fontsize{6} {6} \selectfont \substack{
\alpha_1  \alpha_1  \ \beta_1  \beta_2  \\
\alpha_2   \phantom{\alpha_1}\ \beta_1   \phantom{\beta_2}\\
}}}
\times
e_{{\fontsize{6} {6} \selectfont \substack{
\beta_1  \beta_2  \ \alpha_1  \alpha_1  \\
\beta_1   \phantom{\beta_2}\ \alpha_2   \phantom{\alpha_1}\\
}}}
\times
e_{{\fontsize{6} {6} \selectfont \substack{
\alpha_1  \alpha_1  \ 1  2  \\
\alpha_2   \phantom{\alpha_1} \ 1   \phantom{2}\\
}}}
\\
& \phantom{\frac {1} {3} \times
\mathfrak{p} \Big(} -
e_{{\fontsize{6} {6} \selectfont \substack{
1  2  \ \alpha_1  \alpha_1  \\
2   \phantom{2}\ \alpha_2   \phantom{\alpha_2}\\
}}} \times
e_{{\fontsize{6} {6} \selectfont \substack{
\alpha_1  \alpha_1  \ \beta_1  \beta_2  \\
\alpha_2   \phantom{\alpha_1}\ \beta_1   \phantom{\beta_2}\\
}}}
\times
e_{{\fontsize{6} {6} \selectfont \substack{
\beta_1  \beta_2  \ \alpha_1  \alpha_1  \\
\beta_1   \phantom{\beta_2}\ \alpha_2   \phantom{\alpha_1}\\
}}}
\times
e_{{\fontsize{6} {6} \selectfont \substack{
\alpha_1  \alpha_1  \ 1  2  \\
\alpha_2   \phantom{\alpha_1} \ 1   \phantom{2}\\
}}}
\Big)
 = 0.
\end{align*}
\end{example}\qed

From Remark \ref{centrality-BR}, it follows

\begin{theorem}
The Schur elements $\mathbf{S}_\lambda(n)$ are  central in $\mathbf{U}(gl(n))$.
\end{theorem}

Clearly,
\begin{equation}\label{filtration element}
\mathbf{S}_\lambda(n) \in \boldsymbol{\zeta}(n)^{(m)},
\end{equation}
for every $m \geq |\lambda|.$

\begin{theorem}(Vanishing theorem)\label{Vanishing theorem}
We have:
\begin{align*}
&\textrm{If} \ \lambda   \nsubseteq \mu , \  \textrm{then}  &   \mathbf{S}_{\lambda}(n)(v_{\widetilde{\mu}}) &= 0,
\\
&\textrm{If} \ \mu = \lambda, \ \textrm{then} &
\mathbf{S}_{\lambda} (n) (v_{\widetilde{\lambda}}) &=  H(\lambda) \cdot v_{\widetilde{\lambda}},
\end{align*}
where $H(\lambda)$ denotes the hook number of the shape (partition) $\lambda \vdash k.$
\end{theorem}
\begin{proof}
The first assertion is an immediate consequence of Proposition \ref{Vanishing Lemma}, eq. (\ref{cinque}).

Recall that the element
$$
v_{\widetilde{\lambda}} = (D_\lambda|D^P_\lambda)
$$
is the ``canonical''  highest weight vector of the irreducible $gl(n)-$module $Schur_\lambda(n)$.
Clearly
\begin{align*}
\mathbf{S}_\lambda(n) \left( v_{\widetilde{\lambda}} \right) &= \frac {1} {H(\lambda)} \ \sum_S \
SC_{\lambda}^{*} \ C_{\lambda}^{*}D_{\lambda}^{*} \ D_{\lambda}^{*}C_{\lambda}^{*} \ C_{\lambda}^{*}S \big( (D_\lambda|D^P_\lambda) \big)
\\
&= \frac {1} {H(\lambda)} \  \ D_{\lambda}C_{\lambda}^{*} \ C_{\lambda}^{*}D_{\lambda}^{*}  \ D_{\lambda}^{*}C_{\lambda}^{*}
\ C_{\lambda}^{*}D_\lambda \big( (D_\lambda|D^P_\lambda) \big).
\end{align*}

By Proposition \ref{hook lemma}, eqs. (\ref{proper}), (\ref{trivial}),
\begin{align*}
\mathbf{S}_\lambda(n) \left( v_{\widetilde{\lambda}} \right) &=  \frac {1} {H(\lambda)} (-1)^{{k} \choose {2}}  \  H(\lambda) \cdot
D_{\lambda}C_{\lambda}^{*} \ C_{\lambda}^{*}D_{\lambda}^{*} \ D_{\lambda}^{*}C_{\lambda}^{*} \big( (C_{\lambda}^{*}|D^P_\lambda) \ (\lambda!)^{-1} \big)
\\
&= \frac {1} {H(\lambda)} (-1)^{{k} \choose {2}}  \  H(\lambda) \cdot
D_{\lambda}C_{\lambda}^{*} \ C_{\lambda}^{*}D_{\lambda}^{*}  \big( (D_{\lambda}^{*}|D^P_\lambda)  \big)
\\
&= \frac {1} {H(\lambda)} (-1)^{{k} \choose {2}}  \  H(\lambda) \cdot D_{\lambda}C_{\lambda}^{*} \big( (C_{\lambda}^{*}|D^P_\lambda) \ (\lambda!)^{-1} \big) \cdot H(\lambda) (-1)^{{k} \choose {2}}
\\
&= H(\lambda)  \cdot (D_\lambda|D^P_\lambda).
\end{align*}
\end{proof}

\begin{theorem}(Triangularity/orthogonality of the actions on highest weight vectors)\label{Schur action}
We have:
\begin{align*}
&\textrm{If} \ |\mu| < |\lambda|, \  \textrm{then}  &   \mathbf{S}_{\lambda}(n)(v_{\widetilde{\mu}}) &= 0,
\\
&\textrm{If} \ |\mu| = |\lambda|, \ \textrm{then} &
\mathbf{S}_{\lambda} (n) (v_{\widetilde{\mu}}) &= \delta_{\lambda, \mu} \cdot H(\lambda) \cdot v_{\widetilde{\mu}}.
\end{align*}

\end{theorem}
\begin{proof}
The first assertion is an immediate consequence of Proposition \ref{Vanishing Lemma}, eq. (\ref{due}).
The fact that, if $|\mu| = |\lambda|,$ $\mu \neq \lambda,$ then $\mathbf{S}_{\lambda} (n) (v_{\widetilde{\mu}}) = 0,$
is an immediate consequence of Proposition \ref{Vanishing Lemma}, eq. (\ref{tre}).
\end{proof}

\begin{theorem}\label{Schur basis}

For every $m \in \mathbb{Z}^+$, the set
$$
\big\{ \mathbf{S}_{\lambda}(n); \ \lambda_1 \leq n, \ |\lambda| \leq m \ \big\}
$$
is a linear basis of $\boldsymbol{\zeta}(n)^{(m)}.$

The set
$$
\big\{ \mathbf{S}_{\lambda}(n); \ \lambda_1 \leq n \ \big\}
$$
is a linear basis of the center $\boldsymbol{\zeta}(n).$

\end{theorem}

\subsubsection{The determinantal Capelli generators and the permanental Nazarov/Umeda elements as elements of the Schur basis}

\begin{proposition}
Let $\lambda = (k)$ denote the row shape of length $k.$ We have
$$
\mathbf{S}_{(k)}(n) = \mathbf{H}_k(n).
$$
\end{proposition}
\begin{proof}

We have
\begin{align*}
\mathbf{S}_{(k)}(n) &= \frac {1} {H(\lambda)} \  \ \sum_S \ \mathfrak{p} \big(e_{S,C_{(k)}^{*}} \cdot e_{C_{(k)}^{*},D_{(k)}^{*}} \cdot  e_{D_{(k)}^{*},C_{(k)}^{*}} \cdot e_{C_{(k)}^{*},S} \big) \\
&= \frac {1} {k!} \sum_S \ SC_{(k)}^{*} \
C_{(k)}^{*}D_{(k)}^{*} \ D_{(k)}^{*}C_{(k)}^{*} \ C_{(k)}^{*}S,
\end{align*}
where the sum is extended to all strictly increasing row tableaux $S$ of shape $(k)$.

Notice that
$$
\mathfrak{p} \big(e_{S,C_{(k)}^{*}} \cdot e_{C_{(k)}^{*},D_{(k)}^{*}} \cdot  e_{D_{(k)}^{*},C_{(k)}^{*}} \cdot e_{C_{(k)}^{*},S} \big)
$$
equals
$$
(-1)^{{k} \choose {2}} \
\mathfrak{p} \big(e_{S,C_{(k)}^{*}} \cdot e_{C_{(k)}^{*},C_{(k)}^{*}}  \cdot e_{C_{(k)}^{*},S} \big),
$$
that, in turn, equals
$$
(-1)^{{k} \choose {2}} \ k! \
\mathfrak{p} \big(e_{S,C_{(k)}^{*}}   \cdot e_{C_{(k)}^{*},S} \big).
$$
Therefore,
\begin{align*}
\mathbf{S}_{(k)}(n) &= (-1)^{{k} \choose {2}} \ \sum_{1 \leq i_1 < \cdots < i_k \leq n} \ \mathfrak{p} \big( e_{i_1 , \alpha}  e_{i_2, \alpha} \cdots e_{i_k, \alpha}
e_{\alpha, i_1}e_{\alpha, i_2} \cdots e_{\alpha, i_k } \big)
\\
&=\sum_{1 \leq i_1 < \cdots < i_k \leq n} \ \mathfrak{p} \big( e_{i_k , \alpha} \cdots e_{i_2, \alpha} e_{i_1, \alpha}
e_{\alpha, i_1}e_{\alpha, i_2} \cdots e_{\alpha, i_k } \big)
\\
&= \mathbf{H}_k(n).
\end{align*}
\end{proof}

\begin{remark}
Notice that $e^*_k((k))$, the eigenvalue of $\mathbf{H}_k$ on the row shape $(k)$ of length $k$, equals the hook
coefficient $H( (k) ) = k!.$
\end{remark}

\begin{proposition}
Let $\lambda = (1^k)$ denote the column shape of length $k.$ We have
$$
\mathbf{S}_{(1^k)}(n) = \mathbf{I}_k(n).
$$
\end{proposition}
\begin{proof}

We have
\begin{align*}
\mathbf{S}_{(1^k)}(n) &= \frac {1} {H(\lambda)} \  \ \sum_S \ \mathfrak{p} \big(e_{S,C_{(1^k)}^{*}} \cdot e_{C_{(1^k)}^{*},D_{(1^k)}^{*}} \cdot  e_{D_{(1^k)}^{*},C_{(1^k)}^{*}} \cdot e_{C_{(1^k)}^{*},S} \big) \\
&= \frac {1} {k!} \sum_S \ SC_{(1^k)}^{*} \
C_{(1^k)}^{*}D_{(1^k)}^{*} \ D_{(1^k)}^{*}C_{(1^k)}^{*} \ C_{(1^k)}^{*}S,
\end{align*}
where the sum is extended to all column tableaux $S$ of shape $(1^k)$.

Since the column tableau $C_{(1^k)}^{*}$ is multilinear (that is,
$$
C_{(1^k)}^{*} =
\left(
\begin{array}{c}
\alpha_1\\  \alpha_2 \\ \vdots \\ \alpha_k
\end{array}
\right),
$$
where the $\alpha_i$'s are distinct positive virtual symbols), then each summand
$$
\mathfrak{p} \big(e_{S,C_{(1^k)}^{*}} \cdot e_{C_{(1^k)}^{*},D_{(1^k)}^{*}} \cdot  e_{D_{(1^k)}^{*},C_{(1^k)}^{*}} \cdot e_{C_{(1^k)}^{*},S} \big)
$$
equals
$$
\mathfrak{p} \big(e_{S,D_{(1^k)}^{*}}  \cdot e_{D_{(1^k)}^{*},S} \big).
$$
Therefore
\begin{align*}
\mathbf{S}_{(1^k)}(n) &= \frac {1} {k!} \sum_S \ SD_{(1^k)}^{*}  D_{(1^k)}^{*}S
\\
&= \frac {1} {k!} \sum_{(h_1, \ldots, h_n)} \ \sum_T \ TD_{(1^k)}^{*}  D_{(1^k)}^{*}T,
\end{align*}
where the outer sum is extended over all indexes $h_1 + \cdots + h_n = k$ and inner sum
is extended over all column tableaux $T$ with $h_1$ occurrences of $1$, $h_2$ occurrences of $2$,
$\ldots$, $h_n$ occurrences of $n.$
Moreover, since each element  $TD_{(1^k)}^{*}$ and  $D_{(1^k)}^{*}T$ is row-commutative, then
the inner sum
$$
\sum_T \ TD_{(1^k)}^{*}  D_{(1^k)}^{*}T
$$
equals
$$
{{k} \choose  {h_1,h_2, \ldots, h_n}}
\cdot
\left[
\begin{array}{c}
1 \\ \vdots \\ 1 \\ \vdots \\ n \\ \vdots \\ n
\end{array}
\right| \left.
\begin{array}{c}
\beta_1 \\ \vdots \\ \beta_1\\  \vdots \\ \beta_1 \\ \vdots \\ \beta_1
\end{array}
\right]
\left[
\begin{array}{c}
\beta_1 \\ \vdots \\ \beta_1\\  \vdots \\ \beta_1 \\ \vdots \\ \beta_1
\end{array}
\right| \left.
\begin{array}{c}
1 \\ \vdots \\ 1 \\ \vdots \\ n \\ \vdots \\ n
\end{array}
\right],
$$
where there are  $h_1$ occurrences of $1$, $h_2$ occurrences of $2$,
$\ldots$, $h_n$ occurrences of $n.$
Therefore
\begin{align*}
\mathbf{S}_{(1^k)}(n) &= \frac {1} {k!} \sum_{(h_1, \ldots, h_n)} \ \sum_T \ TD_{(1^k)}^{*}  D_{(1^k)}^{*}T
\\
&= \frac {1} {k!} \sum_{(h_1, \ldots, h_n)} \ {{k} \choose  {h_1,h_2, \ldots, h_n}}
\cdot
\left[
\begin{array}{c}
1 \\ \vdots \\ 1 \\ \vdots \\ n \\ \vdots \\ n
\end{array}
\right| \left.
\begin{array}{c}
\beta_1 \\ \vdots \\ \beta_1\\  \vdots \\ \beta_1 \\ \vdots \\ \beta_1
\end{array}
\right]
\left[
\begin{array}{c}
\beta_1 \\ \vdots \\ \beta_1\\  \vdots \\ \beta_1 \\ \vdots \\ \beta_1
\end{array}
\right| \left.
\begin{array}{c}
1 \\ \vdots \\ 1 \\ \vdots \\ n \\ \vdots \\ n
\end{array}
\right]
\\
&= \sum_{(h_1, h_2, \ldots, h_n) } \ \frac {1} {h_1! h_2! \cdots h_n!} \ [n^{h_n} \cdots 2^{h_2} 1^{h_1} | 1^{h_1} 2^{h_2} \cdots n^{h_n} ]^{*}
\\
&= \mathbf{I}_k(n).
\end{align*}
\end{proof}

\begin{remark}

Notice that  $h^*_k( (1^k))$, the eigenvalue of $\mathbf{I}_k$ on the column shape $(1^k)$ of length $k$, equals the hook
coefficient $H( (1^k) ) = k!.$
\end{remark}

\subsubsection{The Sahi/Okounkov Characterization  Theorem}\label{Characterization  Theorems}
We anticipate the formulation of Theorem \ref{Schur action} in terms of the Harish-Chandra
isomorphism (Definition \ref{HC isomorphism}, below)
$$
\chi_n : \boldsymbol{\zeta}(n) \longrightarrow \Lambda^*(n),
$$
where $\Lambda^*(n)$ denotes the algebra of {\it{shifted symmetric polynomials}} in $n$ variables (see Section \ref{Lambda(n)} below).
\begin{proposition}\label{HC action}
We have that $\chi_n\left( \mathbf{S}_{\lambda}(n) \right)$ is an element $m-$th filtration  element $\Lambda^*(n)^{(m)}$, for any $m \geq |\lambda|$,
of  the algebra
$\Lambda^*(n)$ of shifted symmetric polynomials in $n$ variables. Furthermore
\begin{align*}
&\textrm{If} \ |\mu| < |\lambda|, \  \textrm{then}  &   \chi_n\left( \mathbf{S}_{\lambda}(n) \right)(\widetilde{\mu}) &= 0,
\\
&\textrm{If} \ |\mu| = |\lambda|, \ \textrm{then} &
\chi_n\left( \mathbf{S}_{\lambda}(n) \right)(\widetilde{\mu}) &= \delta_{\lambda, \mu} \cdot H(\lambda),
\end{align*}
where $H(\lambda)$ denotes the hook number of the shape (partition) $\lambda.$
\end{proposition}
By combining Proposition \ref{HC action} with the {\it{Sahi/Okounkov Characterization Theorem}}
(Theorem 1 of \cite{Sahi2-BR} and Theorem 3.3 of \cite{OkOlsh-BR}, see also \cite{Okounkov-BR}), the shifted symmetric polynomial
$\chi_n\left( \mathbf{S}_{\lambda}(n) \right)$  is the
{\it{Schur shifted symmetric polynomial}} $s_{\widetilde{\lambda}|n}^*$ of \cite{OkOlsh-BR}.
It follows that the basis $\big\{ \mathbf{S}_{\lambda}(n); \ \lambda_1 \leq n \ \big\}$
of the center $\boldsymbol{\zeta}(n)$ is the preimage of the basis of Schur shifted symmetric polynomials
of $\Lambda^*(n)$ characterized and described by Sahi \cite{Sahi2-BR}
({\it{recursive procedure}}), and, moreover, it coincides with  the basis of $\boldsymbol{\zeta}(n)$
described by Okounkov in terms of {\it{quantum immanants}} \cite{Okounkov-BR}
(for further descriptions, see also \cite{Okounkov1-BR} and \cite{Nazarov2-BR}).

\begin{remark}
Notice that $e^*_k((k))$, the eigenvalue of $\mathbf{H}_k$ on the row shape $(k)$ of length $k$, equals the hook
coefficient $H( (k) ).$
Similarly,  $h^*_k( (1^k))$, the eigenvalue of $\mathbf{I}_k$ on the column shape $(1^k)$ of length $k$, equals the hook
coefficient $H( (1^k) ).$
\end{remark}

\subsection{Duality in $\boldsymbol{\zeta}(n)$}\label{central duality}

Let
$$
\mathcal{W}_n : \boldsymbol{\zeta}(n) \rightarrow \boldsymbol{\zeta}(n)
$$
be the algebra automorphism defined by setting
$$
\mathcal{W}_n \Big(  \mathbf{H}_k(n)  \Big) = \mathbf{I}_k(n), \quad k = 1, 2, \ldots, n.
$$

Clearly, Proposition \ref{horizontal strip} and Theorem \ref{vertical strip}.$2$ imply
that, if $\mu_1, \widetilde{\mu}_1 \leq n$, then
\begin{equation}\label{duality}
e^*_k(\widetilde{\mu}) = h^*_k(\mu).
\end{equation}

\begin{theorem}\label{finite duality}
Let $\mu$ be such that $\mu_1, \widetilde{\mu}_1 \leq n.$
For every $\boldsymbol{\varrho} \in \boldsymbol{\zeta}(n)$ the eigenvalue of $\boldsymbol{\varrho}$ on the
$gl(n)-$irreducible module $Schur_{\mu}(n)$ (with highest weight $\widetilde{\mu}$) equals eigenvalue of
$\mathcal{W}_n \Big( \boldsymbol{\varrho} \Big)$ on the
$gl(n)-$irreducible module $Schur_{\widetilde{\mu}}(n)$ (with highest weight $\mu$).
\end{theorem}

\begin{example}  From Example \ref{horizontal example}, the eigenvalue of $\mathbf{H}_2(3)$
on the $gl(3)$-irreducible module $Schur_{(3,2)}(3)$ equals $e^*_2((2,2,1)) = 12.$

From Example \ref{vertical example}, the eigenvalue of $\mathcal{W}_3 \Big(  \mathbf{H}_2(3)  \Big) = \mathbf{I}_2(3)$
on the $gl(3)$-irreducible module $Schur_{(2,2,1)}(3)$ equals $h^*_2((3,2)) = 12.$
\end{example}\qed

The preceding result, in combination with the characterization Theorems of subsection \ref{Characterization Theorems}, implies

\begin{corollary}\label{Schur duality}

Let $\lambda_1, \widetilde{\lambda}_1 \leq n$. Since $H(\lambda) = H(\widetilde{\lambda}),$  then

$$
\mathcal{W}_n \Big(  \mathbf{S}_{\lambda}(n)  \Big) = \mathbf{S}_{\widetilde{\lambda}}(n).
$$
\end{corollary}

Since the $\mathbf{H}_k(n)$'s and the $\mathbf{I}_k(n)$'s, $k =1, 2, \ldots, n$, are elements of the Schur basis associated to pairs
of (row/column) conjugate partitions that satisfy the conditions of Corollary \ref{finite duality} , then
$$
\mathcal{W}_n \Big(  \mathbf{I}_k(n)  \Big) = \mathbf{H}_k(n), \quad k = 1, 2, \ldots, n.
$$

\begin{corollary}

The algebra automorphism $\mathcal{W}_n$ is an involution.
\end{corollary}

\begin{remark}
The preceding  results admit a representation-theoretic interpretation.

Given a partition $\mu$, there are two
 non equivalent functors (on the category of commutative ring with unity) that associate to $\mu$ a $GL(n)-$indecomposable representation, the {\it{Schur functor}} and
 the {\it{Weyl functor}} (see, e.g. \cite{ABW-BR}, where the Weyl functor is called co-Schur functor). When evaluated on
 a field of characteristic zero, the two functors produce irreducible modules $Schur_\mu(n)$ and $Weyl_\mu(n)$ of highest
 weight $\widetilde{\mu}$ and $\mu$, respectively. Hence  $Schur_\mu(n)$ is isomorphic to $Weyl_{\widetilde{\mu}}(n)$, and then,
 passing to characters, the ``duality'' between Schur and Weyl modules can be regarded as a representation-theoretic version of
the classical involution of the algebra $\Lambda(n)$ of symmetric polynomials. In this language, Theorem \ref{finite duality} can
be restated by saying that the eigenvalue of an element $\boldsymbol{\varrho} \in \boldsymbol{\zeta}(n)$ on the Schur module
$Schur_\mu(n)$ equals the eigenvalue its image $\mathcal{W}_n\big(\boldsymbol{\varrho})$ on the Weyl module
$Weyl_\mu(n)$.
\end{remark}

\section{The limit $n \rightarrow  \infty$ for $\boldsymbol{\zeta}(n)$: the algebra $\boldsymbol{\zeta}$}\label{zeta limit}

\subsection{The Capelli monomorphisms $\mathbf{i}_{n+1,n}$}\label{Capelli monomorphisms}

Given $n \in \mathbb{Z}^+$, let  we recall that $\boldsymbol{\zeta}(n)$ denote the center of the enveloping algebra $\mathbf{U}(gl(n))$, and
$$
\mathbf{H}_k(n), \quad k =1, \ldots , n
$$
denote the  Capelli free generators  $\boldsymbol{\zeta}(n)$, for every $n \in \mathbb{Z}^+$.

\

For every $n \in \mathbb{Z}^+$, let
$$
\mathbf{i}_{n+1,n} : \boldsymbol{\zeta}(n) \hookrightarrow \zeta(n+1)
$$
be the algebra monomorphism:
$$
\mathbf{i}_{n+1,n} : \mathbf{H}_k(n) \rightarrow \mathbf{H}_k(n+1), \quad k = 1,2, \ldots, n.
$$

We will refer to the monomorphism $\mathbf{i}_{n+1,n}$ as the {\it{Capelli monomorphisms}}.

\begin{remark}

Given $m \in \mathbb{Z}^+$, let $\boldsymbol{\zeta}(n)^{(m)}$ denote the $m-$th filtration element
of $\boldsymbol{\zeta}(n)$ (with respect to the filtration induced  by the standard filtration of $\mathbf{U}(n)$).

Clearly, the Capelli monomorphisms are {\it{morphisms in  the category of filtered algebras}}, that is
$$
\mathbf{i}_{n+1,n}\Big[ \boldsymbol{\zeta}(n)^{(m)} \Big] \subseteq \boldsymbol{\zeta}(n+1)^{(m)}
$$
\end{remark}

\begin{definition}
We  consider the {\it{direct limit}} (in the category of filtered algebras):
\begin{equation}\label{direct limit}
\underrightarrow{lim} \ \boldsymbol{\zeta}(n) = \boldsymbol{\zeta}.
\end{equation}
\end{definition}
\

The algebra $\boldsymbol{\zeta}$ inherits a structure of filtered algebra, where
$$
{\boldsymbol{\zeta}}^{(m)} = \underrightarrow{lim} \ \boldsymbol{\zeta}(n)^{(m)}.
$$

\

On the other hand, for every $n \in \mathbb{Z}^+$, we may consider the projection operator
$$
\boldsymbol{\pi}_{n,n+1} : \boldsymbol{\zeta}(n+1) \twoheadrightarrow \boldsymbol{\zeta}(n),
$$
such that
$$
\boldsymbol{\pi}_{n,n+1}(\mathbf{H}_k(n+1)) = \mathbf{H}_k(n)\, \quad k = 1, 2, \ldots, n,
$$
$$
\boldsymbol{\pi}_{n,n+1}(\mathbf{H}_{n+1}(n+1)) = 0.
$$

The following Remarks and Proposition are fairly obvious from the definitions.

\begin{remark}

\begin{enumerate}

\

\item

$$
Ker \big( \boldsymbol{\pi}_{n,n+1} \big) = \Big(  \mathbf{H}_{n+1}(n+1) \Big),
$$
the bilateral ideal of $\boldsymbol{\zeta}(n+1)$ generated by the element $\mathbf{H}_{n+1}(n+1)$.

\item

 The  projection $\boldsymbol{\pi}_{n,n+1}$ is the (filtered) left inverse  of the Capelli monomorphism $\mathbf{i}_{n+1,n}.$

In symbols,
$$
 \boldsymbol{\pi}_{n,n+1} \circ \mathbf{i}_{n+1,n} = Id_{\boldsymbol{\zeta}(n)}.
$$
\end{enumerate}

\end{remark}

\begin{proposition}\label{inverso filtrato}

If $n \geq m  $ , then the restriction  $\boldsymbol{\pi}^{(m)}_{n,n+1}$ of $\boldsymbol{\pi}_{n,n+1}$
to  $\boldsymbol{\zeta}(n+1)^{(m)}$
and the restriction  $\mathbf{i}^{(m)}_{n+1,n}$ of  $\mathbf{i}_{n+1,n}$ to  $\boldsymbol{\zeta}(n)^{(m)}$
 are  the  inverse  of  each  other.

\end{proposition}

\

\

\begin{claim}

The crucial point is that the projections $\boldsymbol{\pi}_{n,n+1}$ admit an {\it{intrinsic/invariant}} presentation
that is founded on the {\it{Olshanski decomposition}}.

\end{claim}
\

\subsection{The Olshanski decomposition/projection }\label{The Olshanski decomposition/projection}

We recall a special case of an essential costruction due to Olshanski \cite{Olsh1-BR}, \cite{Olsh3-BR}.
For the sake of simplicity, we follow Molev (\cite{Molev1-BR}, pp. 928 ff.).

\

Let $\mathbf{U}(gl(n+1))^0$ be the centralizer in $\mathbf{U}(gl(n+1))$ of the element $e_{n+1,n+1}$ of the standard basis
of $gl(n+1)$, regarded as an element of $\mathbf{U}(gl(n+1))$.

\

Let $\mathbf{I}(n+1)$ be the {\it{left ideal}} of $\mathbf{U}(gl(n+1))$ generated by the elements
$$
e_{i,n+1}, \quad i = 1, 2, \ldots, n+1.
$$

\

Let $\mathbf{I}(n+1)^0$  be the intersection
\begin{equation}\label{first Olshanski decomposition}
\mathbf{I}(n+1)^0 = \mathbf{I}(n+1) \cap \mathbf{U}(gl(n+1))^0.
\end{equation}

We recall that $\mathbf{I}(n+1)^0$ is a {\it{bilateral ideal}} of $\mathbf{U}(gl(n+1)^0$, and  the following {\it{direct sum decomposition}} hold

\begin{equation}\label{direct sum decomposition}
\mathbf{U}((gl(n+1))^0 = \mathbf{U}(gl(n)) \oplus \mathbf{I}(n+1)^0.
\end{equation}

\

Therefore, the {\it{Olshanski map}}
$$
\mathcal{M}_{n+1} : \mathbf{U}((gl(n+1))^0 \twoheadrightarrow \mathbf{U}(gl(n))
$$
that maps any element in the direct summand  $\mathbf{U}(gl(n))$ to itself and any element
in the direct summand  $\mathbf{I}(n+1)^0$ to zero is a well-defined algebra epimorphism.

\

Since $\boldsymbol{\zeta}(n+1)$ is a subalgebra of $\mathbf{U}(n+1)^0$, the direct sum decomposition (\ref{direct sum decomposition})
induces a  direct sum decomposition
of any element in $\boldsymbol{\zeta}(n+1)$ and the $\mathcal{M}_{n+1}$ map defines, by restriction, an algebra epimorphism
$$
\boldsymbol{\mu}_{n,n+1} : \boldsymbol{\zeta}(n+1) \twoheadrightarrow \boldsymbol{\zeta}(n).
$$

\

 In plain words, any element $\boldsymbol{\varrho} \in \boldsymbol{\zeta}(n+1)$ admits a {\it{unique}} decomposition

 \begin{equation}\label{Olshanski decomposition}
 \boldsymbol{\varrho} = \boldsymbol{\varrho}' \dotplus \boldsymbol{\varrho}^0, \quad \boldsymbol{\varrho}' \in \boldsymbol{\zeta}(n), \
 \boldsymbol{\varrho}^0 \in \mathbf{I}(n+1)^0.
 \end{equation}

We call the decomposition (\ref{Olshanski decomposition}) the {\it{Olshanski decomposition}} of the
element $\boldsymbol{\varrho} \in \boldsymbol{\zeta}(n)$.

In this notation, the projection
$$
\boldsymbol{\mu}_{n,n+1} : \boldsymbol{\zeta}(n+1) \twoheadrightarrow \boldsymbol{\zeta}(n),
$$
$$
\boldsymbol{\mu}_{n+1,n}(\boldsymbol{\varrho}) = \boldsymbol{\varrho}', \quad \boldsymbol{\varrho} \in \boldsymbol{\zeta}(n+1)
$$
is defined.
\

We claim that

\begin{proposition}\label{Olshanski Capelli}
\

$$
\mathbf{H}_k(n+1) = \mathbf{H}_k(n) \dotplus \mathbf{H}_k(n+1)^0,
$$
where
$$
\mathbf{H}_k(n+1)^0 = \mathbf{H}_k(n+1) - \mathbf{H}_k(n) \in \mathbf{I}(n+1)^0,
$$
and
$$
\mathbf{H}_k(n) \in \boldsymbol{\zeta}(n).
$$
\end{proposition}

\begin{example}We have:
\begin{align*}
\mathbf{H}_2(4) &= [21|12]+[31|13]+[41|14]+[32|23]+[42|24]+[43|34]
\\
&= \mathbf{H}_2(3) \dotplus \mathbf{H}_2(4)^0,
\end{align*}
where
$$
\mathbf{H}_2(3) = [21|12]+[31|13]+[32|23] \in \boldsymbol{\zeta}(3),
$$
and
$$
\mathbf{H}_2(4)^0 = [41|14]+[42|24]+[43|34] \in \mathbf{I}(4)^0.
$$
\end{example}\qed

\begin{proposition}\label{pi=mu}

The map $\boldsymbol{\mu}_{n,n+1}$ is the {\it{same}} as the map $\boldsymbol{\pi}_{n,n+1}$.

\end{proposition}

Therefore, in the following, we refer to the projections
$$\boldsymbol{\mu}_{n,n+1} = \boldsymbol{\pi}_{n,n+1}$$
as the {\it{Capelli-Olshanski projections}}.

\

\begin{remark}\label{projective limit}

From Proposition \ref{inverso filtrato}, the algebra $\boldsymbol{\zeta}$ (direct limit) is the same as the
{\it {projective limit in the category of filtered algebras}}
$$
\boldsymbol{\zeta} = \underleftarrow{lim} \ \boldsymbol{\zeta}(n)
$$
with respect to the system of Capelli-Olshanski projections (Molev, \cite{Molev1-BR}).

\end{remark}

\

\subsection{Main results}\label{main}

From their virtual presentation, we directly infer the  {\it{Olshanski decompositions}} and  {\it{Capelli-Olshanski projections}} :

\begin{proposition}\label{Olshanski altri}

\

\begin{enumerate}

\item
$
\mathbf{I}_k(n+1) = \mathbf{I}_k(n) \dotplus \mathbf{I}_k(n+1)^0,
$
where
$$
\mathbf{I}_k(n+1)^0 = \mathbf{I}_k(n+1) - \mathbf{I}_k(n) \in \mathbf{I}_k(n+1)^0,
$$
and
$$
\mathbf{I}_k(n) \in \boldsymbol{\zeta}(n).
$$
Then
\begin{equation}
\boldsymbol{\pi}_{n,n+1}(\mathbf{I}_k(n+1)) = \mathbf{I}_k(n).
\end{equation}

\item
$
\mathbf{K}_{\lambda}(n+1) = \mathbf{K}_{\lambda}(n) \dotplus \mathbf{K}_{\lambda}(n+1)^0,
$
where
$$
\mathbf{K}_{\lambda}(n+1)^0 = \mathbf{K}_{\lambda}(n+1) - \mathbf{K}_{\lambda}(n) \mathbf{I}_k(n+1)^0,
$$
and
$$
\mathbf{K}_{\lambda}(n) \in \boldsymbol{\zeta}(n).
$$
Then
\begin{equation}
\boldsymbol{\pi}_{n,n+1}(\mathbf{K}_{\lambda}(n+1)) = \mathbf{K}_{\lambda}(n).
\end{equation}

\item
$
\mathbf{J}_{\lambda}(n+1) = \mathbf{J}_{\lambda}(n) \dotplus \mathbf{J}_{\lambda}(n+1)^0,
$
where
$$
\mathbf{J}_{\lambda}(n+1)^0 = \mathbf{J}_{\lambda}(n+1) - \mathbf{J}_{\lambda}(n) \in \mathbf{I}_k(n+1)^0,
$$
and
$$
\mathbf{J}_{\lambda}(n) \in \boldsymbol{\zeta}(n).
$$
Then
\begin{equation}
\boldsymbol{\pi}_{n,n+1}(\mathbf{J}_{\lambda}(n+1)) = \mathbf{J}_{\lambda}(n).
\end{equation}

\item
$
\mathbf{S}_{\lambda}(n+1) = \mathbf{S}_{\lambda}(n) \dotplus \mathbf{S}_{\lambda}(n+1)^0,
$
where
$$
\mathbf{S}_{\lambda}(n+1)^0 = \mathbf{S}_{\lambda}(n+1) - \mathbf{S}_{\lambda}(n) \in \mathbf{I}_k(n+1)^0,
$$
and
$$
\mathbf{S}_{\lambda}(n) \in \boldsymbol{\zeta}(n).
$$
Then
\begin{equation}
\boldsymbol{\pi}_{n,n+1}(\mathbf{S}_{\lambda}(n+1)) = \mathbf{S}_{\lambda}(n).
\end{equation}

\end{enumerate}

\end{proposition}

By combining the preceding Proposition with Proposition \ref{inverso filtrato}, we get

\begin{theorem}\label{germs}We have:
\begin{itemize}
\renewcommand{\labelitemi}{\normalfont -}
\item
$\mathbf{I}_k(n) \in \boldsymbol{\zeta}(n)^{(k)}$, then
$$
\mathbf{i}_{n+1, n}(\mathbf{I}_k(n)) = \mathbf{I}_k(n+1), \quad n \geq k;
$$
\item
$\mathbf{K}_{\lambda}(n) \in \boldsymbol{\zeta}(n)^{ ( | \lambda | ) }$, then
$$
\mathbf{i}_{n+1, n}(\mathbf{K}_{\lambda}(n)) = \mathbf{K}_{\lambda}(n+1), \quad n \geq |\lambda|;
$$
\item
$\mathbf{J}_{\lambda}(n) \in \boldsymbol{\zeta}(n)^{ ( | \lambda | ) }$, then
$$
\mathbf{i}_{n+1, n}(\mathbf{J}_{\lambda}(n)) = \mathbf{J}_{\lambda}(n+1), \quad n \geq |\lambda|;
$$
\item
$\mathbf{S}_{\lambda}(n) \in \boldsymbol{\zeta}(n)^{ ( | \lambda | ) }$, then
$$
\mathbf{i}_{n+1, n}(\mathbf{S}_{\lambda}(n)) = \mathbf{S}_{\lambda}(n+1), \quad n \geq |\lambda|.
$$
\end{itemize}
\end{theorem}

Passing to the direct limit $\underrightarrow{lim} \ \boldsymbol{\zeta}(n) = \boldsymbol{\zeta},$ we  set:

\begin{definition}\label{series} $ $

\begin{itemize}
\renewcommand{\labelitemi}{\normalfont }
\item
$
\mathbf{H}_k  = \underrightarrow{lim}  \ \mathbf{H}_k(n) \in \boldsymbol{\zeta}, \quad n \geq k.
$

\item
$
\mathbf{I}_k = \underrightarrow{lim} \  \mathbf{I}_k(n)  \in \boldsymbol{\zeta}, \quad n \geq k.
$

\item
$
\mathbf{K}_{\lambda} = \underrightarrow{lim}  \  \mathbf{K}_{\lambda}(n)  \in \boldsymbol{\zeta}, \quad n \geq |\lambda|.
$

\item
$
\mathbf{J}_{\lambda} = \underrightarrow{lim} \  \mathbf{J}_{\lambda}(n)  \in \boldsymbol{\zeta}, \quad n \geq |\lambda|.
$

\item
$
\mathbf{S}_{\lambda} = \underrightarrow{lim} \ \mathbf{S}_{\lambda}(n)  \in \boldsymbol{\zeta}, \quad n \geq |\lambda|.
$

\end{itemize}
\end{definition}

From the definition of Capelli monomorphisms and Theorem \ref{germs}, it follows

\begin{proposition}\label{formal series}
The elements
$\mathbf{H}_{k}, \mathbf{I}_{k}, \mathbf{K}_{\lambda}, \mathbf{J}_{\lambda}, \mathbf{S}_{\lambda} \in \boldsymbol{\zeta}$
can be consistently written as {\text{formal series}}. More precisely, setting $L^* = \mathbb{Z}^* = \{ 1, 2, \ldots \}$,
\begin{itemize}
\item \begin{align*}
\mathbf{H}_k & =
 \sum_{i_1 < \cdots < i_k } \ [ i_k \cdots i_2 i_1 | i_1 i_2 \cdots i_k ] = \\
 & =  \sum_{i_1 < \cdots < i_k } \ \mathfrak{p} \big( e_{i_k , \alpha} \cdots e_{i_2, \alpha} e_{i_1, \alpha}
e_{\alpha, i_1}e_{\alpha, i_2} \cdots e_{\alpha, i_k } \big)
\end{align*}

\item \begin{align*}
\mathbf{I}_k & =
 \sum_{j_1 < j_2 < \cdots < j_p } \ (i_{j_1}! \ i_{j_2}! \cdots i_{j_p}!)^{-1}
  \ [j_p^{i_{j_p}} \cdots j_2^{i_{j_2}} j_1^{i_{j_1}} | j_1^{i_{j_1}} j_2^{i_{j_2}}  \cdots j_p^{i_{j_p}} ]^{*}
 = \\
 & =  \sum_{j_1 < j_2 < \cdots < j_p } \ (i_{j_1}! \ i_{j_2}! \cdots i_{j_p}!)^{-1} \
 \mathfrak{p}  \big( e_{j_p , \beta}^{i_{j_p}} \cdots e_{j_2, \beta}^{i_{j_2}} e_{j_1, \beta}^{i_{j_1}}
e_{\beta, j_1}^{i_{j_1}}e_{\beta, j_2}^{i_{j_2}} \cdots e_{\beta, j_p}^{i_{j_p}} \big), \\
\end{align*}
where $\beta \in A_1$ denotes {\it{any}} negative virtual symbol,
the sum is extended to all $p-$tuples $j_1 < j_2 < \cdots < j_p$ in $L^*$ ($p\leq k$),
and to all the $p-$tuples of exponents $(i_{j_1}, i_{j_2}, \cdots, i_{j_p})$ such that
$i_{j_1}+i_{j_2}+ \cdots  + i_{j_p} = k$
and any
$$
[j_p^{i_{j_p}} \cdots j_2^{i_{j_2}} j_1^{i_{j_1}} | j_1^{i_{j_1}} j_2^{i_{j_2}}  \cdots j_p^{i_{j_p}} ]^{*}
$$
is a {\textit{permanental}} Capelli bitableau with one row.

\item $$
\mathbf{K}_\lambda =
\sum_S \ \mathfrak{p} \big( e_{S,C_{\lambda}^*} \cdot e_{C_{\lambda}^*,S} \big) = \sum_S \ SC_{\lambda}^* \ C_{\lambda}^*S,
$$
where the sum is extended to all  row-increasing tableaux $S$ on the  alphabet $L^*$.

\item $$
\mathbf{J}_\lambda =
\sum_S \ (o_S)^{-1} \ \mathfrak{p} \big( e_{S,D^*_{\widetilde{\lambda}}} \cdot e_{D^*_{\widetilde{\lambda}},S} \big) =
\sum_S \  (o_S)^{-1} \ SD^*_{\widetilde{\lambda}} \ D^*_{\widetilde{\lambda}}S,
$$
where the sum is extended to all  column-nondecreasing tableaux $S$ (of shape $\widetilde{\lambda}$) on
the  alphabet $L^*$.

\item
\begin{align*}
\mathbf{S}_\lambda &=
\frac {1} {H(\lambda)} \  \ \sum_S \ \mathfrak{p} \big(e_{S,C_{\lambda}^{*}} \cdot e_{C_{\lambda}^{*},D_{\lambda}^{*}} \cdot  e_{D_{\lambda}^{*},C_{\lambda}^{*}}
\cdot e_{C_{\lambda}^{*},S} \big) \\
&= \frac {1} {H(\lambda)} \ \  \sum_S \ SC_{\lambda}^{*} \ C_{\lambda}^{*}D_{\lambda}^{*} \
D_{\lambda}^{*}C_{\lambda}^{*} \ C_{\lambda}^{*}S, \\
\end{align*}
 where the sum is extended to all  row-increasing tableaux $S$ on the alphabet $L^*$.
\end{itemize}

\end{proposition}

\newpage

From Proposition \ref{projective limit}, it follows

\begin{corollary} We have:

\begin{itemize}
\renewcommand{\labelitemi}{\normalfont }
\item
$
\underleftarrow{lim}  \ \mathbf{H}_k(n) = \mathbf{H}_k \in \boldsymbol{\zeta},
$

\item
$
\underleftarrow{lim} \  \mathbf{I}_k(n) = \mathbf{I}_k \in \boldsymbol{\zeta},
$

\item
$
\underleftarrow{lim}  \  \mathbf{K}_{\lambda}(n) = \mathbf{K}_{\lambda} \in \boldsymbol{\zeta},
$

\item
$
 \underleftarrow{lim} \  \mathbf{J}_{\lambda}(n) = \mathbf{J}_{\lambda} \in \boldsymbol{\zeta},
$

\item
$
 \underleftarrow{lim} \ \mathbf{S}_{\lambda}(n) = \mathbf{S}_{\lambda} \in \boldsymbol{\zeta}.
$

\end{itemize}

\end{corollary}

Due the fact that the algebra $\boldsymbol{\zeta}$ is defined as a direct limit, we infer:

\begin{theorem} $ $

\begin{enumerate}

\item
The set
$$
\Big\{  \mathbf{H}_k; \  k \in \mathbb{Z}^+ \Big\}
$$
is a system of free algebraic generators of $\boldsymbol{\zeta}$.

\item
The set
$$
\Big\{  \mathbf{I}_k; \  k \in \mathbb{Z}^+ \Big\}
$$
is a system of free algebraic generators of $\boldsymbol{\zeta}$.

\item
The set
$$
\Big\{ \mathbf{K}_{\lambda}; \  \lambda \ any \ partition \ \Big\}
$$
is a linear basis  of $\boldsymbol{\zeta}$.

\item
The set
$$
\Big\{ \mathbf{J}_{\lambda}; \  \lambda \ any \ partition \ \Big\}
$$
is a linear basis  of $\boldsymbol{\zeta}$.

\item
The set
$$
\Big\{ \mathbf{S}_{\lambda}; \  \lambda \ any \ partition \ \Big\}
$$
is a linear basis  of $\boldsymbol{\zeta}$.

\end{enumerate}

\end{theorem}

\section{The algebra $\Lambda^*(n)$ of shifted symmetric polynomials and the Harish-Chandra Isomorphism}\label{Lambda(n)}

\subsection{The Harish-Chandra isomorphism $\chi_n : \boldsymbol{\zeta}(n) \longrightarrow \Lambda^*(n)$}

In this subsection we follow  Okounkov and Olshanski \cite{OkOlsh-BR}.

As in the classical context of the algebra $\Lambda(n)$ of symmetric
polynomials in $n$ variables $x_1, x_2, \ldots, x_n$, the algebra
$\Lambda^*(n)$ of {\textit{shifted symmetric polynomials}} is an algebra
of polynomials $p(x_1, x_2, \ldots, x_n)$  but the ordinary symmetry
is replaced by the {\textit{shifted symmetry}}:
$$
 f(x_1, \ldots , x_i, x_{i+1}, \ldots, x_n) = f(x_1, \ldots , x_{i+1} - 1, x_i + 1,
 \ldots, x_n),
$$
for $i = 1, 2, \ldots, n - 1.$

\vskip 0.3cm

\begin{examples}\label{Ex shifted-BR}

Two basic classes of shifted symmetric polynomials are provided by the sequences of {\textit{shifted elementary symmetric polynomials}} and
{\textit{shifted complete symmetric polynomials}}.

\begin{itemize}
\renewcommand{\labelitemi}{\normalfont }
\item {\textbf{Elementary shifted symmetric polynomials}}

For every $k \in \mathbb{N}$ let
\begin{equation}\label{shifted elementary-BR}
\mathbf{e}_k^{*}(x_1, x_2, \ldots, x_n)
= \sum_{1 \leq i_1 < i_2 < \cdots < i_r \leq n} \ (x_{i_1}  + k  - 1)
(x_{i_2}  + k - 2) \cdots (x_{i_k}),
\end{equation}
and $\mathbf{e}_0^{*}(x_1, x_2, \ldots, x_n) = \mathbf{1}.$

\item {\textbf{Complete shifted symmetric polynomials}}

For every $r \in \mathbb{N}$ let
\begin{equation}\label{shifted complete-BR}
\mathbf{h}_k^{*}(x_1, x_2, \ldots, x_n)
= \sum_{1 \leq i_1 \leq i_2 < \cdots \leq i_k \leq n} \ (x_{i_1}  - k  + 1)
(x_{i_2}  - k + 2) \cdots (x_{i_k}),
\end{equation}
and $\mathbf{h}_0^{*}(x_1, x_2, \ldots, x_n) = \mathbf{1}.$

\end{itemize}

\end{examples}

\vskip 0.3cm

\begin{definition}\label{HC isomorphism}

The {\textit{Harish-Chandra isomorphism}} $\chi_n$ is the algebra isomorphism
$$
\chi_n : \boldsymbol{\zeta}(n) \longrightarrow \Lambda^*(n), \qquad  \ A \mapsto \chi_n(A),
$$
$\chi_n(A)$ being the shifted symmetric polynomial such that, for every highest weight module $V_{\mu}$,
the evaluation $\chi_n(A)(\mu_1, \mu_2, \ldots , \mu_n)$ equals the eigenvalue of
$A \in \boldsymbol{\zeta}(n)$ in  $V_{\mu}$ (\cite{OkOlsh-BR}, Proposition $\mathbf{2.1}$).

\end{definition}

\subsection{The Harish-Chandra images of the Capelli free generators $\mathbf{H}_k(n)$}

From  Corollary \ref{Capelli eigenvalues}.$1$, it follows

\begin{proposition}\label{elementary}
$$
\chi_n(\mathbf{H}_k(n)) =  \mathbf{e}_k^{*}(x_1, x_2, \ldots, x_n) \in \Lambda^*(n),
$$
for every $k =  1, 2, \ldots, n.$
\end{proposition}

\subsection{The Harish-Chandra images of the Nazarov/Umeda elements $\mathbf{I}_k(n)$}

From Theorem \ref{vertical strip}.$2$, it follows
\begin{proposition}
For every $k \in \mathbb{Z}^+$,

$$  \chi_n( \mathbf{I}_k(n) ) = \mathbf{h}_k^{*}(x_1, x_2, \ldots, x_n) \in \Lambda^*(n).$$

\end{proposition}

\subsection{The shifted Schur polynomials $\mathbf{s}^*_\lambda(x_1, \ldots, x_n)$ as images of the Schur elements $\mathbf{S}_\lambda(n)$}

Recall that, given a variable $z$ and a natural integer $p$, the symbol $(z)_p$ denotes
(see, e.g. \cite{Aigner-BR}) the {\it{falling factorial polynomials}}:
$$
(z)_p = z(z - 1) \cdots (z - p + 1), \quad p \geq 1,
\quad \quad
(z)_0 = 1.
$$

Let $\lambda$ be a partition, $\lambda_1 \leq n$.

Following \cite{OkOlsh-BR}, set
$$
\mathbf{s}^*_\lambda(x_1, \ldots, x_n) =
\frac { det \Big[ (x_i + n -i)_{\widetilde{\lambda}_i + n - j } \Big]} { det \Big[ (x_i + n -i)_{n - j} \Big]},
$$
for every $1 \leq i, j \leq n.$

The polynomials $\mathbf{s}^*_\lambda(x_1, \ldots, x_n)$ are shifted symmetric polynomials (in symbols,
$\mathbf{s}^*_\lambda(x_1, \ldots, x_n) \in \Lambda^*(n)$) and are called the
{\it{shifted Schur polynomials}} in $n$ variables.

From the Characterization Theorems for the Schur elements $\mathbf{S}_\lambda(n) \in \boldsymbol{\zeta}(n)$
(see subsection \ref{Characterization  Theorems})
and the the Characterization Theorems for the shifted Schur polynomials \cite{OkOlsh-BR}, we have:

\begin{theorem}
For every partition $\lambda$, $\lambda_1 \leq n$,
$$  \chi_n( \mathbf{S}_\lambda(n) ) = \mathbf{s}^*_\lambda(x_1, \ldots, x_n).$$
\end{theorem}

\subsection{The fundamental theorems for the algebra $\Lambda^*(n)$}

 From Theorem \ref{Capelli generators} and Proposition \ref{elementary}, it follows

\begin{proposition}

The set
$$
\Big\{  \mathbf{e}_k^{*}(x_1, x_2, \ldots, x_n); \ k = 1, 2, \ldots, n   \Big\}
$$
is a set of free algebra generators of the polynomial algebra $\Lambda^*(n)$.
\end{proposition}

Since, for every $k \in \mathbb{Z}^+$,  the {\it{indicator}} (top degree homogeneous part) of $\mathbf{h}_k^{*}(x_1, x_2, \ldots, x_n)$
is the classical complete homogeneous symmetric polynomial $\mathbf{h}_k(x_1, x_2, \ldots, x_n) \in \Lambda(n)$,
from the preceding discussions it also follows

\begin{proposition}

The set
$$
\Big\{  \mathbf{h}_k^{*}(x_1, x_2, \ldots, x_n); \ k = 1, 2, \ldots, n   \Big\}
$$
is a set of free algebra generators of the polynomial algebra $\Lambda^*(n)$.
\end{proposition}

\begin{proposition}

The set
$$
\Big\{  \mathbf{s}^*_\lambda(x_1, \ldots, x_n); \ \lambda_1 \leq n   \Big\}
$$
is a linear basis of the polynomial algebra $\Lambda^*(n)$.
\end{proposition}

\section{The algebra $\Lambda^*$  of shifted symmetric functions and the
Harish-Chandra isomorphism $\chi : \boldsymbol{\zeta} \rightarrow \Lambda^*$}\label{Lambda}

\subsection{The monomorphism $\mathbf{i}^*_{n+1,n}$  and the projection $\boldsymbol{\pi}^*_{n,n+1}$}

Let

$$
\mathbf{i}^*_{n+1,n} : \Lambda^*(n) \hookrightarrow \Lambda^*(n+1)
$$
be the algebra monomorphism such that
$$
\mathbf{i}^*_{n+1,n}\big( \mathbf{e}_k^{*}(x_1, x_2, \ldots, x_n) \big) = \mathbf{e}_k^{*}(x_1, x_2, \ldots, x_n, x_{n+1}),
$$
for $k = 1, 2, \ldots, n.$

\begin{remark}

Given $m \in \mathbb{Z}^*$, let ${\Lambda^*(n)}^{(m)}$ denote the $m-$th filtration element
of $\Lambda^*(n)$ (with respect to the filtration induced  by the standard filtration of the algebra of polynomials
in the variables $x_1, x_2, \ldots, x_n$).

Clearly, the  monomorphisms $\mathbf{i}^*_{n+1,n}$ are {\it{morphisms in  the category of filtered algebras}}, that is
$$
\mathbf{i}^*_{n+1,n}\Big[ \Lambda^*(n)^{(m)} \Big] \subseteq \Lambda^*(n+1)^{(m)}.
$$
\end{remark}

\begin{definition}
We  consider the {\it{direct limit}} (in the category of filtered algebras):
\begin{equation}\label{direct limit}
\underrightarrow{lim} \ \Lambda^*(n) = \Lambda^*.
\end{equation}
\end{definition}
\

The algebra $\Lambda^*$ inherits a structure of filtered algebra, where
$$
{\Lambda^*}^{(m)} = \underrightarrow{lim} \ \Lambda^*(n)^{(m)}.
$$

Let

$$
\boldsymbol{\pi}^*_{n,n+1} : \Lambda^*(n+1) \twoheadrightarrow \Lambda^*(n)
$$
be the algebra epimorphism such that
$$
\boldsymbol{\pi}^*_{n,n+1}\big( \mathbf{f}^{*}(x_1, x_2, \ldots, x_n, x_{n+1}) \big) = \mathbf{f}^{*}(x_1, x_2, \ldots, x_n, 0),
$$
for every $\mathbf{f}^{*}(x_1, x_2, \ldots, x_n, x_{n+1}) \in \Lambda^*(n+1).$
Clearly,
$$
\boldsymbol{\pi}^*_{n,n+1}\big( \mathbf{e}_k^{*}(x_1, x_2, \ldots, x_n, x_{n+1}) \big) = \mathbf{e}_k^{*}(x_1, x_2, \ldots, x_n),
$$
for $k = 1, 2, \ldots, n,$ and
$$
\boldsymbol{\pi}^*_{n,n+1}\big( \mathbf{e}_{n+1}^{*}(x_1, x_2, \ldots, x_n, x_{n+1}) \big) = 0.
$$

As for the centers $\boldsymbol{\zeta}(n+1)$ and $\boldsymbol{\zeta}(n)$, the following Remarks and Proposition on $\Lambda^*(n+1)$ and
$\Lambda^*(n)$ are obvious from the definitions.

\begin{remark}
We have
\begin{enumerate}

\item
$Ker \big( \boldsymbol{\pi}^*_{n,n+1} \big)$
is the bilateral ideal
$$
\Big(  \mathbf{e}_{n+1}^{*}(x_1, x_2, \ldots, x_n, x_{n+1}) \Big)
$$
of $\Lambda^*(n+1)$
generated by the element $\mathbf{e}_{n+1}^{*}(x_1, x_2, \ldots, x_n, x_{n+1})$.

\item
The  projection $\boldsymbol{\pi}^*_{n,n+1}$ is the left inverse  of the  monomorphism $\mathbf{i}^*_{n+1,n}.$
In symbols,
$$
 \boldsymbol{\pi}^*_{n,n+1} \circ \mathbf{i}^*_{n+1,n} = Id_{\Lambda^*(n)}.
$$
\end{enumerate}

\end{remark}

\begin{proposition}\label{inverso filtrato funzioni}
If $n \geq m  $ , then the restriction  ${\boldsymbol{\pi}^*_{n,n+1}}^{(m)}$ of $\boldsymbol{\pi}^*_{n,n+1}$
to  ${\Lambda^*(n+1)}^{(m)}$
and the restriction  $\mathbf{i}^*_{n+1,n}$ of  $\mathbf{i}^*_{n+1,n}$ to  ${\Lambda^*(n)}^{(m)}$
 are  the  inverse  of  each  other.
\end{proposition}

\begin{remark}
From Proposition \ref{inverso filtrato funzioni}, the algebra $\Lambda^*$ (direct limit) is the same as the
projective limit in the category of filtered algebras
$$
\Lambda^* = \underleftarrow{lim} \ \Lambda^*(n)
$$
with respect to the system of the projections $\boldsymbol{\pi}^*_{n,n+1}$, and therefore, the algebra $\Lambda^*$
is the algebra of shifted symmetric functions of \cite{OkOlsh-BR}.
\end{remark}

\subsection{The Harish-Chandra isomorphism $\chi : \boldsymbol{\zeta} \rightarrow \Lambda^*$}

Consider the following commutative diagram:

\begin{equation}\label{diagramma commutativo}
\begin{tikzpicture}[description/.style={fill=white,inner sep=2pt}]
\fontsize{15.0}{20.0}
\matrix (m) [matrix of math nodes, row sep=5em,
column sep=9em, text height=3.0ex, text depth=1.0ex]
{
\textbf{$\boldsymbol{\zeta}^{(m)}(n)$}
&
\textbf{$\boldsymbol{\zeta}^{(m)}(n+1)$}
\\
\textbf{$\Lambda ^{*(m)}(n)$}
&
\textbf{$\Lambda ^{*(m)}(n+1)$}
\\
};

\fontsize{12}{14}

\draw [line width=0.02cm,<<-] (-2.6,-1.7) -- (1.7,-1.7);
\node  [above] at (0,-1.7) {${\boldsymbol{\pi}^*_{n,n+1}}$};
\draw [line width=0.02cm,right hook->] (-2.6,-2.0) -- (1.7,-2.0);
\node  [below] at (0,-2.0) {${\boldsymbol{i}^*_{n+1,n}}$};

\draw [line width=0.02cm, <<-] (-2.6,1.9) -- (1.7,1.9);
\node  [above] at (0,1.9) {${\boldsymbol{\pi}_{n,n+1}}$};
\draw [line width=0.02cm,right hook->] (-2.6,1.6) -- (1.7,1.6);
\node  [below] at (0,1.6) {${\boldsymbol{i}_{n+1,n}}$};

\draw [line width=0.02cm, ->] (-3.7,1.4) -- (-3.7,-1.4);
\node   at (-3.3,0.0) {{$\chi_n$}};
\draw [line width=0.02cm, ->] (3.2,1.4) -- (3.2,-1.4);
\node   at (3.8,0.0) {{$\chi_{n+1}$}};

\end{tikzpicture}
\end{equation}

\begin{theorem}\label{isomorfismo HC infinito}

\

If $n \geq m$, the pairs of horizontal arrows in the commutative diagram (\ref{diagramma commutativo}) denote
mutually inverse isomorphisms.

\end{theorem}

\
Passing to the direct limit, we get the isomorphism of filtered algebras:
$$
    \chi : \boldsymbol{\zeta} \approx  \Lambda^*.
$$

\

In particular, we infer the images in $\Lambda^*$ of the free systems of algebraic generators of $\boldsymbol{\zeta}$:

$$
\Big\{  \mathbf{H}_k; \  k \in \mathbb{Z}^+ \Big\}, \qquad  \Big\{  \mathbf{I}_k; \  k \in \mathbb{Z}^+ \Big\}.
$$
with respect to the  isomorphism
$\chi.$

\begin{proposition}
We have

\begin{enumerate}

\item
For every $k \in \mathbb{Z}^+$,
$$
\chi \big( \mathbf{H}_k \big) = \mathbf{e}^*_k \in \Lambda^*,
$$
where
$$
\mathbf{e}^*_k = \sum_{ i_1 < i_2 < \cdots < i_k } \ (x_{i_1}  + k  - 1)
(x_{i_2}  + k - 2) \cdots (x_{i_k}), \quad i_s \in \mathbb{Z}^+,
$$
$\mathbf{e}^*_k$ the $k-$th elementary shifted symmetric function;

\item
For every $k \in \mathbb{Z}^+$,
$$
\chi \big( \mathbf{I}_k \big) = \mathbf{h}^*_k \in \Lambda^*,
$$
where
$$
\mathbf{h}_k^{*}
= \sum_{i_1 \leq i_2 < \cdots \leq i_k } \ (x_{i_1}  - k  + 1)
(x_{i_2}  - k + 2) \cdots (x_{i_k}), \quad i_s \in \mathbb{Z}^+,
$$
$\mathbf{h}^*_k$ the $k-$th complete shifted symmetric function.

\end{enumerate}

\end{proposition}

\subsection{On the isomorphism $\chi$}

Notice that, for every $\boldsymbol{\varrho}(n) \in \boldsymbol{\zeta}(n)^{(m)}$, it follows
$$
\chi_n(\boldsymbol{\varrho}(n))(x_1, x_2, \ldots, x_n) =
{\boldsymbol{\pi}^*_{n,n+1}}^{(m)} \Big( \chi_{n+1} \big( \mathbf{i}^*_{n+1,n} (\boldsymbol{\varrho}(n)) \big) \Big) (x_1, x_2, \ldots, x_n) =
$$
$$
 = \chi_{n+1} \big( {\mathbf{i}^*_{n+1,n}}^{(m)} (\boldsymbol{\varrho}(n)) \big)  (x_1, x_2, \ldots, x_n, 0).
$$

Then, for every partition $\mu$ and for every $\boldsymbol{\varrho}(n) \in \boldsymbol{\zeta}(n)^{(m)}$,
if
$$
n \geq  max\{ m,  \ l(\widetilde{\mu}) = \mu_1 \},
$$
then
$$
\chi_n(\boldsymbol{\varrho}(n))(\widetilde{\mu}) =  \chi_{n+1} \big( {\mathbf{i}^*_{n+1,n}}^{(m)} (\boldsymbol{\varrho}(n)) \big)  (\widetilde{\mu}).
$$

Therefore, the sequence
$$
\Big( \ \chi_{n+1} \big( {\mathbf{i}^*_{n+1,n}}^{(m)} (\boldsymbol{\varrho}(n)) \big)  (\widetilde{\mu}) \ \Big)_{n \in \mathbb{N}^+}
$$
is definitively constant {\it{ equal to}} $\chi_n(\boldsymbol{\varrho}(n))(\widetilde{\mu}),$
and, passing to the direct limit
$$
 \boldsymbol{\varrho} = \underrightarrow{lim} \ \boldsymbol{\varrho}(n) \in \boldsymbol{\zeta}^{*(m)},
$$

\begin{proposition}\label{The map universal HC}

The eigenvalue
\begin{equation}\label{universal HarishChandra}
\chi(\boldsymbol{\varrho})(\widetilde{\mu}) = \chi_n(\boldsymbol{\varrho}(n))(\widetilde{\mu}), \quad n \ sufficiently \ large
\end{equation}
is well-defined.
\end{proposition}

Equation (\ref{universal HarishChandra}) may be regarded as the explicit definition of the isomorphism
$$
    \chi : \boldsymbol{\zeta} \approx  \Lambda^*.
$$

\subsection{Duality in $\boldsymbol{\zeta}$ and $\Lambda^*$}

Let
$$
\mathcal{W} : \boldsymbol{\zeta} \rightarrow \boldsymbol{\zeta^*}
$$
denote the automorphism such that
$$
\mathcal{W} \Big( \mathbf{H}_k \Big) = \mathbf{I}_k, \quad   for \ every  \ k \in \mathbb{Z}^+,
$$

and let

$$
w : \Lambda^* \rightarrow \Lambda^*
$$
denote the automorphism such that
$$
w \Big( \mathbf{e}^*_k \Big) = \mathbf{h}^*_k, \quad   for \ every  \ k \in \mathbb{Z}^+.
$$

\

Clearly, we have:

\begin{proposition}

$$
 \chi \circ \mathcal{W} = w \circ \chi.
$$

\end{proposition}

\

From Theorem \ref{finite duality} and Proposition \ref{The map universal HC}, we infer:

\begin{theorem}
\
For every $\boldsymbol{\varrho} \in \boldsymbol{\zeta}^{(m)}$ and for every  partition $\mu$, we have:

\begin{equation}
\Big( ( \chi \circ  \mathcal{W} )  ( \boldsymbol{\varrho} )    \Big)(\mu) = \Big( \chi \circ \boldsymbol{\varrho} \Big)(\widetilde{\mu}).
\end{equation}
\end{theorem}

\

\begin{corollary}

\

\begin{enumerate}

\item
For every partition $\lambda$,
 $$
  \mathcal{W} \Big( \mathbf{S}_{\lambda}\Big) = \mathbf{S}_{\widetilde{\lambda}}.
 $$

\item
In particular,
$$
  \mathcal{W} \Big( \mathbf{I}_k \Big) = \mathbf{H}_k, \quad for \ every \ k \in\mathbb{Z}^+;
$$
then, the automorphisms $\mathcal{W}$ and $w$ are involutions.

\end{enumerate}
\end{corollary}

\end{document}